\pdfoutput=1
\documentclass[oneside,english]{amsart}
\usepackage[T1]{fontenc}
\usepackage[latin9]{inputenc}
\usepackage{geometry}
\usepackage{babel}
\usepackage{mathrsfs}
\usepackage{amsbsy}
\usepackage{amstext}
\usepackage{amsthm}
\usepackage{amssymb}
\usepackage{graphicx}
\usepackage[all]{xy}
\usepackage[unicode=true,pdfusetitle,
 bookmarks=true,bookmarksnumbered=false,bookmarksopen=false,
 breaklinks=false,pdfborder={0 0 1},backref=false,colorlinks=false]
 {hyperref}

\makeatletter

\providecommand{\tabularnewline}{\\}

\numberwithin{equation}{section}
\numberwithin{figure}{section}
  \theoremstyle{plain}
  \newtheorem*{thm*}{\protect\theoremname}
\theoremstyle{plain}
\newtheorem{thm}{\protect\theoremname}[section]
  \theoremstyle{definition}
  \newtheorem{defn}[thm]{\protect\definitionname}
  \theoremstyle{remark}
  \newtheorem*{rem*}{\protect\remarkname}
  \theoremstyle{plain}
  \newtheorem{lem}[thm]{\protect\lemmaname}
  \theoremstyle{definition}
  \newtheorem*{example*}{\protect\examplename}
  \theoremstyle{plain}
  \newtheorem{prop}[thm]{\protect\propositionname}
  \theoremstyle{plain}
  \newtheorem{cor}[thm]{\protect\corollaryname}
\newenvironment{lyxlist}[1]
{\begin{list}{}
{\settowidth{\labelwidth}{#1}
 \setlength{\leftmargin}{\labelwidth}
 \addtolength{\leftmargin}{\labelsep}
 }}
{\end{list}}
  \theoremstyle{definition}
  \newtheorem{example}[thm]{\protect\examplename}

\author{Sungkyung Kang}
\address{Mathematical Institute, University of Oxford, Andrew Wiles Building,
Radcliffe Observatory Quarter, Woodstock Road, Oxford, OX2 6GG, UK}
\email{sungkyung.kang@maths.ox.ac.uk}
\thanks{This project has received funding from the European Research Council (ERC) under the European Unions Horizon 2020 research and innovation programme (grant agreement No 674978).}


\makeatother

  \providecommand{\corollaryname}{Corollary}
  \providecommand{\definitionname}{Definition}
  \providecommand{\examplename}{Example}
  \providecommand{\lemmaname}{Lemma}
  \providecommand{\propositionname}{Proposition}
  \providecommand{\remarkname}{Remark}
  \providecommand{\theoremname}{Theorem}
\providecommand{\theoremname}{Theorem}

\begin{document}

\title{A transverse link invariant from $\mathbb{Z}_{2}$-equivariant Heegaard
Floer cohomology}

\author{Sungkyung Kang}
\begin{abstract}
We define an invariant of based transverse links, as a well-defined
element inside the equivariant Heegaard Floer cohomology of its branched
double cover, defined by Lipschitz, Hendricks, and Sarkar. We prove
the naturality and functoriality of equivariant Heegaard Floer cohomology
for branched double covers of $S^{3}$ along based knots, and then
prove that our transverse link invariant $c_{\mathbb{Z}_{2}}(\xi_{K})$
is an well-defined element which is always nonvanishing and functorial
under certain classes of symplectic cobordisms, and describe its behavior
under negative stabilization. It follows that we can use properties
of $c_{\mathbb{Z}_{2}}(\xi_{K})$ to give a condition on transverse
knots $K$ which implies the vanishing/nonvanishing of the contact
class $c(\xi_{K})$.
\end{abstract}

\maketitle

\section{Introduction}

In the paper \cite{OSz-contact}, Ozsvath and Szabo defined an element
$c(\xi)\in\widehat{HF}(M)$ associated to a contact 3-manifold $(M,\xi)$,
which is an invariant of the isotopy class of the given contact structure
$\xi$ on $M$. In particular, they defined an element in $\widehat{HF}(M)$
associated to an open book decomposition of $M$ which supports $\xi$,
and then proved its invariance under isotopy and positive stabilization.
Later, in the paper \cite{HKM-HF}, Honda, Kazez, and Matic provided
a new way to define the element $c(\xi)$, by working with Heegaard
diagrams induced by arc diagrams on open books.

In the paper \cite{eqv-Floer}, Lipshitz, Hendricks, and Sarkar defined
an $\mathbb{F}_{2}[\theta]$-module $HF_{\mathbb{Z}_{2}}(L_{1},L_{2})$
associated to a pair of Lagrangian submanifolds $L_{1},L_{2}$ in
a symplectic manifold $M$, where the group $\mathbb{Z}_{2}$ acts
on $M$ by symplectomorphisms and leaves $L_{1},L_{2}$ invariant
as sets. The equivariant Floer cohomology $HF_{\mathbb{Z}_{2}}(L_{1},L_{2})$
turned out to be invariant under $\mathbb{Z}_{2}$-invariant Hamiltonian
isotopies, and in some special cases, noninvariant Hamiltonian isotopies.
This construction was applied to construct an $\mathbb{F}_{2}[\theta]$-module
\[
\widehat{HF}_{\mathbb{Z}_{2}}(\Sigma(L),p)\in\text{Mod}_{\mathbb{F}_{2}[\theta]}
\]
 associated to a bridge diagram of a based link $(L,p)$ on a sphere,
whose isomorphism type is an invariant of the isotopy class of $(L,p)$.

In this paper, we construct an element $c_{\mathbb{Z}_{2}}(\xi_{L})\in\widehat{HF}_{\mathbb{Z}_{2}}(\Sigma(L),p)$,
where $L$ is a tranverse knot in the standard contact 3-sphere $(S^{3},\xi_{std})$,
by considering the contact branched double cover $(\Sigma(L),\xi_{L})$
of $(S^{3},\xi_{std})$, branched along $L$, as defined by Plamenevskaya\cite{contact-branch}.
We prove that this element is indeed a well-defined element inside
the equivariant Floer cohomology, and becomes an invariant of the
transverse (based) isotopy class of $(L,p)$. 

Then, before discussing the functoriality of $c_{\mathbb{Z}_{2}}(\xi_{std})$,
we first prove the functoriality of $\widehat{HF}_{\mathbb{Z}_{2}}$.
Using the techniques introduced by Juhasz and Thurston\cite{Juhasz-naturality},
we prove that the $\mathbb{F}_{2}[\theta]$-module $\widehat{HF}_{\mathbb{Z}_{2}}(\Sigma(K),p)$
is natural in the sense that it admits an action of $MCG(S^{3},K,p)$,
when $(K,p)$ is a based knot. Then, using the naturality of $\widehat{HF}_{\mathbb{Z}_{2}}$
for based knots, we prove that a based cobordism $S=(S_{0},s)$ between
two based knots $(K_{1},p_{1})$ and $(K_{2},p_{2})$ in $S^{3}$
defines a map 
\[
\hat{f}_{S}\,:\,\widehat{HF}_{\mathbb{Z}_{2}}(\Sigma(K_{2}),p_{2})\rightarrow\widehat{HF}_{\mathbb{Z}_{2}}(\Sigma(K_{1}),p_{1}),
\]
 which is an invariant of the isotopy class of $S$ rel $\partial S$.
Here, a based cobordism is a cobordism togther with a curve from $p_{1}$
to $p_{2}$.

After estabilishing the functoriality of $\widehat{HF}_{\mathbb{Z}_{2}}$,
we discuss the functoriality of the element $c_{\mathbb{Z}_{2}}(\xi_{std})$
inside $\widehat{HF}_{\mathbb{Z}_{2}}(\Sigma(K),p)$, when $(K,p)$
is a based transverse knot in $(S^{3},p)$. Recall that any smooth
cobordism between knots(or links) is a composition of isotopies, births,
saddles, and deaths. We define their analogues (except for deaths)
in the symplectic setting, which turn out to be well-defined up to
a weaker version of symplectic isotopy, and restrict our attention
to symplectically constructible symplectic cobordisms, which can roughly
be defined to be based symplectic cobordisms which can be weakly symplectically
isotoped to a composition of symplectic versions of isotopies, births,
and saddles. Then we prove that $c_{\mathbb{Z}_{2}}(\xi_{std})$ is
preserved under the maps associated to symplectically constructible
based cobordisms. 

To summarize, we prove the following theorem.
\begin{thm*}
For a based link $(K,p)$ in $S^{3}$, the $\mathbb{F}_{2}[\theta]$-module
$\widehat{HF}_{\mathbb{Z}_{2}}(\Sigma(K),p)$ is natural. Given a
based knot cobordism $(S,s)$ in $S^{3}\times I$ with $\partial S=K_{1}\times\{0\}\cup K_{2}\times\{1\}$
and $\partial s=\{(p_{1},0),(p_{2},1)\}$, we have an $\mathbb{F}_{2}[\theta]$-module
homomorphism 
\[
\hat{f}_{(S,s)}\,:\,\widehat{HF}_{\mathbb{Z}_{2}}(\Sigma(K_{2}),p_{2})\rightarrow\widehat{HF}_{\mathbb{Z}_{2}}(\Sigma(K_{1}),p_{1}),
\]
 which is an invariant of the isotopy class of $(S,s)$. When $(K,p)$
is a based transverse knot inside $(S^{3},\xi_{std})$, we have an
element 
\[
c_{\mathbb{Z}_{2}}(\xi_{K})\in\widehat{HF}(\Sigma(K),p)
\]
 which is an invariant of the transverse isotopy class of $(K,p)$.
Furthermore, if $(S,s)$ is a symplectically constructible symplectic
cobordism in the $(S^{3}\times I,d(e^{t}\alpha_{std}))$ between based
transverse knots, say, $\partial S=K_{1}\times\{0\}\cup K_{2}\times\{1\}$
and $\partial s=\{(p_{1},0),(p_{2},1)\}$, then we have 
\[
\hat{f}_{(S,s)}(c_{\mathbb{Z}_{2}}(\xi_{K_{2}}))=c_{\mathbb{Z}_{2}}(\xi_{K_{1}}).
\]
\end{thm*}
Note that, since moving the basepoint by an isotopy preserves the
element $c_{\mathbb{Z}_{2}}(\xi_{K})$, we can drop the choice of
a basepoint on $K$ and say that $c_{\mathbb{Z}_{2}}(\xi_{K})$ is
an invariant of the transverse isotopy class of a given transverse
knot $K$ in $(S^{3},\xi_{std})$. 

Surprisingly, basic properties of $c_{\mathbb{Z}_{2}}(\xi_{K})$ allows
us to give a condition on the self-linking number of $K$ which ensures
the vanishing/nonvanishing of the (non-equivariant) contact class
$c(\xi_{K})$. This can be seen as a Heegaard Floer analogue of a
similar condition, for the Plamenveskaya $\psi$-invariant, proven
in \cite{Plamenevskaya-2008}. The two results are shown to be the
same for transverse representatives of quasi-alternating knots.
\begin{thm}
Let $K$ be a knot in $S^{3}$ and $T$ be a transverse representative
of $K$. Then $c(\xi)\ne0$ if $d_{3}(\xi_{K})=\frac{q_{\tau}(K)-1}{2}$
and $c(\xi)=0$ if $d_{3}(\xi_{K})>\frac{q_{\tau}(K)-1}{2}+v_{\tau}(K)$.
\end{thm}

It is natural to ask whether the transverse knot invariant $c_{\mathbb{Z}_{2}}(\xi_{K})$
is effective, in the sense that it can distinguish between two knots
with the same self-linking number and topological knot type. There
are several transverse knot invariants defined previously. For example,
the Plamenevskaya invariant of transverse knots is an element of Khovanov
homology, as shown in \cite{plamanevskaya-Kh}; Wu\cite{Wu-sl(n)}
defined an $\mathfrak{sl}_{n}$ invariant of transverse knots, as
an element of $\mathfrak{sl}_{n}$ homology. Lisca, Ozsvath, Stipsicz,
and Szabo\cite{LOSS-inv} defined the LOSS invariant of Legendrian
and transverse knots, as an element of knot Floer homology, which
was proved to be the same (up to automorphisms of $\widehat{HFK}$)
as the HFK grid invariant by Baldwin and Vela-Vick\cite{GRID=00003DLOSS}.
Also, Ekholm, Etnyre, Ng, and Sullivan\cite{knot contact homology}
defined transverse homology, which is a filtered version of knot contact
homology. Baldwin and Sivek\cite{monopole-LOSS} defined a monopole
version of the LOSS invariant, as an element of the sutured monopole
homology, which is functorial under Lagrangian concordances and maps
to the LOSS invariant via isomorphism between KHM and HFK\cite{eqv-LOSS}.
Among them, the LOSS invariant, its monopole version, and the transverse
homology are proven to be effective. It is still not known whether
the others are effective invariants of transverse knots. We do not
know whether the invariant $c_{\mathbb{Z}_{2}}$ is effective.

\subsection*{Acknowledgement}

The author would like to thank Andras Juhasz, Robert Lipshitz, Marco
Golla, and Olga Plamenevskaya for helpful discussions and suggestions.
The author would also like to thank Hyunwoo Kwon for his help on various
figures drawn in this paper.

\section{Branched double covers along transverse links}

Recall that, given a link $L$ in $S^{3}$, we can remove its neighborhood,
take the double cover with respect to the meridian of its boundary,
and then reglue a solid torus along the boundary to get the branched
double cover $\Sigma(L)$ of $S^{3}$ along $L$. The covering transformation,
i.e. natural $\mathbb{Z}_{2}$-action on $\Sigma(L)$ is an orientation-preserving
homeomorphism.

Now suppose that we are working with the standard contact sphere $(S^{3},\xi_{std})$,
and a transverse link $L$ inside it. It has a standard contact neighborhood:
\[
N(L)\simeq(S^{1}\times D^{2},\ker(\alpha=d\phi+r^{2}d\theta)),
\]
 where $\phi$ parametrizes $S^{1}$ and $(r,\theta)$ are the polar
coordinates on $D^{2}$. The pullback of $\alpha$ along the branched
covering $p\,:\,(z,r,\theta)\mapsto(z,r^{2},2\theta)$ is given by
\[
p^{\ast}\alpha=dz+2r^{4}d\theta,
\]
 which satisfies the contact condition away from the fixed locus $L=\{r=0\}$.
However, the branched double cover construction of Plamenevskaya\cite{contact-branch}
tells us that we may consider the interpolated 1-forms 
\[
\alpha_{f}=dz+f(r^{2})d\theta,
\]
 for smooth increasing functions $f$ which satisfy $f(r^{2})=r^{2}$
near $r=0$ and $f(r^{2})=2r^{4}$ away from $r=0$. For such a function
$f$, the form $\alpha_{f}$ is always contact. Also, for any two
such functions $f$ and $g$, the forms $\alpha_{f}$ and $\alpha_{g}$
are obviously isotopic by a radial isotopy. Hence, by gluing the solid
torus to the double cover of $S^{3}-N(L)$, we see that the contact
branched double cover 
\[
(\Sigma(L),\xi_{L})=\widetilde{(S^{3}-N(L),\xi_{std})}\cup(S^{1}\times D^{2},\alpha_{f})
\]
 is well-defined up to isotopy supported near $L$.

Now we consider the case when $L$ is braided along the $z$-axis
in $S^{3}$. This notion can be made precise as follows:
\begin{defn}
Consider the genus 0 open book of $S^{3}$, as follows:
\[
\pi\,:\,S^{3}-\{z\text{-axis}\}\rightarrow S^{1}.
\]
 Then a link $L$ in $S^{3}$ is braided if it does not intersect
the $z$-axis and the map $\pi|_{L}\,:\,L\rightarrow S^{1}$ is a
regular covering map.
\end{defn}

Clearly, when $L$ is braided along the $z$-axis, it is a closed
braid. The corresponding braid word is unique up to positive/negative
stabilizations and conjugations inside the braid group.

When $L$ is transverse in $(S^{3},\xi_{std})$, we do not have positive
stabilizations, since they increase the self-linking number of $L$
by 2. However, we have the following theorems.
\begin{thm}
(Bennequin \cite{Bennequin-original}) For any transverse link $L$
in $(S^{3},\xi_{std})$, there exists a transverse braid $B$ around
the $z$-axis, such that $L$ is transversely isotopic to $B$.
\end{thm}

\begin{thm}
\label{thm:O-S}(Orevkov-Shevchishin \cite{Orevkov-Shevchishin})
Two (closed) transverse braids around the $z$-axis are transversely
isotopic as transverse links if and only if they are related by braid
isotopies, conjugations in the braid group, and positive braid stabilizations.
\end{thm}

We adopt the usual notation for braids. If we consider braids with
$n$ strings, the corresponding braid group is generated by the standard
generators $\sigma_{1},\cdots,\sigma_{n-1}$, where $\sigma_{i}$
creates a positive crossing between the $i$th and the $(i+1)$th
strands. This notation can also be applied to transverse braids around
the $z$-axis.

Now we recall the way to construct the contact branched double cover
from a braid representative of a given transverse link, due to Plamenevskaya\cite{contact-branch}.
Suppose that the transverse link $L$ is represented by a transverse
braid $B$ with a braid word $w=\sigma_{i_{1}}^{\pm1}\cdots\sigma_{i_{k}}^{\pm1}$.
Consider the disk $D$ with $n$ points $p_{1},\cdots,p_{n}$ in its
interior and pairwise (interior-)disjoint simple arcs $c_{i}$ connecting
$p_{i}$ and $p_{i+1}$. Then we define the following self-diffeomorphism
of $D$:
\[
h(w)=T_{i_{1}}^{\pm1}\cdots T_{i_{k}}^{\pm1}\in\mbox{Diff}^{+}(D,\partial D,\{p_{1},\cdots,p_{n}\}).
\]
 Here, $T_{i}$ denotes the Dehn half-twist along the arc $c_{i}$.
Clearly this definition depends on the choice of points $p_{i}$ and
arcs $c_{i}$, but since any two choices of such data are related
by a self-diffeomorphism of $D$(due to the fact that any two such
data are isotopic to each other), we see that $h(w)$ is well-defined
up to conjugation. Since the self-diffeomorphism $h(w)$ is orientation-preserving
and fixes the boundary $\partial D$ pointwise, the pair $(D,h)$
becomes an abstract open book of $S^{3}$. 

We then consider the branched double cover $S$ of $D$, branched
along the points $p_{1},\cdots,p_{n}$. The arcs $c_{i}$ smoothly
lift to smooth simple closed curves $C_{i}$ and the self-diffeomorphism
$h$ of $D$ lifts smoothly to a self-diffeomorphism $\tilde{h}$
of $S$, which is now given by products of positive/negative Dehn
twists along the curves $C_{i}$. By the argument used above, the
conjugacy class of $\tilde{h}$ is uniquely determined. Hence the
abstract open book $\mathcal{P}_{B}=(S,\tilde{h})$ is defined up
to diffeomorphism. Since this abstract open book has a natural $\mathbb{Z}_{2}$-action(i.e.
covering transformation) and the monodromy $\tilde{h}$ is equivariant
with respect to that action, $\mathcal{P}_{B}$ represents a uniquely
determined closed contact $3$-manifold $(M_{B},\xi_{M_{B}})$ with
a contact $\mathbb{Z}_{2}$-action, which turns out to be the contact
branched double cover of $(S^{3},\xi_{std})$ along $L$. In other
words, we get an open book description of $(\Sigma(L),\xi_{L})$.
\begin{thm}
\cite{contact-branch} We have a $\mathbb{Z}_{2}$-equivariant contactomorphism
\[
(M_{B},\xi_{M_{B}})\simeq(\Sigma(L),\xi_{L}).
\]
\end{thm}

\begin{rem*}
The construction of $(M_{B},\xi_{B})$ depends on the braid isotopy
class of $B$, not only on its link isotopy class. However, by the
definition of contact branched double covers, we know that if $L,L^{\prime}$
are transversely isotopic, then the double covers $(\Sigma(L),\xi_{L})$
and $(\Sigma(L^{\prime}),\xi_{L^{\prime}})$ are ($\mathbb{Z}_{2}$-equivariantly)
contactomorphic. Hence we see that, if two transverse braids $B,B^{\prime}$
are related by braid isotopies and positive stabilizations(conjugations
and positive stabilizations in terms of braid group elements), then
$(M_{B},\xi_{B})\simeq(M_{B^{\prime}},\xi_{B^{\prime}})$.
\end{rem*}

\section{$\widehat{HF}_{\mathbb{Z}_{2}}$ of branched double covers of $S^{3}$
along a based link}

Hendricks, Lipshitz, and Sarkar\cite{eqv-Floer} constructed the $\mathbb{F}_{2}[\theta]$-module
$\widehat{HF}_{\mathbb{Z}_{2}}(L_{0},L_{1})$ when $\mathbb{Z}_{2}$
acts symplectically on a symplectic manifold $(M,\omega)$ and $L_{0},L_{1}\subset M$
are transversely intersecting Lagrangians which are fixed by the $\mathbb{Z}_{2}$-action
and satisfy Hypothesis 3.2 of \cite{eqv-Floer}, and use it to define
$\widehat{HF}_{\mathbb{Z}_{2}}(\Sigma(L),z)$ for the branched double
cover $\Sigma(L)$, where $(L,z)$ is a based link in $S^{3}$. In
this section, we will briefly review their construction to make the
whole paper more self-contained.

\subsection*{Homotopy coherent diagrams of almost complex structures}

By a cylindrical complex structure on a symplectic manifold $(M,\omega)$,
we mean a smooth 1-parameter family $J=J(t)$, $0\le t\le1$, of almost
complex structures on $M$, compatible with $\omega$. By an eventually
cylindrical almost complex structure on $(M,\omega)$, we mean a smooth
1-parameter family $\tilde{J}(s)$, $s\in\mathbb{R}$, of cylindrical
complex structures, which is constant outside a compact subset of
$\mathbb{R}$, modulo translation by $\mathbb{R}$. Denote the set
of eventually cylindrical almost complex structures by $\mathcal{J}$.
Then it carries a natural topology by declaring that a sequence $\{\tilde{J}_{i}\}$
converges if and only if we can replace $\tilde{J}_{i}$ by their
representatives so that every $\tilde{J}_{i}$ is constant outside
a fixed compact set $C$ and $\{\tilde{J}_{i}|_{C}\}$ is convergent
in the $C^{\infty}$ topology. Given an element $J\in\mathcal{J}$,
its limit at $-\infty$ and $\infty$ are well-defined cylindrical
complex structures, which we will denote as $J_{-\infty}$ and $J_{+\infty}$.
Let $\mathcal{J}(J_{-\infty},J_{+\infty})$ be the subspace of $\mathcal{J}$
consisting of $J^{\prime}\in\mathcal{J}$ with $J_{-\infty}^{\prime}=J_{-\infty}$
and $J_{+\infty}^{\prime}=J_{+\infty}$. 

We now define a category $\overline{\mathcal{J}}$ as follows. Its
objects are cylindrical complex structures. For any cylindrical complex
structures $J$ and $J^{\prime}$, the morphism set $\overline{\mathcal{J}}(J,J^{\prime})$
consists of finite nonempty sequences 
\[
(\tilde{J}^{1},\cdots,\tilde{J}^{n})\in\mathcal{J}(J,J_{1})\times\cdots\times\mathcal{J}(J_{n-1},J^{\prime}),
\]
 where $J_{1},\cdots,J_{n-1}$ are cylindrical complex structures,
modulo the equivalence relation 
\[
(\tilde{J}^{1},\cdots,\tilde{J}^{i-1},\tilde{J}^{i},\tilde{J}^{i+1},\cdots,\tilde{J}^{n})\sim(\tilde{J}^{1},\cdots,\tilde{J}^{i-1},\tilde{J}^{i+1},\tilde{J}^{n}),
\]
 whenever $\tilde{J}^{i}$ is a constant path. Then we have a natural
composition map 
\[
\circ\,:\,\overline{\mathcal{J}}(J,J^{\prime})\times\overline{\mathcal{J}}(J^{\prime},J^{\prime\prime})\rightarrow\overline{\mathcal{J}}(J,J^{\prime\prime}),
\]
 defined by concatenation which makes $\overline{\mathcal{J}}$ into
a category. This category can be given a natural topology, which turns
it into a topological category; see Section 3.2 of \cite{eqv-Floer}
for details. 

Note that, since the space $\overline{\mathcal{J}}(J,J^{\prime})$
are weakly contractible, we can also consider continuous multi-parameter
families of elements in $\overline{\mathcal{J}}(J,J^{\prime})$ and
consider them as higher morphisms in the category $\overline{\mathcal{J}}$.
Thus we can define the notion of homotopy coherent diagrams in $\overline{\mathcal{J}}$
as follows.
\begin{defn}
Given a small category $\mathcal{C}$, a homotopy coherent $\mathcal{C}$-diagram
$F$ in $\overline{\mathcal{J}}$ consists of the following data.

(i) For each $x\in\text{ob}(\mathcal{C})$, an object $F(x)$ of $\overline{\mathcal{J}}$,

(ii) For each integer $n\ge1$ and each composable sequence of morphisms
$f_{1},\cdots,f_{n}$ in $\mathcal{C}$, i.e. $f_{i+1}\circ f_{i}$
is defined for each $i$, a continuous family 
\[
F(f_{n},\cdots,f_{1})\,:\,[0,1]^{n-1}\rightarrow\overline{\mathcal{J}}(F(x_{0}),F(x_{1})),
\]
 so that the conditions in Definition 3.3 of \cite{eqv-Floer} are
satisfied.
\end{defn}

The source category $\mathcal{C}$ which we will use frequently is
the groupoid $\mathcal{E}\mathbb{Z}_{2}$, which has two objects $a$
and $b$, and four morphisms, as follows.
\begin{itemize}
\item $\text{Hom}_{\mathcal{C}}(x,x)=\{\text{id}_{x}\}$ for $x=a,b$,
\item $\text{Hom}_{\mathcal{C}}(a,b)=\{\alpha\}$ and $\text{Hom}_{\mathcal{C}}(b,a)=\{\beta\}$.
\end{itemize}

\subsection*{The freed Floer complex}
\begin{defn}
\label{def:jholcurves}Given an morphism $\tilde{J}=(\tilde{J}^{1},\cdots,\tilde{J}^{n})\in\overline{\mathcal{J}}$,
where each $\tilde{J}^{i}\in\mathcal{J}(J_{i-1},J_{i})$ is nonconstant,
and points $x,y\in L_{0}\cap L_{1}$, a $\tilde{J}$-holomorphic disk
from $x$ to $y$ is a sequence 
\[
(v^{0,1},\cdots,v^{0,m_{0}},u^{1},v^{1,1}\cdots,v^{1,m_{1}},u^{2}\cdots,u^{n},v^{n,1},\cdots,v^{n,m_{n}}),
\]
 where $m_{0},\cdots,m_{n}$ are nonnegative integers and the following
conditions are satisfied.
\begin{itemize}
\item Each $v^{i,j}$ is a $J_{i}$-holomorphic Whitney disk with boundary
on $L_{0}$ and $L_{1}$, connecting some points $x^{i,j-1}$ and
$x^{i,j}$ in $L_{0}\cap L_{1}$.
\item Each $u^{i}$ is a $\tilde{J}^{i}$-holomorphic Whitney disk with
boundary on $L_{0}$ and $L_{1}$, connecting some points $x^{i-1}$
and $x^{i}$ in $L_{0}\cap L_{1}$.
\item $x^{i,0}=x^{i}$ for $i\ge1$, $x^{i,m_{i}}=x^{i+1}$ for all $i$,
$x^{0,0}=x$, and $x^{n,m_{n}}=y$.
\end{itemize}
\end{defn}

Given a map $\tilde{J}\,:\,[0,1]^{k}\rightarrow\overline{\mathcal{J}}$,
points $x,y\in L_{0}\cap L_{1}$, and a homotopy class $\phi\in\pi_{2}(x,y)$,
let $\mathcal{M}(x,y;\tilde{J})$ denote the moduli space of pairs
$(p,u)$ where $p\in[0,1]^{k}$ and $u$ is a $\tilde{J}(p)$-holomorphic
disk from $x$ to $y$, and $\mathcal{M}(\phi;\tilde{J})$ denotes
the subspace of disks representing the class $\phi$. Then we have
a splitting, as follows. 
\[
\mathcal{M}(x,y;\tilde{J})=\coprod_{\phi\in\pi_{2}(x,y)}\mathcal{M}(\phi;\tilde{J})
\]
 Since $\tilde{J}$ is a $k$-parameter family, the expected dimension
of $\mathcal{M}(\phi;\tilde{J})$ is given by $\mu(\phi)+k$. When
a homotopy coherent diagram $F\,:\,\mathcal{C}\rightarrow\overline{\mathcal{J}}$
is sufficiently generic, i.e. for any points $x,y\in L_{0}\cap L_{1}$,
a homotopy class $\phi\in\pi_{2}(x,y)$, and a composable sequence
of morphisms $f_{1},\cdots,f_{n}$ in $\mathcal{C}$, the space $\mathcal{M}(\phi;F(f_{n},\cdots,f_{1}))$
is transversely cut out, and its dimension is given by $\mu(\phi)+n-1$.

Now, given a suffciently generic homotopy coherent diagram $F\,:\,\mathcal{C}\rightarrow\overline{\mathcal{J}}$,
for each $a\in\text{ob}(\mathcal{C})$, let $G(a)$ be the Floer chain
complex $(CF(L_{0},L_{1}),\partial_{F(a)})$ with respect to the cylindrical
complex structure $F(a)$. For a composable sequence of morphisms
$f_{1},\cdots,f_{n}$ in $\mathcal{C}$, and a $k$-dimensional face
$\sigma$ of $[0,1]^{n-1}$, we have a $k$-dimensional subfamily:
\[
F(f_{n},\cdots,f_{1})|_{\sigma}\,:\,[0,1]^{k}\rightarrow\overline{\mathcal{J}}.
\]
 So we define:
\[
G(f_{n},\cdots,f_{1})(\sigma\otimes x)=\sum_{y\in L_{0}\cap L_{1}}\sum_{\phi\in\pi_{2}(x,y),\,\mu(\phi)=1-k}\left|\mathcal{M}(\phi;F(f_{n},\cdots,f_{1})|_{\sigma})\right|\cdot y.
\]
 Then $G\,:\,\mathcal{C}\rightarrow\text{Kom}_{\mathbb{F}_{2}}$ turns
out to be a homotopy coherent $\mathcal{C}$-diagram in the ($\infty$-)category
of complexes of $\mathbb{F}_{2}$-vector spaces, in the following
sense.
\begin{defn}
Given a small category $\mathcal{C}$, a homotopy coherent $\mathcal{C}$-diagram
$F$ in $\text{Kom}_{\mathbb{F}_{2}}$ consists of:
\begin{itemize}
\item For each $x\in\text{ob}(\mathcal{C})$, a chain complex $F(x)\in\text{ob}(\text{Kom}_{\mathbb{F}_{2}})$.
\item For each $n\ge1$ and each composable sequence $f_{1},\cdots,f_{n}$
of morphisms in $\mathcal{C}$, a chain map 
\[
G(f_{n},\cdots,f_{1})\,:\,I_{\ast}^{\otimes n-1}\otimes G(x_{0})\rightarrow G(x_{n}),
\]
 where $I_{\ast}=C_{\ast}^{\text{simplicial}}([0,1])$, such that
$G(f_{n},\cdots,f_{1})(t_{1}\otimes\cdots\otimes t_{n-1})$, where
$t_{i}\in I_{\ast}$, is equal to:
\begin{itemize}
\item $G(f_{n},\cdots,f_{2})(\pi(t_{1})\otimes t_{2}\otimes\cdots\otimes t_{n-1})$
if $f_{1}=[0,1]$, where $\pi\,:\,I_{\ast}\rightarrow\mathbb{F}_{2}$
is the map induced the projection $[0,1]\rightarrow\{\text{pt}\}$;
\item $G(f_{n},\cdots,f_{i+1},f_{i-1},\cdots,f_{1})(t_{1}\otimes\cdots\otimes m(t_{i-1}\otimes t_{i})\otimes\cdots\otimes t_{n-1})$
if $f_{i}=[0,1]$ and $1<i<n$, where $m\,:\,I_{\ast}\otimes I_{\ast}\rightarrow I_{\ast}$
is the map induced by the multiplication $[0,1]\times[0,1]\rightarrow[0,1]$,
\item $G(f_{n-1},\cdots,f_{1})(t_{1}\otimes\cdots\otimes t_{n-2}\otimes\pi(t_{n-1}))$
if $f_{n}=[0,1]$,
\item $G(f_{n},\cdots,f_{i+1}\circ f_{i},\cdots,f_{1})(t_{1}\otimes\cdots\otimes t_{i-1}\otimes t_{i+1}\otimes\cdots\otimes t_{n-1})$
if $t_{i}=\{1\}$,
\item $G(f_{n},\cdots,f_{i+1})(t_{i+1}\otimes\cdots\otimes t_{n-1})\circ G(f_{i},\cdots,f_{1})(t_{1}\otimes\cdots\otimes f_{i-1})$
if $t_{i}=\{0\}$.
\end{itemize}
\end{itemize}
\end{defn}

Since $G$ is homotopy coherent, we can consider its homotopy colimit
$\text{hocolim }G$, which is a single chain complex of $\mathbb{F}_{2}$-vector
spaces.
\begin{defn}
When $\mathcal{C}=\mathcal{E}\mathbb{Z}_{2}$ and the homotopy coherent
diagram $F\,:\,\mathcal{E}\mathbb{Z}_{2}\rightarrow\overline{\mathcal{J}}$
is $\mathbb{Z}_{2}$-equivariant and sufficiently generic, the chain
complex $\text{hocolim }G$ is defined as the freed Floer complex,
and denoted as $\widetilde{CF}_{\mathbb{Z}_{2}}(L_{0},L_{1})$. Since
we have a natural $\mathbb{Z}_{2}$-action on the freed Floer complex,
the complex 
\[
CF_{\mathbb{Z}_{2}}(L_{0},L_{1})=\text{Hom}_{\mathbb{F}_{2}[\mathbb{Z}_{2}]}(\widetilde{CF}_{\mathbb{Z}_{2}}(L_{0},L_{1}),\mathbb{F}_{2})
\]
 is defined as the equivariant Floer cochain complex, and its cohomology
\[
HF_{\mathbb{Z}_{2}}(L_{0},L_{1})=H^{\ast}(CF_{\mathbb{Z}_{2}}(L_{0},L_{1}))
\]
 is defined as the equivariant Floer cohomology.
\end{defn}

In the category $\mathcal{E}\mathbb{Z}_{2}$, any composable sequence
of morphisms is either of the form $\alpha_{n}=(\alpha,\beta,\alpha,\cdots)$
or of the form $\beta_{n}=(\beta,\alpha,\beta,\cdots)$. Thus the
elements of the freed Floer chain complex $\widetilde{CF}_{\mathbb{Z}_{2}}(L_{0},L_{1})$
are of the form $\alpha_{n}\otimes x$ or $\beta_{n}\otimes x$ for
$x\in L_{0}\cap L_{1}$. The differential is of the following form.
\begin{align*}
\partial(\alpha_{n}\otimes x) & =\alpha_{n}\otimes(\partial x)+\beta_{n-1}\otimes x+\sum_{i=1}^{n}\alpha_{n-i}\otimes\begin{cases}
G_{\beta,\alpha,\cdots}(x) & \text{if }n-i\text{ is odd}\\
G_{\alpha,\beta,\cdots}(x) & \text{if }n-i\text{ is even}
\end{cases}\\
\partial(\beta_{n}\otimes x) & =\beta_{n}\otimes(\partial x)+\alpha_{n-1}\otimes x+\sum_{i=1}^{n}\beta_{n-i}\otimes\begin{cases}
G_{\alpha,\beta,\cdots}(x) & \text{if }n-i\text{ is odd}\\
G_{\beta,\alpha,\cdots}(x) & \text{if }n-i\text{ is even}
\end{cases}
\end{align*}
 Here, $G_{\alpha,\beta,\cdots}(x)$ and $G_{\beta,\alpha,\cdots}(x)$
can be evaluated by counting holomorphic disks of Maslov index $2+i-n$
from $x$. Note that the holomorphic disks which we are counting here
are the ones defined in \ref{def:jholcurves}.

The $\mathbb{Z}_{2}$-action on the freed Floer chain complex is given
by 
\[
\tau(\alpha_{n}\otimes x)=\beta_{n}\otimes\tau x,
\]
 where $\mathbb{Z}_{2}=\left\langle \tau\right\rangle $. Hence the
elements of the equivariant Floer cochain complex $CF_{\mathbb{Z}_{2}}(L_{0},L_{1})$
are of the form $\theta^{n}\otimes x^{\ast}$, where $\theta$ is
a formal variable and $x^{\ast}$ is the dual of a given Floer generator
$x$, and its differential is given by 
\begin{equation}
d(\theta^{n}\otimes x^{\ast})=\theta^{n}\otimes dx^{\ast}+\theta^{n+1}\otimes\tau^{\ast}x^{\ast}+\sum_{i=1}^{\infty}\theta^{n+i}\otimes(x^{\ast}\circ G_{\alpha,\beta,\cdots}).\label{eq:eqvdiffeqn}
\end{equation}
 We have an action of $\mathbb{F}_{2}[\theta]$ on $CF_{\mathbb{Z}_{2}}(L_{0},L_{1})$,
given as follows.
\[
\theta\cdot(\theta^{n}\otimes x^{\ast})=\theta^{n+1}\otimes\tau^{\ast}x^{\ast}
\]
 Since the differential of $CF_{\mathbb{Z}_{2}}(L_{0},L_{1})$ is
$\theta$-equivariant, we get a natural $\mathbb{F}_{2}[\theta]$-module
structure on $HF_{\mathbb{Z}_{2}}(L_{0},L_{1})$.

The quasi-isomorphism type of $CF_{\mathbb{Z}_{2}}(L_{0},L_{1})$,
and thus the $\mathbb{F}_{2}[\theta]$-module isomorphism type of
$HF_{\mathbb{Z}_{2}}(L_{0},L_{1})$, turns out to be invariant under
the choice of sufficiently generic $\mathbb{Z}_{2}$-equivariant diagrams
in $\overline{\mathcal{J}}$, $\mathbb{Z}_{2}$-equivariant Hamiltonian
isotopies, and non-$\mathbb{Z}_{2}$-equivariant Hamiltonian isotopies
which satisfy the conditions in Proposition 3.25 of \cite{eqv-Floer}.
The proof is given in Proposition 3.23, 3.24, and 3.25 of \cite{eqv-Floer}.

When a generic $\mathbb{Z}_{2}$-equivariant cylindrical complex structure
achieves transversality for all Whitney disks of Maslov index at most
1, then we have the following isomorphism.
\[
CF_{\mathbb{Z}_{2}}(L_{0},L_{1})\simeq CF(L_{0},L_{1})\otimes_{\mathbb{F}_{2}[\mathbb{Z}_{2}]}^{L}\mathbb{F}_{2}.
\]

\subsection*{Diffeomorphism maps}

Given a $\mathbb{Z}_{2}$-equivariant symplectomorphism $\phi:M\rightarrow M^{\prime}$
which sends $\mathbb{Z}_{2}$-invariant Lagrangians $L_{0},L_{1}\subset M^{\prime}$
to $L_{0}^{\prime}=\phi(L_{0})$ and $L_{1}^{\prime}=\phi(L_{1})$,
we have a naturally defined chain isomorphism 
\[
\phi_{\ast}\,:\,\widetilde{CF}_{\mathbb{Z}_{2}}(L_{0},L_{1};J)\rightarrow\widetilde{CF}_{\mathbb{Z}_{2}}(L_{0}^{\prime},L_{1}^{\prime};\phi(J)),
\]
 for homotopy coherent diagrams $J$, when $(M,L_{0},L_{1})$ satisfies
Hypothesis 3.2 of \cite{eqv-Floer}. Hence we get a natural map between
equivariant Floer cohomology.
\[
\phi^{\ast}:HF_{\mathbb{Z}_{2}}(L_{0}^{\prime},L_{1}^{\prime},\phi(J))\xrightarrow{\sim}HF_{\mathbb{Z}_{2}}(L_{0},L_{1},J)
\]

\subsection*{Equivariant triangle maps}

Suppose that $L_{0}$, $L_{0}^{\prime}$, and $L_{1}$ are $\mathbb{Z}_{2}$-invariant
Lagrangians, which are pairwise transverse, and there is a $\tau$-invariant
$\omega$-compatible almost complex structure $J$ on $M$ which achieves
transversality for all moduli spaces of holomorphic disks with boundary
on $(L_{0},L_{0}^{\prime})$ of Maslov index at most 1. Fix a cocycle
$\Theta\in CF(L_{0},L_{0}^{\prime})$, which is $\mathbb{Z}_{2}$-invariant.
As in the proof of Proposition 3.25 of \cite{eqv-Floer}, when $L_{0}\cap L_{0}^{\prime}\cap L_{1}=\emptyset$,
define a category $\mathscr{D}$ as follows. (The case when $L_{0}\cap L_{0}^{\prime}\cap L_{1}$
is nonempty can be done by extending $\mathscr{D}$ to include all
possible Hamiltonian perturbations of $L_{1}$)
\begin{itemize}
\item $\text{ob}(\mathscr{D})=\{0,1\}\times\text{ob}(\overline{\mathcal{J}})$.
\item For $(i,J),\,(i,J^{\prime})\in\{i\}\times\text{ob}(\overline{\mathcal{J}})$,
$\text{Hom}_{\mathscr{D}}((i,J),(i,J^{\prime})=\text{Hom}_{\overline{\mathcal{J}}}(J,J^{\prime})=\mathcal{J}(J,J^{\prime})$.
\item $\text{Hom}_{\mathscr{D}}((0,J),(1,J^{\prime}))$ is the space of
sequences $(\tilde{J}_{-i},\cdots,\tilde{J}_{-1},\tilde{J}_{0},\tilde{J}_{1},\cdots,\tilde{J}_{j})$
where:
\begin{itemize}
\item For $k\ne0$, $\tilde{J}_{k}\in\mathcal{J}(J_{k},J_{k+1})$ for some
sequence $J_{-i},\cdots,J_{j+1}$ of cylindrical complex structures.
\item $J_{-i}=J$ and $J_{j+1}=J^{\prime}$.
\item $\tilde{J}_{0}\in\mathcal{J}_{\triangle}$ agrees with $J_{0}$ on
some cylindrical neighborhood $[n,\infty)\times[0,1]$ of $p_{2}$,
$J_{1}$ on some cylindrical neighborhood of $[n,\infty)\times[0,1]$
of $p_{3}$, and $J$ on some cylindrical neighborhood $[n,\infty)\times[0,1]$
of $p_{1}$.
\end{itemize}
\end{itemize}
Here, $\mathcal{J}_{\triangle}$ is the space of almost complex structures
parametrized by $\triangle$, where $\triangle$ is a disk with three
boundary punctures $p_{1},p_{2},p_{3}$, together with identifications
of a small closed neighborhood of $p_{i}$ with $[n,\infty)\times[0,1]$.
Then, like $\overline{\mathcal{J}}$, the category $\mathscr{D}$
also becomes a topological category.

For families $\tilde{J}\,:\,[0,1]^{\ell}\rightarrow\mathcal{J}_{\triangle}$
and a homotopy class $\phi$ of triangles in $(L_{0},L_{0}^{\prime},L_{1})$,
we can consider the moduli space 
\[
\mathcal{M}(\phi;\tilde{J})=\cup_{t\in[0,1]^{\ell}}\mathcal{M}(\phi;\tilde{J}(t)).
\]
 For generic $\tilde{J}$, the moduli space $\mathcal{M}(\phi;\tilde{J})$
is transversely cut out, and so we can define a map 
\[
G(\tilde{J})\,:\,CF(L_{0},L_{1};J_{0})\rightarrow CF(L_{0}^{\prime},L_{1};J_{1})
\]
 by the following equation. Here, $p\in\Theta$ means that $p$ runs
over all Floer generators appearing in the given cocycle $\Theta$.
\[
G(x)=\sum_{y\in L_{0}^{\prime}\cap L_{1}}\sum_{p\in\Theta}\sum_{\phi\in\pi_{2}(x,y,p),\,\mu(\phi)=-\ell}\left|\mathcal{M}(\phi;\tilde{J})\right|\cdot y.
\]
 From any homotopy coherent diagram $F_{0}\,:\,\mathcal{C}\rightarrow\mathscr{D}$,
we can use the map $G$ to construct a homotopy coherent diagram of
$\mathcal{C}$ in $\text{Kom}_{\mathbb{F}_{2}}$. Now, for sufficiently
generic $\mathbb{Z}_{2}$-equivariant homotopy coherent diagrams $F,F^{\prime}\,:\,\mathcal{E}\mathbb{Z}_{2}\rightarrow\overline{\mathcal{J}}$,
we can extend $(\{0\}\times F)\cup(\{1\}\times F^{\prime})$ to a
sufficiently generic $\mathbb{Z}_{2}$-equivariant homotopy coherent
diagram 
\[
G\,:\,\mathscr{I}\times\mathcal{E}\mathbb{Z}_{2}\rightarrow\mathscr{D}
\]
 where $\mathscr{I}$ has objects 0 and 1, and the only non-identity
morphism in $\mathscr{I}$ is given by $\text{Hom}_{\mathscr{I}}(0,1)=\{f_{0,1}\}$.
Applying Floer theory gives a homotopy coherent diagram 
\[
G^{\prime\prime}\,:\,\mathscr{I}\times\mathcal{E}\mathbb{Z}_{2}\rightarrow\text{Kom}_{\mathbb{F}_{2}}.
\]
 This gives a map $\text{hocolim }G^{\prime\prime}\,:\,\widetilde{CF}_{\mathbb{Z}_{2}}(L_{0},L_{1};F)\rightarrow\widetilde{CF}_{\mathbb{Z}_{2}}(L_{0}^{\prime},L_{1},F^{\prime})$.
Since it is $\mathbb{Z}_{2}$-equivariant, we get the equivariant
triangle map between equivariant Floer cochain complexes:
\[
F_{\Theta}\,:\,CF_{\mathbb{Z}_{2}}(L_{0},L_{1};F)\rightarrow CF_{\mathbb{Z}_{2}}(L_{0}^{\prime},L_{1},F^{\prime}).
\]

\subsection*{Equivariant Floer cohomology of branched double covers of $S^{3}$}

A based link is a pair $(L,p)$ where $L$ is a link in $S^{3}$ and
$p\in L$ is a choice of a basepoint. Given a genus 0 Heegaard surface
$\Sigma\subset S^{3}$, a based link $(L,p)$ is in a bridge position
with respect to $\Sigma$ if, for a Heegaard splitting $S^{3}=H_{a}\cup_{\Sigma}H_{b}$,
the connected arcs $\{a_{i}\}$ and $\{b_{i}\}$, $1\le i\le n$,
given by 
\[
\cup a_{i}=L\cap H_{a},\,\cup b_{i}=L\cap H_{b},
\]
 satisfy the following conditions.
\begin{itemize}
\item There exist disks $D_{a_{i}}$ and $D_{b_{j}}$ such that $a_{i}\subset\partial D_{a_{i}}\subset a_{i}\cup\Sigma$
and $b_{j}\subset\partial D_{b_{j}}\subset b_{j}\cup\Sigma$.
\item The disks $D_{a_{i}}$ and $D_{b_{j}}$ can be chosen to have pairwise
disjoint interiors.
\item $p\in L\cap\Sigma$.
\end{itemize}
If $\partial D_{a_{i}}=a_{i}\cup A_{i}$ and $\partial D_{B_{j}}=b_{j}\cup B_{j}$
where $A_{i}$ and $B_{j}$ are simple arcs on $\Sigma$, we say that
$(\{A_{i}\},\{B_{j}\})$ is the bridge diagram for the based link
$(L,p)$. Given a bridge diagram, by taking the branched double cover
of the whole diagram, with the branching locus given by $L\cap\Sigma$,
and removing the curves which contain the basepoint $p$, gives a
Heegaard diagram $(\tilde{\Sigma},\boldsymbol{\alpha},\boldsymbol{\beta},p)$
together with the covering $\mathbb{Z}_{2}$-action, where the alpha(beta)-curves
are given by the inverse images of the arcs $A_{i}$($B_{i}$).

The $\mathbb{Z}_{2}$-equivariant Floer cohomology theory can be applied
to the symplectic $\mathbb{Z}_{2}$-action on the symmectric power
$(\text{Sym}^{g}(\tilde{\Sigma}-\{p\}),\mathbb{T}_{\alpha},\mathbb{T}_{\beta})$.
It turns out that the $\mathbb{F}_{2}[\theta]$-isomorphism class
of the equivariant Floer cohomology 
\[
\widehat{HF}_{\mathbb{Z}_{2}}(\Sigma(L),p)=HF_{\mathbb{Z}_{2}}(\mathbb{T}_{\alpha},\mathbb{T}_{\beta})
\]
 is an invariant of the isotopy class of $(L,p)$. 

The proof of the invariance uses the fact that any two bridge diagram
of a based link are related by three types of moves: isotopy, handleslide,
and stabilization. An isotopy of bridge diagram $(\{A_{i}\},\{B_{j}\})$
is an isotopy of each arc $A_{i}$ and $B_{j}$, while fixing their
endpoints. A handleslide is a move which replaces $A_{i}$(or $B_{j}$)
by $A_{i}^{\prime}$(or $B_{j}^{\prime})$ with the same endpoints,
when there exists another arc $A_{k}$ such that $A_{i}$ and $A_{i}^{\prime}$
bound a disk $D$ which contains $A_{k}$, and $D$ does not intersect
any other A-arcs(or B-arcs). Finally, a stabilization is a move which
adds an A-arc and a B-arc to an arc $A_{i}$(or $B_{j})$ near one
of its endpoint; see Figure 13 of \cite{eqv-Floer} for details.

An isotopy of an arc induces a Hamiltonian isotopy of an alpha(beta)-curve
on the branched double cover, which induces an isomorphism of the
equivariant Floer cohomology $\widehat{HF}_{\mathbb{Z}_{2}}(\Sigma(L),p)$.
A handleslide also induced an isomorphism of the equivariant Floer
cohomology, by Proposition 3.25 of \cite{eqv-Floer}. A stabilization
can be seen as a combination of a creation of an unknot and a saddle
move, which can be translated as a composition of a stabilization
map, i.e. performing a connected sum with a genus 1 Heegaard diagram
of $S^{3}$, followed by an equivariant triangle map. The proof that
this also induces an isomorphism of equivariant cohomology is given
in the proof of Theorem 1.24 in \cite{eqv-Floer}.
\begin{rem*}
The equivariant triangle map can also be used to construct a cobordism
map in equivariant Floer cohomology, using the construction given
in Lemma 6.10 of \cite{eqv-Floer}. It is constructed by slicing a
given cobordism into basic pieces, i.e. cylinders, births, deaths,
and saddles. Saddles correspond to equivariant triangle maps, births/deaths
correspond to the stabilization/destabilization of the surface $\tilde{\Sigma}$,
and cylinders correspond to isotopy maps. After isotoping a given
bridge diagram as drawn in Figure 4 of \cite{eqv-Floer}, this saddle
map becomes the map induced by the surgery cobordism map between the
ordinary Heegaard Floer homology. We will show in this paper that
this cobordism map is independent of all auxiliary choices, and thus
is well-defined.
\end{rem*}

\section{Weak admissibility and naturality of $\widehat{HF}_{\mathbb{Z}_{2}}$}

Recall that, to define hat-versions of HF-groups and the maps between
them, we need to assure that all diagrams we consider satisfy weak
admissibility. In particular, to define the Floer chain complex, we
need weak admissibility for Heegaard diagrams. We also need admissibilities
for higher polygons; the triangle maps need weakly admissible triple-diagrams
and the proof of its associativity needs weakly admissible quadruple-diagrams.
Counting of higher polygons is not needed.

In this section, we will prove that all Heegaard (double, triple,
quadruple)-diagrms that we will use in this paper will be weakly admissible,
thus allowing us to freely use all hat-flavored aspects of Heegaard
Floer theory. The Heegaard diagrams that we will use are involutive
ones, which we will define as follows.
\begin{defn}
A Heegaard (double-)diagram $(\Sigma,\boldsymbol{\alpha},\boldsymbol{\beta},z)$
is involutive (with respect to an orientation-preserving involution
$\tau$) if $\tau$ fixes $z$ and the alpha- and beta-curves setwise
and $\tau|_{\alpha_{i}}$, $\tau|_{\beta_{j}}$ is orientation-reversing
with two distinct fixed points. 

A Heegaard triple-diagram $(\Sigma,\boldsymbol{\alpha},\boldsymbol{\beta},\boldsymbol{\gamma},z)$
is involutive with respect to $\tau$ if it is a small perturbation
of a diagram $(\Sigma,\boldsymbol{\alpha}_{0},\boldsymbol{\beta}_{0},\boldsymbol{\gamma}_{0},z)$
such that the tuples $(\Sigma,\boldsymbol{\alpha}_{0},\boldsymbol{\beta}_{0},z)$,
$(\Sigma,\boldsymbol{\beta}_{0},\boldsymbol{\gamma}_{0},z)$, $(\Sigma,\boldsymbol{\alpha}_{0},\boldsymbol{\gamma}_{0},z)$
are involutive with respect to $\tau$.

A Heegaard quadruple-diagram is involutive if it is a small perturbation
of a diagram such that the triple-diagrams given by choosing any three
of the curve systems among the given four are involutive with respect
to $\tau$.
\end{defn}

For simplicity, we usually do not specify the action of an involution
$\tau$ in figures, unless necessary. An example of possible local
pictures of Heegaard triple-diagrams and quadruple-diagrams around
the intersection points $\boldsymbol{\alpha}_{0}\cap\boldsymbol{\beta}_{0}\cap\boldsymbol{\gamma}_{0}(\cap\boldsymbol{\delta}_{0})$
are drawn in Figure \ref{fig:4.1}. Note that any diagram given by
a small perturbation can be obtained by permuting the labels of $\alpha,\beta,\cdots$
in the figures.

\begin{figure}[tbph]
\resizebox{.8\textwidth}{!}{\includegraphics{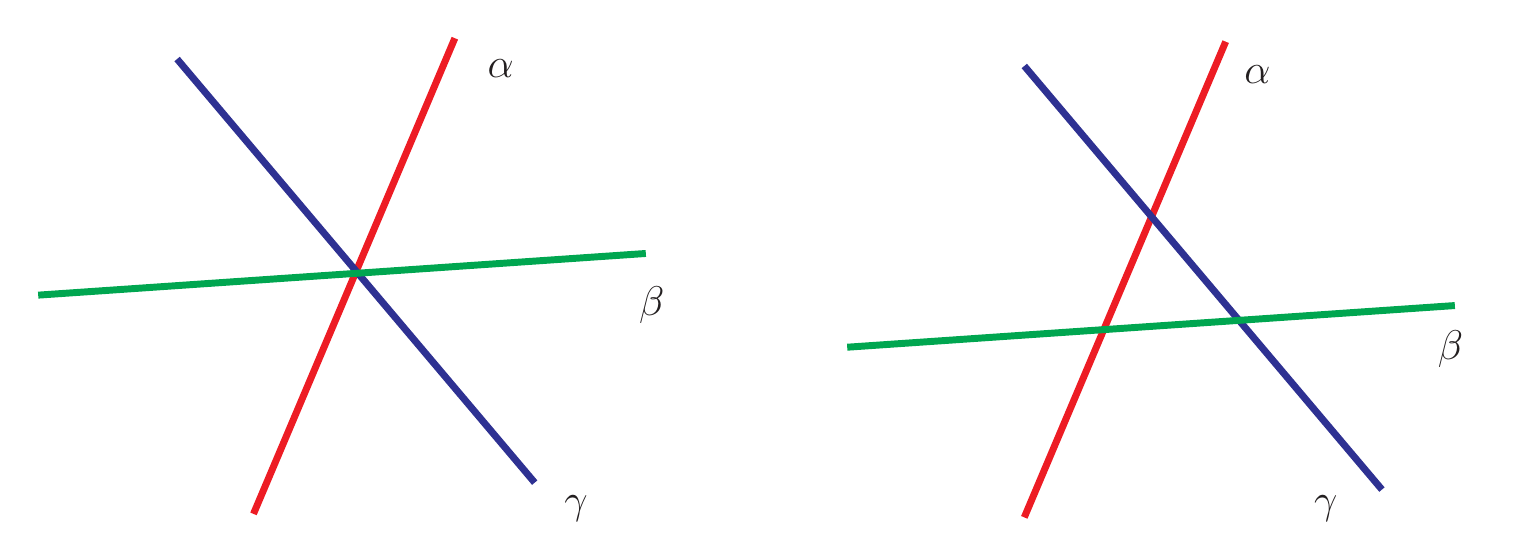}}

\resizebox{.8\textwidth}{!}{\includegraphics{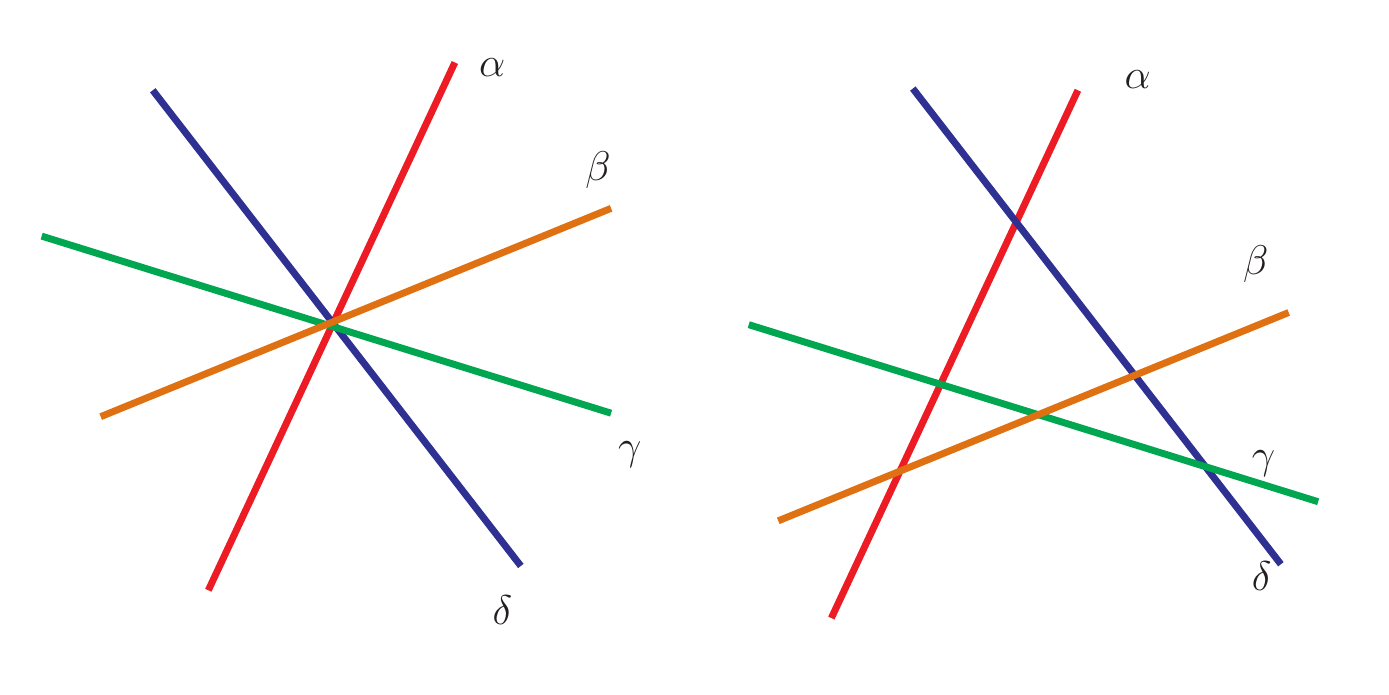}} \caption{\label{fig:4.1}Possible local pictures of nice Heegaard diagrams.
A choice of a perturbation changes the figures on the left to the
figures of nice diagrams, shown on the right.}
\end{figure}
\begin{lem}
Every involutive Heegaard diagram is weakly admissible.
\end{lem}

\begin{proof}
Suppose that a nontrivial periodic domain $D$ with nonnegative coefficients
is given. If we denote the involution by $\tau$, the domain $\tilde{D}=D+\tau_{\ast}D$
is nontrivial, periodic, has nonnegative coefficients and is $\tau$-invariant.
Thus, in a neighborhood of any $\tau$-invariant point $p\in\alpha_{i}\cap\beta_{j}$,
the domain $\tilde{D}$ should have coefficients as described in Figure
\ref{fig:4.2}.

By periodicity, $2a=2b$, i.e. $a=b$. But then, near an adjacent
intersection point $q\in\alpha_{i}\cap\beta_{k}$, the domain $\tilde{D}$
is given as in Figure \ref{fig:4.3}. By periodicity, $a+c=a+d$,
i.e. $c=d$. Continuing in this manner, we see that if two components
of $\Sigma-\cup\alpha_{i}-\cup\beta_{j}$ share a segment of a beta-curve,
then the coefficients of $\tilde{D}$ for those components are the
same. But by the same argument, we can prove the same for components
sharing a segment of a alpha-curve. This implies that all coefficients
of $\tilde{D}$ should be the same. Since $n_{z}(\tilde{D})=0$, we
deduce that $\tilde{D}=0$. Contradiction.
\end{proof}
\begin{figure}[tbph]
\resizebox{.4\textwidth}{!}{\includegraphics{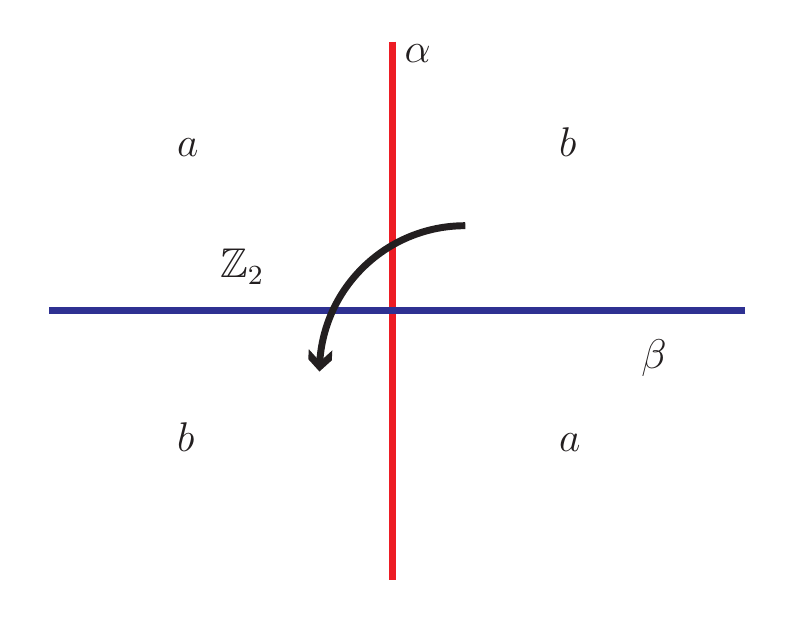}} \caption{\label{fig:4.2}The periodic domain $\tilde{D}$ near an invariant
intersection point.}
\end{figure}
\begin{figure}[tbph]
\resizebox{.4\textwidth}{!}{\includegraphics{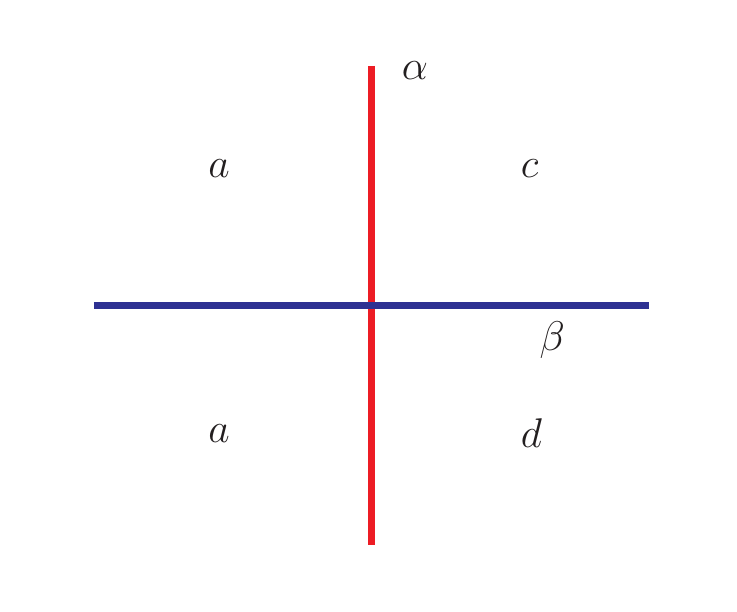}} \caption{\label{fig:4.3}The periodic domain $\tilde{D}$ near a non-invariant
intersection point.}
\end{figure}
\begin{lem}
Every involutive Heegaard triple-diagram $\mathcal{D}=(\Sigma,\boldsymbol{\alpha},\boldsymbol{\beta},\boldsymbol{\gamma},z)$
is weakly admissible.
\end{lem}

\begin{proof}
Suppose that a nontrivial triply-periodic domain $D$ with nonnegative
coefficients in $\mathcal{D}$ is given. If we denote the involution
by $\tau$, $\tau_{\ast}D$ would be well-defined away from the triple
intersections. A typical local picture near a triple intersection
point is drawn in Figure \ref{fig:4.4}.

We denote the coefficients by $a,b,c,d,e,f,g$ as in Figure \ref{fig:4.4},
and claim that there exists an integer $g^{\prime}$ which makes the
domain described in Figure \ref{fig:4.5} achieve periodicity. Define
$g^{\prime}$ as $g^{\prime}=a+c-b$. Then we have 
\[
c+e-d=b+g^{\prime}-a+e-d=g^{\prime}+g-f+f-g=g^{\prime},
\]
 and similarly $a+e-f=g^{\prime}$, so our choice of $g^{\prime}$
makes the domain periodic; denote the resulting periodic domain by
$\tau_{\ast}D$. Then $\tau_{\ast}D$ may not have nonnegative coefficients,
since we do not know whether $g^{\prime}\ge0$ holds. However, we
have 
\begin{align*}
g+g^{\prime} & =g+a+c-b\\
 & =g+f+b-g+d+b-g-b\\
 & =f+d+b-g\\
 & =e+b\ge0.
\end{align*}
 Hence $\tilde{D}=D+\tau_{\ast}D$ is a periodic domain with nonnegative
coefficients. Also, since $a,\cdots,f=0$ implies $g=0$, we see that
$\tilde{D}\ne0$. Now, the local picture of $\tilde{D}$ near triple
intersections is given as in Figure \ref{fig:4.6}. By periodicity,
we get 
\begin{align*}
\tilde{a}+\tilde{b} & =\tilde{c}+\tilde{d},\\
\tilde{a}+\tilde{c} & =\tilde{b}+\tilde{d},\\
\tilde{a}+\tilde{d} & =\tilde{b}+\tilde{c}.
\end{align*}
 Thus $\tilde{a}=\tilde{b}=\tilde{c}=\tilde{d}$. Now, by the argument
used to prove the previous lemma, we see that all coefficients of
$\tilde{D}$ are the same. Since $n_{z}(\tilde{D})=0$, we get $\tilde{D}=0$,
which is a contradiction.
\end{proof}
\begin{figure}[tbph]
\resizebox{.3\textwidth}{!}{\includegraphics{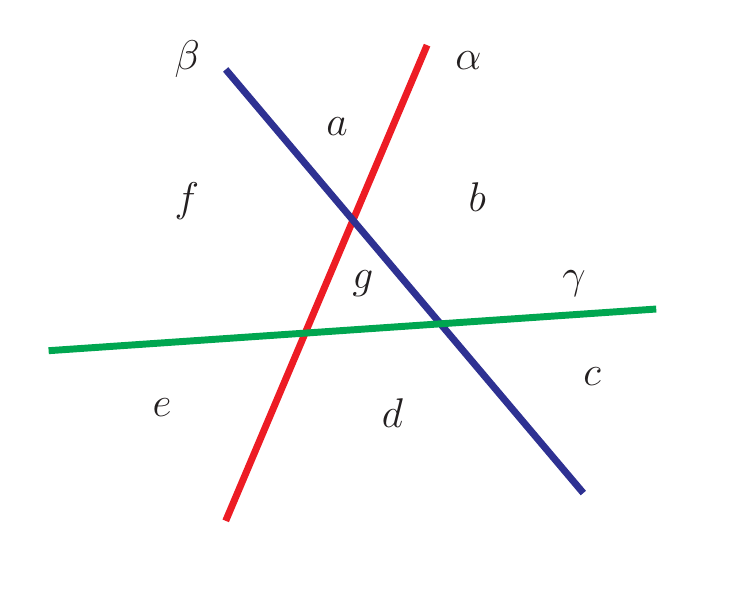}} \caption{\label{fig:4.4}The periodic domain $D$ near a non-invariant intersection
point.}
\end{figure}
 
\begin{figure}[tbph]
\resizebox{.3\textwidth}{!}{\includegraphics{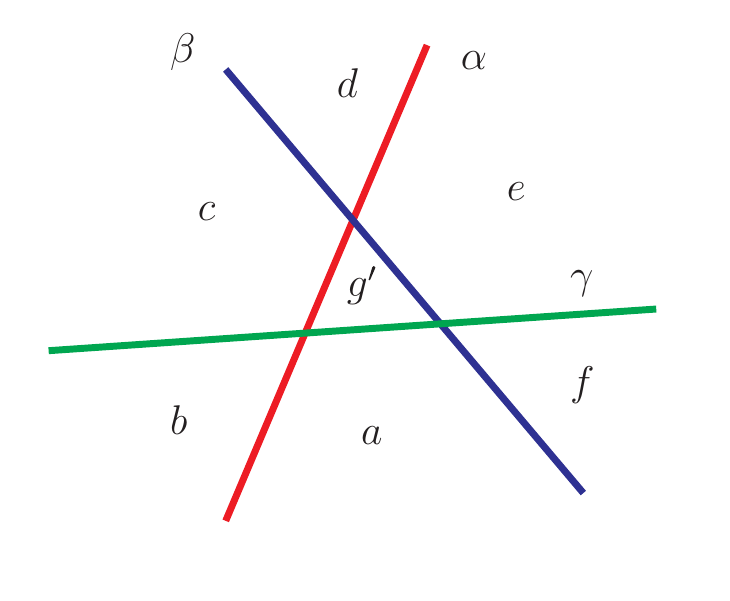}} \caption{\label{fig:4.5}The periodic domain $\tau_{\ast}D$ near a non-invariant
intersection point.}
\end{figure}
\begin{figure}[tbph]
\resizebox{.3\textwidth}{!}{\includegraphics{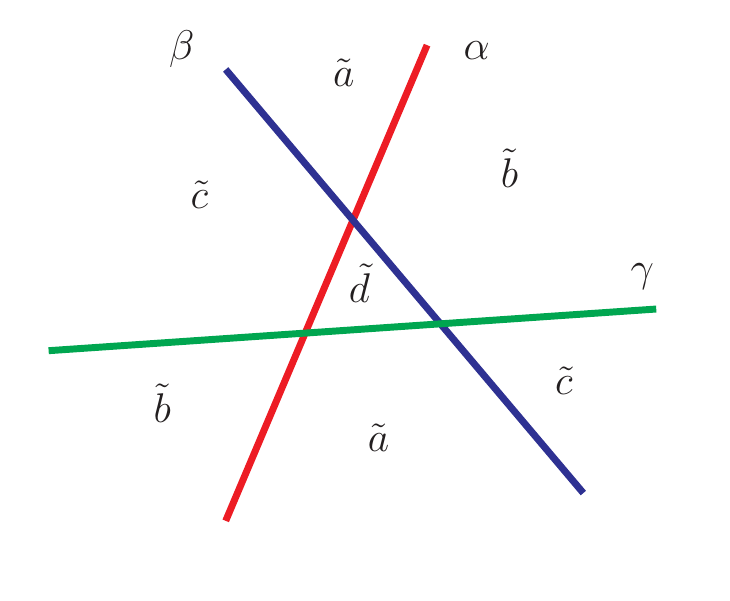}} \caption{\label{fig:4.6}The periodic domain $\tilde{D}$ near a non-invariant
intersection point.}
\end{figure}
 
\begin{lem}
Every involutive Heegaard quadruple-diagram is weakly admissible.
\end{lem}

\begin{proof}
We continue to use the above approach and start from Figure \ref{fig:4.7}
for a nontrivial quadruply-periodic domain $D$ with nonnegative coefficients.
Suppose that we are given $a,b,c,d,e,f,g,h$ and try to find $x,y,z$
so that the resulting domain becomes periodic. If such $x,y,z$ exists,
then we must have 
\begin{align*}
y & =b+h-f,\\
z & =b+g-d,\\
x & =c+e-a.
\end{align*}
 The remaining three equations are 
\begin{align*}
c+y-h-x & =0,\\
x+g-e-z & =0,\\
b+x-y-z & =0.
\end{align*}
Using the first three set of equations, we get the followng.
\begin{align*}
0=c+y-h-x & =c+b+h-f-h-c-e+a\\
 & =b-f-e+a,\\
0=x+g-e-z & =c+e-a+g-e-b-g+d\\
 & =c+d-a-b,\\
0=b+x-y-z & =b+c+e-a-b-h+f-b-g+d\\
 & =c+d+e+f-a-b-g-h.
\end{align*}
 Hence we get 
\[
a+b=c+d=e+f=g+h,
\]
 which we will call as the filling condition. Now, if a given $a,b,c,d,e,f,g,h$
satisifes the filling condition, we can reverse the above argument
to deduce that there exists a unique choice of $x,y,z$ which makes
the resulting domain periodic.

Here we notice that the filling condition is invariant respect to
the change 
\[
a\leftrightarrow b,c\leftrightarrow d,e\leftrightarrow f,g\leftrightarrow h.
\]
 This implies that there exists a unique choice of integers $x^{\prime},y^{\prime},z^{\prime}$
making the domain described in Figure \ref{fig:4.8} periodic, where
coefficients for all other domains are transformed by the involution
$\tau$, whch makes the given quadruple-diagram nice. Denote the resulting
domain by $\tau_{\ast}D$. Then $\tilde{D}=D+\tau_{\ast}D$ has the
local picture as in Figure \ref{fig:4.9}, near any $\tau$-invariant
intersection point, by the filling condition. Note that we obviously
have $\tilde{D}\ne0$. Now, by the argument used for Heegaard (double)
diagrams, we deduce that all coefficients of $\tilde{D}$ are the
same. Since $n_{z}(\tilde{D})=0$, we must have $\tilde{D}=0$, which
is a contradiction.
\end{proof}
\begin{figure}[tbph]
\resizebox{.3\textwidth}{!}{\includegraphics{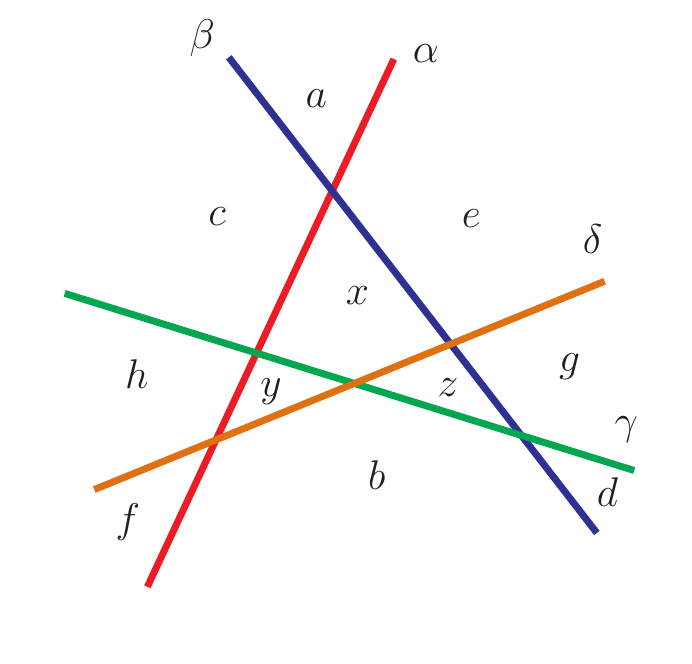}} \caption{\label{fig:4.7}The periodic domain $D$ near a non-invariant intersection
point.}
\end{figure}
 
\begin{figure}[tbph]
\resizebox{.3\textwidth}{!}{\includegraphics{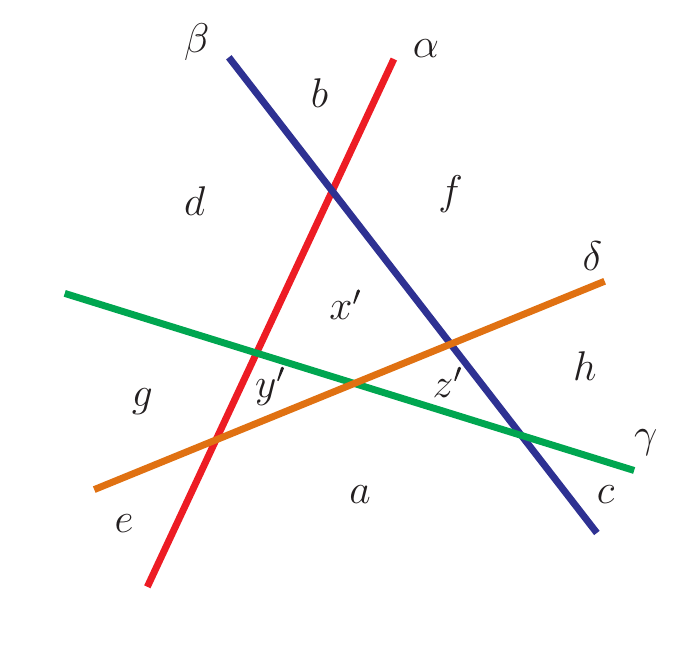}} \caption{\label{fig:4.8}The periodic domain $D$ near a non-invariant intersection
point.}
\end{figure}
 
\begin{figure}[tbph]
\resizebox{.3\textwidth}{!}{\includegraphics{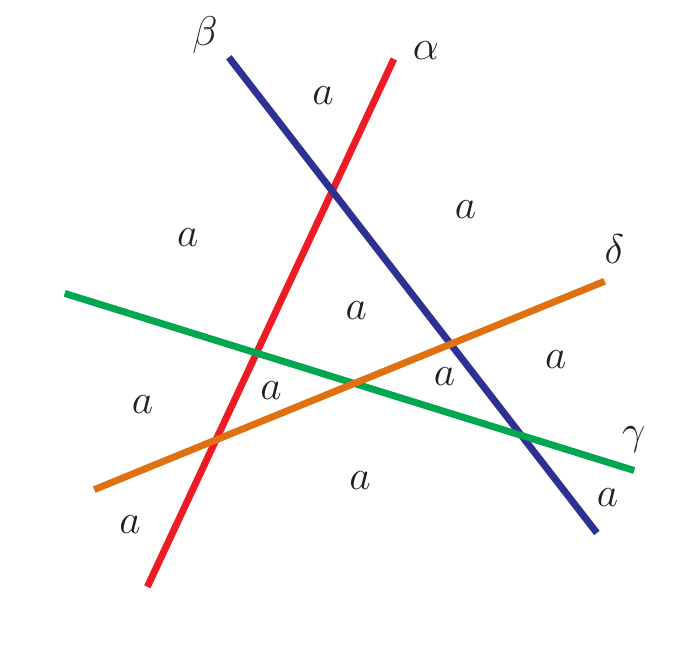}} \caption{\label{fig:4.9}The periodic domain $D$ near a non-invariant intersection
point.}
\end{figure}

By the above results, we see that we can now freely talk about counting
holomorphic disks, triangles, and squares in involutive diagrams.
Proposition 3.25 of \cite{eqv-Floer} proves that counting triangles
of Maslov indices at most zero gives a map between equivariant Floer
(co)homologies. 

\section{The equivariant contact class}

Choose a system of pairwise disjoint simple arcs $a_{i}^{0}$ in $D$,
where $a_{i}^{0}$ connects $p_{i}$ to a boundary point. Pick a point
$p_{1}$ among $p_{1},\cdots,p_{n}$ and regard it as a basepoint;
$z:=p_{1}$ (such a system is called a \textbf{half-arc basis}). The
arcs $a_{i}^{0}$ lift to nonseparating smooth simple arcs $a_{i}$
in $S$, which pass through $p_{i}$ and connect two points in $\partial S$.
We claim that $\{a_{i}\}_{i\ne1}$ is an arc basis on the surface
$S$, so that the data $(S,\{a_{i}\}_{i\ne1},\tilde{h},z=p_{1})$
defines an open book diagram of $(M_{B},\xi)$ in the sense of \cite{HKM-HF},
which is invariant under the covering transformation. To prove this,
recall that a pairwise disjoint system of simple arcs $\{a_{i}\}\subset S$
is called an arc basis if the two endpoints of each $a_{i}$ lie on
$\partial S$ and $S-\cup a_{i}$ is a disk. Since each $a_{i}$ is
a lift of $a_{i}^{0}$ and $D-\cup_{i\ne1}a_{i}^{0}$ is a disk with
one branching point $p_{1}$ in its interior, its inverse image $S-\cup_{i\ne1}a_{i}$
is the double cover of a disk branched along an interior point, which
is a disk. This proves our claim.
\begin{example*}
Suppose that we are given a disk with four marked points $p_{1},\cdots,p_{4}$
in its interior and a half-arc basis $\{a_{2}^{0},a_{3}^{0},a_{4}^{0}\}$
given as in Figure \ref{fig:5.1}.
\begin{figure}[tbph]
\resizebox{.4\textwidth}{!}{\includegraphics{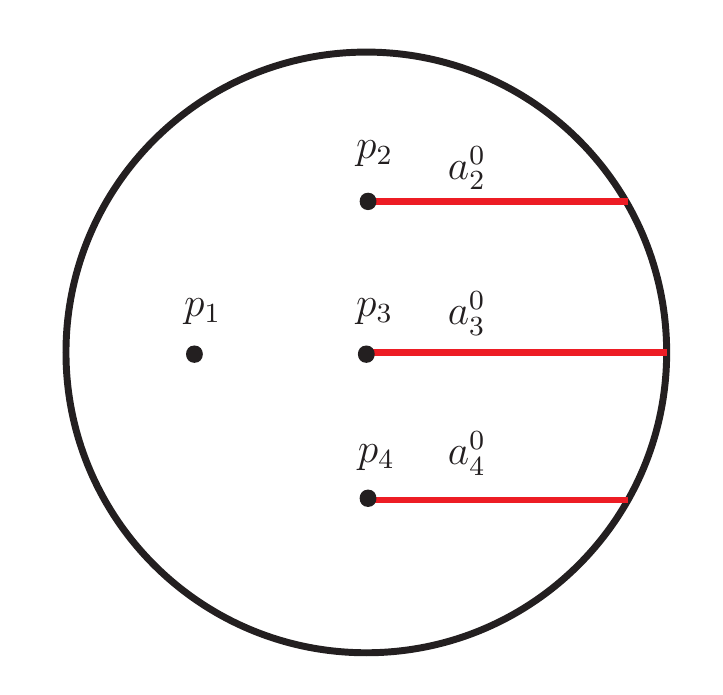}} \caption{\label{fig:5.1}A disk with four marked points and a half-arc basis}
\end{figure}

After taking the branched double cover, we get the twice-puncture
torus with three arcs $a_{2},a_{3},a_{4}$. From Figure \ref{fig:5.2},
We see that $S-(a_{2}\cup a_{3}\cup a_{4})$ is a disk, i.e. $\{a_{2},a_{3},a_{4}\}$
gives an arc basis on $S$.
\begin{figure}[tbph]
\resizebox{.8\textwidth}{!}{\includegraphics{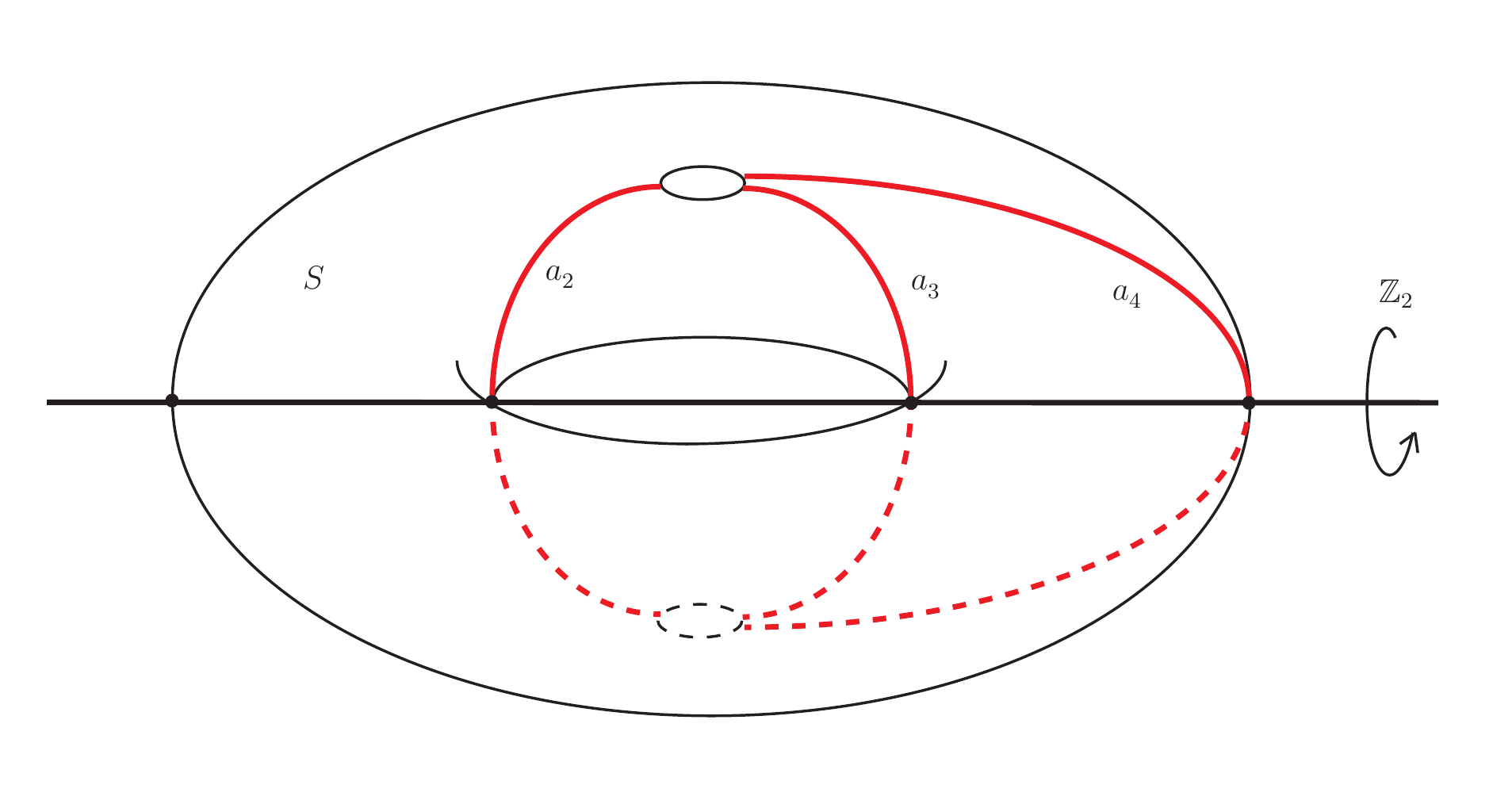}} \caption{\label{fig:5.2}The branched double cover.}
\end{figure}
 
\end{example*}
Thus, if we define the following $\alpha$- and $\beta$-curves on
the Heegaard surface 
\[
\Sigma=(S\times\{0,1\})/\partial S\times\{0,1\},
\]
 we get a $\mathbb{Z}_{2}$-invariant Heegaard diagram of $M_{B}$:
\begin{eqnarray*}
\alpha_{i} & = & a_{i}\cup a_{i},\\
\beta_{i} & = & b_{i}\cup\tilde{h}(b_{i}),
\end{eqnarray*}
 where $b_{i}$ is a slight perturbation of $a_{i}$ in positive direction
along an orientation of $\partial S$ and $|a_{i}\cap b_{i}|=1$.
For simplicity, we denote the sets of $\alpha_{i}$ and $\beta_{i}$
as $\boldsymbol{\alpha}$ and $\boldsymbol{\beta}$. If we denote
the half-arc basis which we have started with by $\mathcal{A}$, then
the contact element $EH(\xi_{K},\mathcal{A})$ is defined as the element
$\{p_{1},\cdots,p_{n}\}$ in the Heegaard Floer cochain complex of
$M_{B}$, as in section 3.1 of \cite{HKM-HF}.
\[
EH(\xi_{K},\mathcal{A})\in CF^{\ast}(\Sigma,\boldsymbol{\alpha},\boldsymbol{\beta},z)\simeq CF^{\ast}(M_{B}),
\]
 which is a $\mathbb{Z}_{2}$-invariant cocycle. This element induces
the following element in the equivariant Floer cochain complex
\[
EH_{\mathbb{Z}_{2}}(\xi_{K},\mathcal{A})=\theta^{0}\otimes EH(\xi_{K})\in\widehat{CF}_{\mathbb{Z}_{2}}(\Sigma,\boldsymbol{\alpha},\boldsymbol{\beta},z).
\]
 When the choice of a half-arc basis $\mathcal{A}$ is not important,
we will drop $\mathcal{A}$ and write $EH_{\mathbb{Z}_{2}}(\xi_{K})$
for simplicity.
\begin{rem*}
As we have seen above, given a half-arc basis $\{a_{i}^{0}\}$ and
a monodromy $h$ of $D^{2}$ which fixes the points $p_{i}$, we can
take its branched double cover along the points $p_{i}$ to get an
arc basis $\{a_{i}\}$ in the open book diagram $(S,\tilde{h})$,
and applying the Honda-Kazez-Mati\'{c} construction to it gives a
$\mathbb{Z}_{2}$-invariant Heegaard diagram. Now, if we apply the
Honda-Kazez-Mati\'{c} construction directly to the given half-arc
basis, what we get is a bridge diagram of a link in $S^{3}$, drawn
on a genus 0 Heegaard surface, as follows.
\begin{align*}
S^{2}\simeq\Sigma= & (D^{2}\times\{0,1\})/(\partial D^{2}\times\{0,1\}),\\
A_{i} & =a_{i}^{0}\cup a_{i}^{0},\\
B_{i} & =b_{i}^{0}\cup\tilde{h}(b_{i}^{0}).
\end{align*}
 Here, $b_{i}^{0}$ is a slight perturbation of $a_{i}^{0}$ along
the positive direction of $\partial D^{2}$, so that $a_{i}^{0}$
and $b_{i}^{0}$ intersect only at the endpoint of $a_{i}^{0}$ which
lies in the interior of $D^{2}$. Then, taking its branched double
cover along the set $\{p_{1},p_{2},\cdots\}\times\{0,1\}$ also gives
a $\mathbb{Z}_{2}$-invariant Heegaard diagram. The two Heegaard diagrams
we get are identical. To summarize, we have a following commutative
diagram of objects which we consider in this paper.
\[
\xymatrix{\text{half-arc diagram}\ar[rrr]^{\text{branched double cover}}\ar[d]_{\text{HKM construction}} &  &  & \text{arc diagram}\ar[d]^{\text{HKM construction}}\\
\text{bridge diagram}\ar[rrr]^{\text{branched double cover}} &  &  & \text{Heegaard diagram}
}
\]
\end{rem*}
Now we argue that, for a generic almost complex structure $\mathfrak{j}$
on $S^{2}$, the symmetric product of the lifted structure $\tilde{\mathfrak{j}}$
on $\Sigma$ achieves equivariant transversality.
\begin{thm}
\label{thm:eqvtrans}For a generic 1-parameter family of almost complex
structures $\mathfrak{j}$ on $S^{2}$, the $\mathbb{Z}_{2}$-equivariant
cylindrical complex structure $\text{Sym}^{g}(\tilde{\mathfrak{j}})$
on $\text{Sym}^{g}(\Sigma)$, where $g$ is the genus of $\Sigma$,
achieves equivariant transversality, in the sense of \cite{eqv-Floer}.
\end{thm}

\begin{proof}
Since the $\mathbb{Z}_{2}$-invariant locus $(\mathbb{T}_{\alpha}\cap\mathbb{T}_{\beta})^{inv}$
consists of 0-dimensional components of $(\text{Sym}^{g}(\Sigma))^{inv}$
(see Section 6.1 of \cite{eqv-Floer} for details), for any choice
of an almost complex structure $J$ on $\text{Sym}^{g}(\Sigma)$,
any $J$-holomorphic disk connecting two points in $\mathbb{T}_{\alpha}\cap\mathbb{T}_{\beta}$
is not completely contained in $(\text{Sym}^{g}(\Sigma))^{inv}$.
Thus, as in the proof of Proposition 5.13 in \cite{quiver-Floer},
the argument used in the proof of Corollary 1.12 in \cite{eqv-Floer}
actually gives transversality for all homotopy classes of Whitney
disks in this case.
\end{proof}
The above theorem tells us that we only have to work with $\mathbb{Z}_{2}$-invariant
almost complex structures of the form $\text{Sym}^{g}(\tilde{\mathfrak{j}})$.
For such almost complex structures, the argument in Section 3.1 of
\cite{HKM-HF} tells us that there are no holomorphic disks going
towards $EH(\xi_{K})$. Thus, if we denote the generator of $\mathbb{Z}_{2}$
as $\tau$, the higher degree terms in the formula \ref{eq:eqvdiffeqn}
vanishes, except for the term $\theta^{1}\otimes EH(\xi_{K})$ which
comes from the constant disk of Maslov index 0. So the following equality
holds.
\[
d_{\mathbb{Z}_{2}}(EH_{\mathbb{Z}_{2}}(\xi_{K}))=\theta^{0}\otimes d(EH(\xi_{K}))+\theta^{1}\otimes(EH(\xi_{K})+\tau^{\ast}EH(\xi_{K}))=0.
\]
 Hence $EH_{\mathbb{Z}_{2}}(\xi_{K})$ is a cocycle, i.e. defines
a cohomology class 
\[
c_{\mathbb{Z}_{2}}(\xi_{K}):=[EH_{\mathbb{Z}_{2}}(\xi_{K})]\in\widehat{HF}_{\mathbb{Z}_{2}}(\Sigma(K)).
\]
\begin{defn}
Given a half-arc basis $\{a_{i}^{0}\}_{i\ne1}$ on a disk $D$, suppose
that an (half-)arc $b$ starting at $p_{i}$ and ending in $\partial D$
satisfies the property that there exits a unique $j\ne i,1$ such
that $p_{j}$ is contained in the region bounded by $\partial D$,
$a_{i}^{0}$, and $b$. If $a_{j}^{0}$ is also contained in that
region, we say that $\{a_{k}^{0}\}_{k\ne1,i}\cup\{b\}$ is obtained
by performing a half-arc slide of $a_{i}^{0}$ along $a_{j}^{0}$.
\end{defn}

\begin{figure}[tbph]
\resizebox{.4\textwidth}{!}{\includegraphics{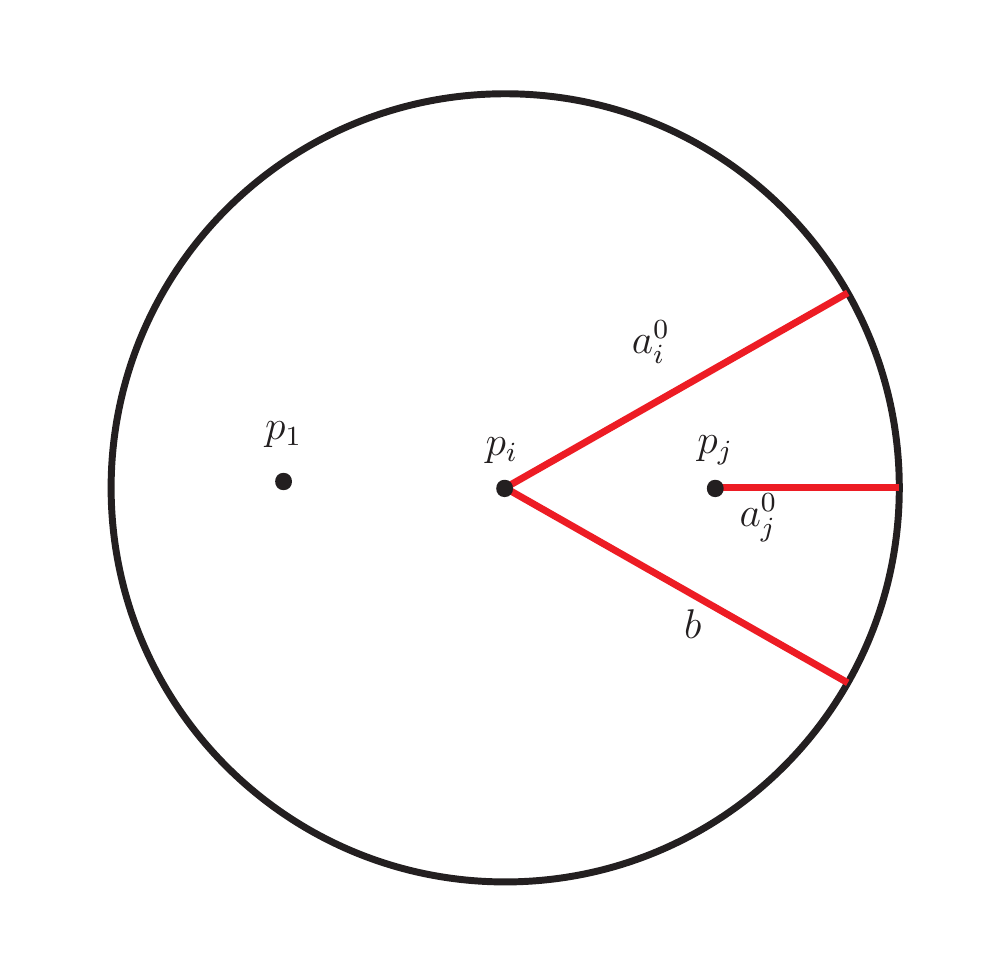}} \caption{\label{fig:5.3}A picture describing a half-arc slide.}
\end{figure}
\begin{prop}
\label{prop:halfarc}Let $\{a_{i}\}_{1<i\le n},\{b_{i}\}_{1<i\le n}$
on a disk $D$ be two half-arc bases, where $a_{i}$ and $b_{i}$
connect an interior point $p_{i}\in\text{int}(D)$ with $\partial D$.
Then they are related by a sequence of isotopies and half-arc slides.
\end{prop}

\begin{proof}
We can isotope the arcs so that $\partial a_{i}=\partial b_{i}$ for
all $1<i\le n$. Then, for each $k$, the closed curve $\gamma$ given
by the concatenation of $a_{k}$ and $b_{k}$ gives an element of
$\pi_{1}(D-\{p_{1}\},p_{k})\simeq\mathbb{Z}$. If the homotopy class
of $[\gamma]$ is $s$ times the generator, then we can apply $n|s|$
half-arc slides on $b_{k}$ so that $a_{k}$ is homotopic to $b_{k}$
in $D-\{p_{1}\}$. Thus we can apply this process for each $1<k\le n$,
so that $a_{k}$ is homotopic to $b_{k}$ rel endpoints in $D-\{p_{1}\}$.

Now assume that, for any $i,j$ with $1<i,j\le n$, $a_{i}$ and $b_{i}$
intersect transversely, and denote the total number of intersection
points between a-half-arcs and b-half-arcs by $N$. We can find a
disk $D$ which is innermost, i.e. no half-arcs intersect its interior.
Then, we can apply an isotopy along the disk to remove a pair of intersection
points; the number of remaining intersection points in $N-2$. Therefore,
by induction on $N$, we see that, after a sequence of isotopies,
we may assume that $a_{i}$ and $b_{i}$ cobound a disk $D_{i}$ for
each $1<i\le n$, and the disks $D_{i}$ are pairwise disjoint. Then,
we can isotope the half-arcs along the disks $D_{i}$ to isotope $a_{i}$
to $b_{i}$.
\end{proof}
Now we will prove invariance of $c_{\mathbb{Z}_{2}}(\xi_{K})$ with
respect to half-arc slides. Lifting the whole picture to $S$ shows
that, in the branched double cover, a half-arc slide corresponds to
an arcslide in the sense of \cite{HKM-HF}, which then corresponds
to an $\alpha$-handleslide followed by a $\beta$-handleslide. Hence,
if $\{a_{i}\}$ is obtained by performing an arcslide to $\{\tilde{a}_{i}\}$,
and $\boldsymbol{\alpha},\boldsymbol{\beta},\tilde{\boldsymbol{\alpha}},\tilde{\boldsymbol{\beta}}$
are the associated $\alpha$- and $\beta$-curves on the invariant
Heegaard surface $\Sigma$, then we have the following quasi-isormorphism,
which is induced by a composition of an equivariant triangle map for
an $\alpha$-handleslide followed by an equivariant triangle map for
a $\beta$-handleslide:
\[
\widetilde{CF}_{\mathbb{Z}_{2}}(\Sigma,\tilde{\boldsymbol{\alpha}},\tilde{\boldsymbol{\beta}})\xrightarrow{\sim}\widetilde{CF}_{\mathbb{Z}_{2}}(\Sigma,\boldsymbol{\alpha},\boldsymbol{\beta}).
\]
 Since this quasi-isomorphism is clearly $\mathbb{Z}_{2}$-equivariant,
we get the following induced quasi-isomorphism between equivariant
Floer cochain complexes.
\begin{equation}
\widehat{CF}_{\mathbb{Z}_{2}}(\Sigma,\tilde{\boldsymbol{\alpha}},\tilde{\boldsymbol{\beta}})\xrightarrow{\sim}\widehat{CF}_{\mathbb{Z}_{2}}(\Sigma,\boldsymbol{\alpha},\boldsymbol{\beta}).\label{eq:2eqtrimaps}
\end{equation}
\begin{thm}
\label{thm:half-arcslide}The map (\ref{eq:2eqtrimaps}) sends $EH_{\mathbb{Z}_{2}}(\xi_{K})$
to $EH_{\mathbb{Z}_{2}}(\xi_{K})$.
\end{thm}

\begin{proof}
Note that performing a half-arc slide to a half-arc basis corresponds
to performing an arcslide to the induced arc-basis in the branched
double cover. Thus the Heegaard triple-diagrams we get are the same
as the diagrams which arise in the proof of the invariance of contact
classes under arcslides, as in \cite{HKM-HF}. Since all holomorphic
triangles involved are small(see the proof of Lemma 3.5 in \cite{HKM-HF}),
all holomorphic triangles we count here have Maslov index $0$, and
their moduli spaces consist of a single point by Riemann mapping theorem.
Therefore we deduce that $EH_{\mathbb{Z}_{2}}(\xi_{K})$ is mapped
to $EH_{\mathbb{Z}_{2}}(\xi_{K})$.
\end{proof}
The invariance under perturbations of almost complex structures is
proved similarly.
\begin{thm}
\label{thm:acscontinuation}The map induced by changing the choice
of almost complex structures sends $EH_{\mathbb{Z}_{2}}(\xi_{K})$
to $EH_{\mathbb{Z}_{2}}(\xi_{K})$.
\end{thm}

\begin{proof}
The induced map, which is defined in the proof of Proposition 3.23
of \cite{eqv-Floer}, counts holomorphic disks going towards $EH$
with Maslov index at most 0. By Theorem \ref{thm:eqvtrans}, we can
choose cylindrical complex structures of the form $\text{Sym}^{g}(\tilde{\mathfrak{j}})$
to compute equivariant Floer cohomology. However, by the argument
of Section 3.1 in \cite{HKM-HF}, such disks must intersect the basepoint.
This completes the proof.
\end{proof}
\begin{cor}
\label{cor:isotopy-inv}The map induced by an isotopy, from a half-arc
basis $\mathcal{A}_{0}$ to another basis $\mathcal{A}$, sends $EH_{\mathbb{Z}_{2}}(\xi_{K},\mathcal{A})$
to $EH_{\mathbb{Z}_{2}}(\xi_{K},\mathcal{A}_{0})$.
\end{cor}

\begin{proof}
An isotopy of half-arc basis can be replaced by a 1-parameter family
of self-diffeomorphisms of $D^{2}$, starting from the identity. Such
a family induces a 1-parameter family of cylindrical complex structures
which we use to compute equivariant Floer cohomology. By Theorem \ref{thm:acscontinuation},
we see that the the induced isomorphism maps $EH_{\mathbb{Z}_{2}}(\xi_{K},\mathcal{A})$
to $EH_{\mathbb{Z}_{2}}(\xi_{K},\mathcal{A}_{0})$.
\end{proof}
\begin{thm}
\label{thm:isotopyinv}The map induced by an isotopy of the monodromy
$h$ sends $EH_{\mathbb{Z}_{2}}(\xi_{K})$ to $EH_{\mathbb{Z}_{2}}(\xi_{K})$.
\end{thm}

\begin{proof}
Let $\{h_{t}\}$ be an isotopy of monodromy functions. As in Theorem
7.3 of \cite{OSz-original}, we can reduce to the case where the isotopy
$\{\tilde{h}_{t}\}$ of self-diffeomorphisms of $\Sigma$ is a $\mathbb{Z}_{2}$-equivariant
Hamiltonian isotopy. Then we can apply the proof of Lemma 3.3 in \cite{HKM-HF}
to deduce that the equivariant isotopy map sends $EH_{\mathbb{Z}_{2}}(\xi_{K})$
to $EH_{\mathbb{Z}_{2}}(\xi_{K})$.
\end{proof}
Thus we proved the invariance under the choice of Floer-theoretic
auxiliary data, so it remains to prove the invariance under a basepoint
change and positive braid stabilization. Before proving the invariance
under positive braid stabilization. we prove the invariance under
the choice of a basepoint and stabilizations.
\begin{thm}
The map induced by changing the (invariant) basepoint sends $EH_{\mathbb{Z}_{2}}(\xi_{K})$
to $EH_{\mathbb{Z}_{2}}(\xi_{K})$.
\end{thm}

\begin{proof}
According to the definition in the previous section, we choose the
basepoint $z$ to be one of the points $p_{1},\cdots,p_{n}$, which
form the fixed locus of the $\mathbb{Z}_{2}$-action. If we choose
another such basepoint $z^{\prime}$, then since our transverse braid
forms a knot, there exists a positive integer $k$ satisfying $h^{k}(z)=z^{\prime}$.
Now applying the self-diffeomorphism $\tilde{h}^{k}$ to the open
book diagram $(\Sigma,\{a_{i}\},\tilde{h},z)$ gives $(\Sigma,\{\tilde{h}^{k}(a_{i})\},\tilde{h}^{k}\tilde{h}\tilde{h}^{-k}=\tilde{h},\tilde{h}^{k}(z)=z^{\prime})$.
But since we can always change the half-arc basis $\{h^{k}(a_{i}^{0})\}$
back to $\{a_{i}^{0}\}$ via half-arc slides, we can change the arc
basis $\{\tilde{h}^{k}(a_{i})\}$ back to $\{a_{i}\}$ via arcslides,
in the same manner. Since the maps induced by arcslides and diffeomorphisms
preserve $EH_{\mathbb{Z}_{2}}^{\ast}$, the theorem follows.
\end{proof}
\begin{thm}
\label{thm:stab}The map induced by a positive braid stabilization
sends $EH_{\mathbb{Z}_{2}}(\xi_{K})$ to $EH_{\mathbb{Z}_{2}}(\xi_{K})$.
\end{thm}

\begin{proof}
The induced map between $CF_{\mathbb{Z}_{2}}^{\ast}$ is induced by
taking RHom at the following $\mathbb{Z}_{2}$-equivariant quasi-isomorphism
of chain complexes(see the proof of Theorem 1.24 in \cite{eqv-Floer}
for details):
\[
\widehat{CF}(\Sigma(B))\rightarrow\widehat{CF}(\Sigma(B\coprod\text{unknot)})\rightarrow\widehat{CF}(\Sigma(B_{+})),
\]
 where $B$ is the original braid and $B_{+}$ is its positive stabilization.
Dualizing this gives 
\[
\widehat{CF}^{\ast}(\Sigma(B_{+}))\rightarrow\widehat{CF}^{\ast}(\Sigma(B\coprod\text{unknot}))\rightarrow\widehat{CF}^{\ast}(\Sigma(B)).
\]
 The second map preserves $EH$ by its definition. The first map is
the saddle map induced by a Legendrian $(-1)$-surgery along a lift
$c$ of a small Legendrian arc connecting $B$ with a trivial braid.
The Heegaard triple diagram for the saddle is drawn in Figure \ref{fig:5.4}.
Then, by the convenient placement of the basepoint, we see that all
holomorphic triangles connecting $\mathbf{x}$ and $\Theta$ are small.
Therefore the induced isomorphism between $CF_{\mathbb{Z}_{2}}^{\ast}$
preserves $EH_{\mathbb{Z}_{2}}$.
\end{proof}
\begin{figure}[tbph]
\resizebox{.7\textwidth}{!}{\includegraphics{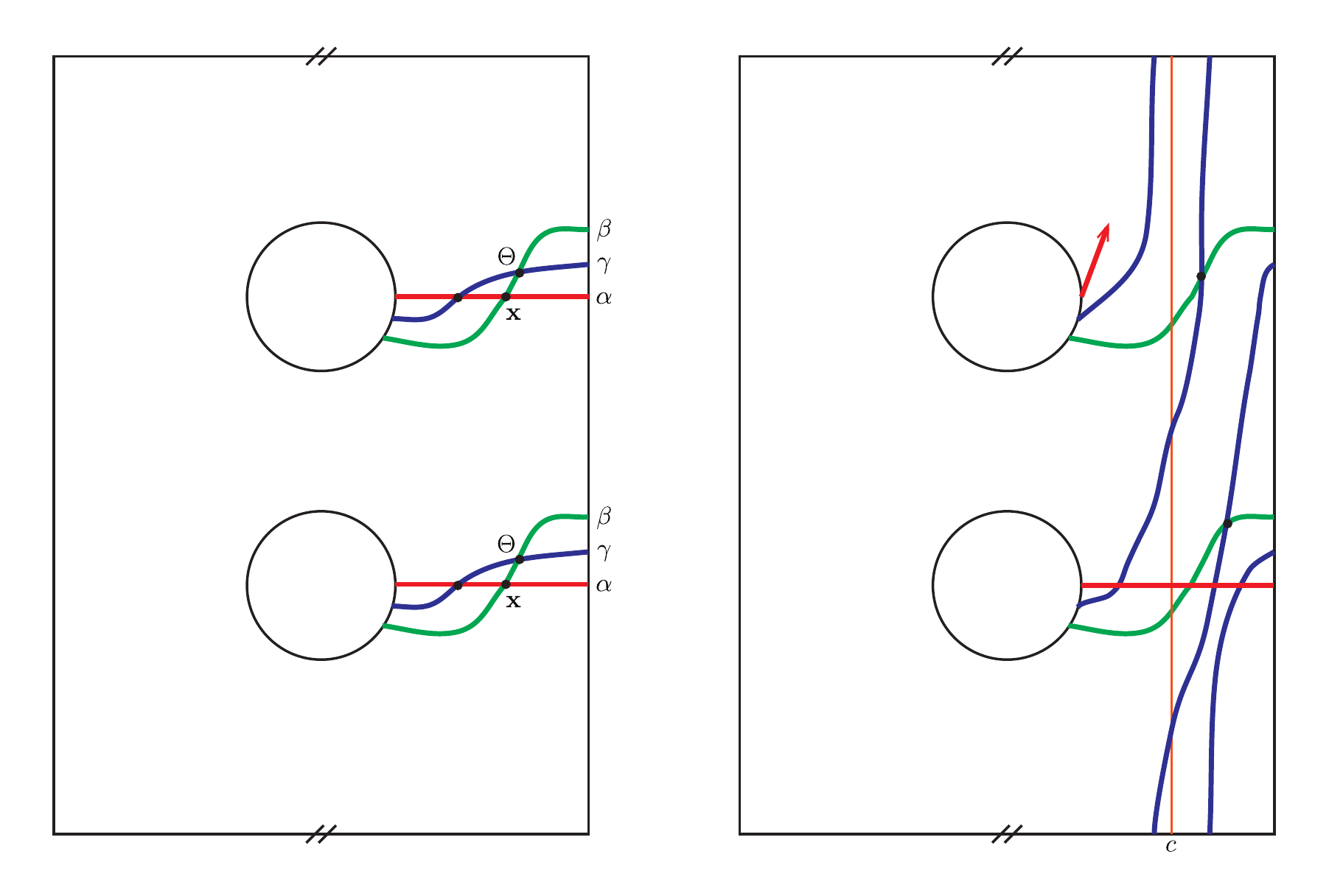}} \caption{\label{fig:5.4}The induced Heegaard triple-diagram.}
\end{figure}

Combining these invariance theorems, we get the complete invariance
of $EH_{\mathbb{Z}_{2}}(\xi_{K})$ and its cohomology class $c_{\mathbb{Z}_{2}}(\xi_{K})$.
\begin{thm}
The cohomology class $c_{\mathbb{Z}_{2}}(\xi_{K})\in\widehat{HF}_{\mathbb{Z}_{2}}(\Sigma(K))$
depends only on the transverse isotopy class of the transverse knot
$K$.
\end{thm}

\begin{proof}
By Theorem \ref{thm:acscontinuation}, the class $c_{\mathbb{Z}_{2}}(\xi_{K})$
is independent of the choice of almost complex structures. From Theorem
\ref{thm:half-arcslide}, Theorem \ref{thm:isotopyinv} and Corollary
\ref{cor:isotopy-inv}, we see that $c_{\mathbb{Z}_{2}}(\xi_{K})$
is invariant under isotopy and half-arc slide. Thus $c_{\mathbb{Z}_{2}}(\xi_{K})$
does not depend on the choice of half-arc basis by Proposition \ref{prop:halfarc},
which means that it only depends on the choice of a transverse braid
representative of the given transverse knot $K$. However, Theorem
\ref{thm:stab}tells us that $c_{\mathbb{Z}_{2}}(\xi_{K})$ is also
invariant under positive braid stabilizations. Therefore, by Theorem
\ref{thm:O-S}, we deduce that $c_{\mathbb{Z}_{2}}(\xi_{K})$ is an
invariant of the transverse isotopy class of $K$.
\end{proof}
\begin{defn}
The class $c_{\mathbb{Z}_{2}}(\xi_{K})$, which is an invariant of
the transverse isotopy class of $K$ in $(S^{3},\xi_{std})$, is called
the equivariant contact class of $(\Sigma(K),\xi_{K})$.
\end{defn}

\subsection*{Equivariant contact classes of transverse links}

When we work with a multi-component transverse link $L$, the same
argument can be applied to establish the existence and the invariance
of equivariant contact classes. However, we have a small issue with
the choice of a basepoint; the equivariant contact class still depends
on the component of $L$ in which the basepoint lies. Hence, what
we get is a cohomology class 
\[
c_{\mathbb{Z}_{2}}(\xi_{L},z)\in\widehat{HF}_{\mathbb{Z}_{2}}(\Sigma(L),z),
\]
 which depends on the component of $L$ in which $z$ lies. Writing
the basepoint $z$ explicitly will be useful in the next section,
where we deal with symplectic functoriality.

\section{Naturality and functoriality of $\widehat{HF}_{\mathbb{Z}_{2}}(\Sigma(K),p)$
when $K$ is a knot}

In this section, we will prove that the equivariant Floer cohomology
$\widehat{HF}_{\mathbb{Z}_{2}}(\Sigma(L),p)$ is well-defined up to
natural isomorphism, in the sense of \cite{Juhasz-naturality}, so
that it admits a natural action of the mapping class group $\text{MCG}(S^{3},L)=\pi_{0}\text{Diff}^{+}(S^{3},L)$.

Recall that any two bridge diagrams of a given based link are related
by isotopies, handleslides, and (de)stabilizations, applied to arcs
which do not contain the basepoint.
\begin{defn}
Let $\{A_{i}\},\{B_{i}\}$ denote the A- and B-arcs of a bridge diagram
of a based link $(L,p)$. A basic move on $L$ is an isotopy, a handleslide,
or a (de)stabilization applied to either a single A-arc or a single
B-arc, which does not contain $p$. An A(B)-equivalence is an isotopy
or a handleslide applied to a single A(B)-arc which does not contain
$p$.
\end{defn}

We can easily point out some naturally arising commutative triangles,
squares, and hexagons, consisting of basic moves. Those diagrams are
described below. Please note that, by an A(B)-arc, we mean an A(B)-arc
of a given bridge diagram of a given based link, which does not contain
the basepoint.

\subsubsection*{(a) A-equivalences and B-equivalences commute with each other}

Given a bridge diagram $D=(\{A_{i}\},\{B_{i}\})$ of a based link
$(L,p)$, we can consider applying an A-equivalence on $A_{i}$ and
a B-equivalence on $B_{j}$. Suppose that applying an A-equivalence
on $A_{i}$ of $D$ transforms it into $D_{a}$ and applying a B-equivalence
on $B_{j}$ of $D$ transforms it into $D_{b}$. Denote the result
of applying both equivalences on $D$ gives $D_{ab}$. Then we have
a following dintinguished square. 
\[
\xymatrix{D\ar[d]_{A}\ar[r]_{B} & D_{b}\ar[d]_{A}\\
D_{a}\ar[r]_{B} & D_{ab}
}
\]

\subsubsection*{(b) Commutative triangles of A(B)-equivalences}

Suppose that applying an A-equivalence on a bridge diagram $D$ gives
$D_{1}$ , applying another A-equivalence on $D_{1}$ gives $D_{2}$,
and there exists an A-equivalence which transforms $D$ into $D_{2}$.
Suppose further that, if two of the three A-equivalences are handleslides,
then they are handleslides along the same A-arc. Then we get a distinguished
triangle. The same thing also holds for B-equivalences.
\[
\xymatrix{D\ar[r]_{A}\ar[dr]_{A} & D_{1}\ar[d]_{A}\\
 & D_{2}
}
\]

\subsubsection*{(c) Handleslides on different arcs commute}

Suppose that applying a handleslide on an A-arc $A_{i}$ of a bridge
diagram $D$ gives $D_{i}$, applying a handleslide on an A-arc $A_{j}$
gives $D_{j}$, and applying both on $D$ gives $D_{ij}$. If $i\ne j$,
then we have a distinguished square. The same thing also holds for
B-equivalences.
\[
\xymatrix{D\ar[r]_{A}\ar[d]_{A} & D_{i}\ar[d]_{A}\\
D_{j}\ar[r]_{A} & D_{ij}
}
\]

\subsubsection*{(d) Commutative hexagon of handleslides}

Suppose that there are three A-arcs $A_{i},A_{j},A_{k}$ lying close
to each other in a bridge diagram $D$, as in Figure \ref{fig:6.1}.
Then we have two choices when handlesliding $A_{i}$ and $A_{j}$
over $A_{k}$ to reach Figure \ref{fig:6.2}; we can either move $A_{j}$
over $A_{k}$ first and then move $A_{i}$ over $A_{k}$ and $A_{j}$,
or move $A_{i}$ over $A_{j}$ and $A_{k}$ first and then move $A_{j}$
over $A_{k}$. 
\begin{figure}[tbph]
\resizebox{.7\textwidth}{!}{\includegraphics{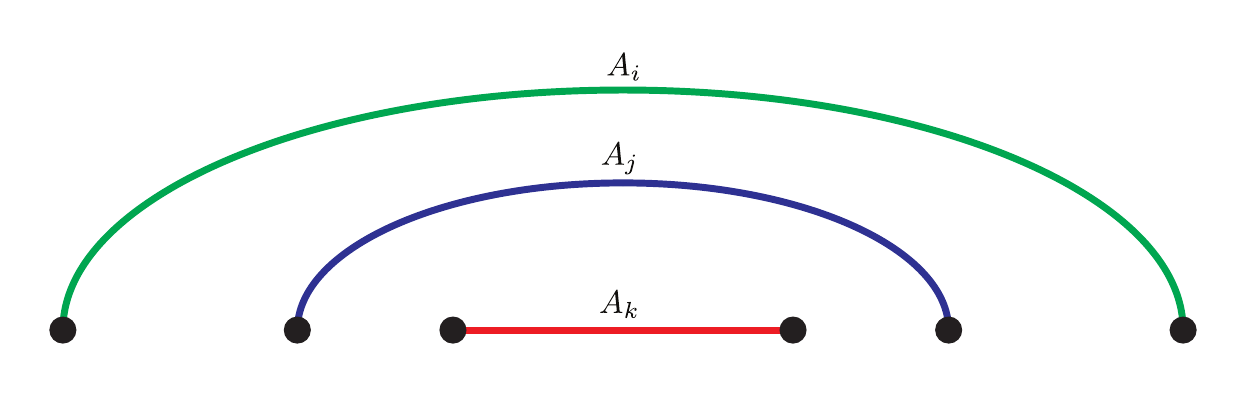}} \caption{\label{fig:6.1}Three A-arcs lying close to each other}
\end{figure}
\begin{figure}[tbph]
\resizebox{.7\textwidth}{!}{\includegraphics{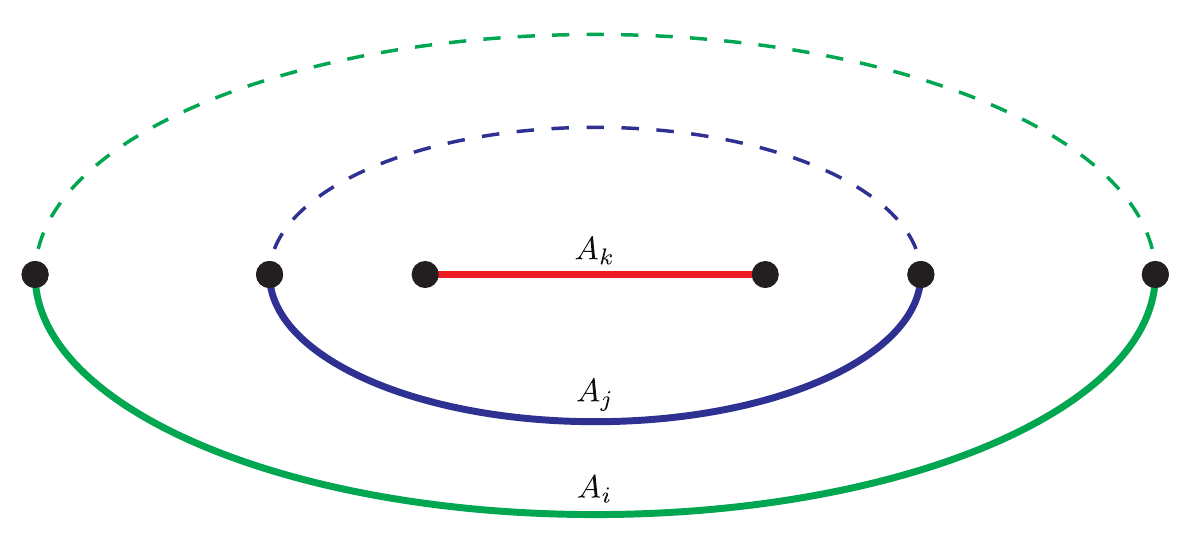}} \caption{\label{fig:6.2}Three A-arcs, after three handleslides}
\end{figure}

This gives a distinguished hexagon, and the same thing holds for B-arcs.
Here, $A_{i}/A_{j}$ denotes the handleslide of $A_{i}$ over $A_{j}$.
\[
\xymatrix{ & D_{1}\ar[r]_{A_{i}/A_{k}} & D_{2}\ar[dr]_{A_{i}/A_{j}}\\
D\ar[dr]_{A_{i}/A_{j}}\ar[ur]_{A_{j}/A_{k}} &  &  & D_{3}\\
 & D_{1}^{\prime}\ar[r]_{A_{i}/A_{k}} & D_{2}^{\prime}\ar[ur]_{A_{j}/A_{k}}
}
\]

\subsubsection*{(e) A(B)-equivalences commute with stabilizations}

Suppose that applying an A(B)-equivalence on an A(B)-arc $A_{i}$($B_{j}$)
in the bridge diagram $D$ gives $D_{1}$, and applying a stabilization
on an arc of $D$, which does not contain the basepoint and is different
from $A_{i}$, gives $D^{s}$, and applying a stabilization on the
corresponding arc of $D_{1}$ gives $D_{1}^{s}$. Then $D_{1}^{s}$
can be obtained from $D^{s}$ by an A(B)-equivalence, and so we get
a distinguished square.
\[
\xymatrix{D\ar[r]_{A}\ar[d]_{s} & D_{1}\ar[d]_{s}\\
D^{s}\ar[r]_{A} & D_{1}^{s}
}
\]

\subsubsection*{(f) Stabilizations applied on different arcs commute }

Consider applying stabilization on two different arcs of a bridge
diagram $D$. There are two possible orders, which give us a distinguished
square.
\[
\xymatrix{D\ar[d]_{s}\ar[r]_{s} & D_{2}\ar[d]_{s}\\
D_{1}\ar[r]_{s} & D_{3}
}
\]

\subsubsection*{(g) Commutative triangle of two stabilizations and an isotopy}

Given a point $x\in\partial A_{i}\cap\partial B_{j}$ of a bridge
diagram $D=(\{A_{i}\},\{B_{j}\})$, suppose that applying a stabilization
on $A_{i}$ near $x$ gives $D_{1}$ and applying a stabilization
on $B_{j}$ near $x$ gives $D_{2}$. Then $D_{1}$ and $D_{2}$ differ
by an isotopy. So we get a distinguished triangle.
\[
\xymatrix{D\ar[r]_{s}\ar[d]_{s} & D_{2}\\
D_{1}\ar[ur]_{\text{isotopy}}
}
\]

Given a bridge diagram $D$ of a based link $(L,p)$, whose set of
endpoints of arcs is given by $S\subset\Sigma$, we can also define
maps on the equivariant Floer cohomology, associated to diffeomorphisms
$\phi\in\text{Diff}^{+}(\Sigma)$ which fix $p$ and $S$ pointwise
in a natural way. Then we get few more types of distinguished diagrams.

\subsubsection*{(h) Diffeomorphism and basic moves commute}

\[
\xymatrix{D\ar[r]_{\text{diffeo}}\ar[d]_{\text{basic}} & \phi(D)\ar[d]^{\text{basic}}\\
D_{1}\ar[r]_{\text{diffeo}} & \phi(D^{\prime})
}
\]

Here, we use diffeomorphisms $\phi\in\text{Diff}^{+}(\Sigma,S,p)$.

\subsubsection*{(i) Diffeomorphism can be undone by basic moves}

Suppose that a diffeomorphism $\phi\in\text{Diff}^{+}(\Sigma,S,p)$
maps a bridge diagram $D$ of $(L,p)$ to $\phi(D)$. Since they represent
the same based link, we can obtain $\phi(D)$ from $D$ by a sequence
of basic moves. So we get a distinguished diagram.
\[
\xymatrix{D\ar[rrrr]_{\text{diffeo}}\ar[dr]_{\text{basic}} &  &  &  & \phi(D)\\
 & D_{1}\ar[r]_{\text{basic}} & \cdots\ar[r]_{\text{basic}} & D_{k}\ar[ur]_{\text{basic}}
}
\]
\begin{rem*}
In the paper \cite{eqv-Floer}, the authors define stabilizations
on an arc only when the stabilization is applied near an endpoint
of an arc and that endpoint is very close to the basepoint. However,
we can similarly define stabilization maps even when the point where
stabilizations occur is not close to the basepoint and the endpoints
of arcs, by taking a family of almost complex structures of type $\text{Sym}^{g}(\mathfrak{j})$
as in Theorem \ref{thm:eqvtrans}, where the $1$-parameter family
$\mathfrak{j}$ is split and has long neck near the point at which
a stabilization is performed. Once we prove that the equivariant Floer
cohomology satisfies the commutative squares of type (e), we can immediately
deduce that such maps are indeed isomorphisms.
\end{rem*}
\begin{lem}
\label{lemma-isotopy-tri}Let $(S,\boldsymbol{\alpha},\boldsymbol{\beta})$
and $(S,\boldsymbol{\beta},\boldsymbol{\gamma})$ be $\mathbb{Z}_{2}$-Heegaard
diagrams given by taking branched double covers of bridge diagrams
of links drawn on a sphere $\Sigma$, and suppose that $\{\boldsymbol{\alpha}_{t}\}_{t\in[0,1]}$
is a $\mathbb{Z}_{2}$-invariant isotopy so that $\boldsymbol{\alpha}_{1}=\boldsymbol{\alpha}$.
Suppose further that $\boldsymbol{\alpha}_{t},\boldsymbol{\beta},\boldsymbol{\gamma}$
are pairwise transverse for $t=0$ and $t=1$. Then, for each $\mathbb{Z}_{2}$-invariant
cycle $\theta_{\beta,\gamma}\in\widehat{CF}(S,\boldsymbol{\beta},\boldsymbol{\gamma})$
and each element $x_{\alpha_{0},\beta}\in H_{\ast}(\widetilde{CF}_{\mathbb{Z}_{2}}(S,\boldsymbol{\alpha}_{0},\boldsymbol{\beta})\otimes_{\mathbb{F}_{2}[\mathbb{Z}_{2}]}\mathbb{F}_{2})$,
we have 
\[
\Gamma_{\{\boldsymbol{\alpha}_{t}\},\gamma}(\hat{f}_{\alpha_{0},\beta,\gamma}(x_{\alpha_{0},\beta}\otimes\theta_{\beta,\gamma}))=\hat{f}_{\alpha,\beta,\gamma}(\Gamma_{\{\boldsymbol{\alpha}_{t}\},\beta}(x_{\alpha_{0},\beta})\otimes\theta_{\beta,\gamma}),
\]
 where $\hat{f}$ denote the equivariant triangle maps, as defined
in the proof of Proposition 3.25 of \cite{eqv-Floer}, and $\Gamma$
denote the equivariant isotopy map.
\end{lem}

\begin{proof}
Recall that we used a topological category $\mathscr{D}$ when constructing
equivariant triangle map. Denote the three edges of the triangle $\triangle$
by $e_{\alpha},e_{\beta},e_{\gamma}$, and parametrize the edge $e_{\alpha}$
by a function $E_{\alpha}\,:\,\mathbb{R}\rightarrow\triangle$. Given
a parametrized family $\tilde{J}\,:\,[0,1]^{\ell}\rightarrow\mathcal{J}_{\triangle}$,
consider the moduli spaces
\begin{align*}
\mathcal{M}_{\tau} & (\tilde{J})=\bigcup_{t\in[0,1]^{\ell}}\left\{ u\,:\,\triangle\rightarrow\text{Sym}^{g}(\Sigma)\left|\begin{array}{c}
u\circ e_{\alpha}(t)\in\mathbb{T}_{\alpha_{t+\tau}},\,u(e_{\beta})\subset\mathbb{T}_{\beta},\,u(e_{\gamma})\subset\mathbb{T}_{\gamma}\\
u\text{ is }J(t)-\text{holomorphic}
\end{array}\right.\right\} ,\\
\mathcal{M}(\tilde{J}) & =\bigcup_{\tau\in\mathbb{R}}\mathcal{M}_{\tau}(\tilde{J}),
\end{align*}
 and split them into homotopy classes $\phi\in\pi_{2}^{\{\mathbb{T}_{\alpha_{t}}\},\mathbb{T}_{\beta},\mathbb{T}_{\gamma}}(\mathbf{x},\mathbf{y},\mathbf{z})$
for $\mathbf{x}\in\mathbb{T}_{\alpha_{0}}\cap\mathbb{T}_{\beta},\mathbf{y\in\mathbb{T}_{\beta}\cap\mathbb{T}_{\gamma}},\mathbf{z}\in\mathbb{T}_{\gamma}\cap\mathbb{T}_{\alpha}$,
as follows.
\[
\mathcal{M}_{\tau}(\tilde{J})=\bigcup_{\phi}\mathcal{M}_{\tau}(\phi;\tilde{J}),\,\,\mathcal{M}(\tilde{J})=\bigcup_{\phi}\mathcal{M}(\phi;\tilde{J})
\]
 Then we define the map $G(\tilde{J})\,:\,\widehat{CF}(S,\boldsymbol{\alpha}_{0},\boldsymbol{\beta})\rightarrow\widehat{CF}(S,\boldsymbol{\alpha},\boldsymbol{\gamma})$
as follows:
\[
G(\tilde{J})(\mathbf{x})=\sum_{\mathbf{z}\in\mathbb{T}_{\gamma}\cap\mathbb{T}_{\alpha}}\sum_{\mathbf{y}\in\theta_{\beta,\gamma}}\sum_{\phi\in\pi_{2}^{\{\mathbb{T}_{\alpha_{t}}\},\mathbb{T}_{\beta},\mathbb{T}_{\gamma}}(\mathbf{x},\mathbf{y},\mathbf{z}),\,\mu(\phi)=-1-\ell}\left|\mathcal{M}(\phi;\tilde{J})\right|\cdot\mathbf{z}.
\]
 As in the proof of Proposition 3.25 of \cite{eqv-Floer}, the function
$G$ induces a map 
\[
F_{G}\,:\,\widetilde{CF}_{\mathbb{Z}_{2}}(S,\boldsymbol{\alpha}_{0},\boldsymbol{\beta})\rightarrow\widetilde{CF}_{\mathbb{Z}_{2}}(S,\boldsymbol{\alpha},\boldsymbol{\gamma}).
\]

There are three types of ends in the moduli space $\mathcal{M}(\phi;\tilde{J})$
when $\mu(\phi)=-\ell$. The first type is the degeneration of the
almost complex structure to the boundary $\tilde{J}|_{\partial[0,1]^{\ell}}$,
which does not contribute to the total count of ends; since we are
using $\mathbb{Z}_{2}$-equivariant diagrams of almost complex structures,
the count of such ends must be even and hence zero in $\mathbb{F}_{2}$.
The second and third types are the ones in the proof of Proposition
8.14 in \cite{OSz-original}, which contribute to 
\[
\Gamma_{\{\boldsymbol{\alpha}_{t}\},\gamma}(\hat{f}_{\alpha_{0},\beta,\gamma}(\mathbf{x}\otimes\theta_{\beta,\gamma}))+\hat{f}_{\alpha,\beta,\gamma}(\Gamma_{\{\boldsymbol{\alpha}_{t}\},\beta}(\mathbf{x})\otimes\theta_{\beta,\gamma})
\]
 and $\partial F_{G}(\mathbf{x})+F_{G}(\partial\mathbf{x})$, respectively.
Therefore we deduce that 
\[
\Gamma_{\{\boldsymbol{\alpha}_{t}\},\gamma}(\hat{f}_{\alpha_{0},\beta,\gamma}(x_{\alpha_{0},\beta}\otimes\theta_{\beta,\gamma}))+\hat{f}_{\alpha,\beta,\gamma}(\Gamma_{\{\boldsymbol{\alpha}_{t}\},\beta}(x_{\alpha_{0},\beta})\otimes\theta_{\beta,\gamma})=\partial F_{G}(x_{\alpha_{0},\beta})+F_{G}(\partial x_{\alpha_{0},\beta}).
\]
 Since $\theta_{\beta,\gamma}$ is $\mathbb{Z}_{2}$-invariant, the
map $F_{G}$ is also $\mathbb{Z}_{2}$-invariant. Therefore we get
the desired result.
\end{proof}
\begin{lem}
\label{lem:lemma-isotopy-concat}Let $(S,\boldsymbol{\alpha},\boldsymbol{\beta})$
be a $\mathbb{Z}_{2}$-Heegaard diagram given by taking branched double
cover of a bridge diagram of a link drawn on a sphere $\Sigma$, and
suppose that $\{\boldsymbol{\alpha}_{t}\}_{t\in[0,1]}$ is a $\mathbb{Z}_{2}$-invariant
isotopy so that $\boldsymbol{\alpha}_{1}=\boldsymbol{\alpha}$. Suppose
further that $\boldsymbol{\alpha}_{t},\boldsymbol{\beta},\boldsymbol{\gamma}$
are pairwise transverse for $t=0,\frac{1}{2},1$. Then, for each element
$x_{\alpha_{0},\beta}\in H_{\ast}(\widetilde{CF}_{\mathbb{Z}_{2}}(S,\boldsymbol{\alpha}_{0},\boldsymbol{\beta})\otimes_{\mathbb{F}_{2}[\mathbb{Z}_{2}]}\mathbb{F}_{2})$,
we have 
\[
\Gamma_{\{\boldsymbol{\alpha}_{t}\}_{t\in\left[\frac{1}{2},1\right]},\beta}(\Gamma_{\{\boldsymbol{\alpha}_{t}\}_{t\in\left[0,\frac{1}{2}\right]},\beta}(x_{\alpha_{0},\beta}))=\Gamma_{\{\boldsymbol{\alpha}_{t}\}_{t\in[0,1]},\beta}(x_{\alpha_{0},\beta}),
\]
 where $\Gamma$ denotes the equivariant isotopy map.
\end{lem}

\begin{proof}
As in the proof of \ref{lemma-isotopy-tri}, we can mimic the proof
of Theorem 7.3 of \cite{OSz-original} to make it work in the equivariant
setting.
\end{proof}
\begin{prop}
\label{prop:commutation-a-h}The equivariant Floer cohomology of based
links in $S^{3}$ makes the distinguished diagrams of type (a)-(h)
commutative.
\end{prop}

\begin{proof}
By equivariant transversality and the above lemma, we only have to
prove that the corresponding commutative diagrams of $\widehat{CF}$
groups are satisfied up to $\mathbb{Z}_{2}$-equivariant chain homotopies. 

For the distinguished diagrams of type (c), we already know that it
is satisfied on the $\widehat{CF}$ level, up to a chain homotopy.
The chain homotopy is given by counting holomorphic squares, after
a perturbation as in Figure 4 of \cite{eqv-Floer}. Here, we can always
perturb a given bridge diagram by an isotopy using Lemma \ref{lemma-isotopy-tri}.
After making such a perturbation, we have no constant triangle of
negative Maslov index, and thus the equivariant triangle map is induced
by the ordinary triangle map with respect to generic 1-parameter families
of almost complex structures. Since holomorphic squares contained
in the $\mathbb{Z}_{2}$-fixed locus are constant squares, and such
squares have Maslov index 0, they are not counted in the square map.
This implies, by the arguments of the proof of Proposition 5.13 in
\cite{quiver-Floer}, that a generic $\mathbb{Z}_{2}$-invariant 1-parameter
family of almost complex structures achieves transversality for squares
of Maslov index $-1$, which tells us that the holomorphic square
map is well-defined for generic $\mathbb{Z}_{2}$-invariant families.
Hence the given chain homotopy is $\mathbb{Z}_{2}$-equivariant, i.e.
the given square diagram commutes up to $\mathbb{Z}_{2}$-equivariant
chain homotopy on the $\widehat{CF}$ level. Therefore we get a commuting
square diagram of corresponding $\widehat{HF}_{\mathbb{Z}_{2}}$ groups.
The same argument can be used to prove the commutativity for distinguished
squares of type (a), (f), and (h). Also, by Theorem 2.14 of \cite{OSz-4-mfd-inv},
we can follow the proof of Lemma 2.15 in \cite{OSz-4-mfd-inv} to
show that distinguished squares of type (e) also commute.

For the distinguished diagrams of type (b), the A-equivalence from
$D$ to $D_{1}$ is given by evaluating the triangle map using a $\mathbb{Z}_{2}$-invariant
cocycle $\Theta_{D,D_{1}}$ which represent the top class, and similarly
consider $\mathbb{Z}_{2}$-invariant cocycles $\Theta_{D,D_{2}}$
and $\Theta_{D_{1},D_{2}}$. By the technique used to prove the commuativity
of distinguished diagrams of type (c), it suffices to prove that the
image of $\Theta_{D,D_{1}}\otimes\Theta_{D_{1},D_{2}}$ under the
triangle map is the same as the cocycle $\Theta_{D,D_{2}}$. Since
the image of $\Theta_{D,D_{1}}\otimes\Theta_{D_{1},D_{2}}$ must also
be a $\mathbb{Z}_{2}$-invariant cocycle which represents the top
class, the proof will be finished if $\Theta_{D,D_{2}}$ is the only
$\mathbb{Z}_{2}$-invariant cocycle which represents the top class.
Since the Heegaard diagram for an isotopy or a saddle obviously admits
a unique representative of its top class, we are done. The same argument
can be used to prove the commutativity of distinguished squares of
type (d) and (g).
\end{proof}
\begin{lem}
\label{lem:identity-diff}Given a bridge diagram $(\{A_{i}\},\{B_{i}\})$
of a based link $(L,p)$ in $S^{3}$, drawn on a sphere $\Sigma=S^{2}$,
let $\{p_{1},\cdots,p_{n}\}$ be the set of endpoints of arcs $A_{i}$
and $B_{i}$, which are not equal to the basepoint $p$. Given a 1-parameter
family of self-diffeomorphisms $\{\phi_{t}\}_{t\in[0,1]}$ of $\Sigma$,
such that $\phi_{0}=\text{id}_{\Sigma}$, each $\phi_{t}$ fixes $p$
pointwise and $\{p_{1},\cdots,p_{n}\}$ setwise, and the images of
A- and B-curves under $\phi_{1}$ intersect transversely with the
original A- and B-curves, consider the following two maps. First,
the isotopy map induced by the isotopy $\{\phi_{t}(A_{i})\}$ and
$\{\phi_{t}(B_{i})\}$ of A- and B-arcs:
\[
\Gamma\,:\,\widehat{HF}_{\mathbb{Z}_{2}}(\Sigma(L),p)\rightarrow\widehat{HF}_{\mathbb{Z}_{2}}(\Sigma(L),p).
\]
Next, the diffeomorphism map 
\[
\phi^{\ast}:\widehat{HF}_{\mathbb{Z}_{2}}(\Sigma(L),p)\rightarrow\widehat{HF}_{\mathbb{Z}_{2}}(\Sigma(L),p).
\]
 Then we have $\Gamma=\phi^{\ast}$.
\end{lem}

\begin{proof}
The argument used in the proof of Lemma 9.5 of \cite{Juhasz-naturality}
and Proposition 9.8 of \cite{OSz-original} directly generalizes to
the equivariant setting. Hence we see that the lemma holds when the
given isotopy $\{\phi_{t}\}$ is sufficiently small. Hence, by Lemma
\ref{lem:lemma-isotopy-concat}, we deduce that the lemma holds for
any isotopy.
\end{proof}
\begin{lem}
\label{lem:diff}Given a bridge diagram $(\{A_{i}\},\{B_{i}\})$ of
a based link $(L,p)$ in $S^{3}$, drawn on a sphere $\Sigma=S^{2}$,
let $\{p_{1},\cdots,p_{n}\}$ be the set of endpoints of arcs $A_{i}$
and $B_{i}$, which are not equal to the basepoint $p$. Given a diffeomorphism
$\phi$ of $\Sigma$ which fix $p$ pointwise and $\{p_{1},\cdots,p_{n}\}$
setwise, choose a sequence of basic moves from $(\{A_{i}\},\{B_{i}\})$
to $(\{\phi(A_{i})\},\{\phi(B_{i})\})$, and denote the induced map
between $\widehat{HF}_{\mathbb{Z}_{2}}$ as follows:
\[
T_{\phi}\,:\,\widehat{HF}_{\mathbb{Z}_{2}}(\Sigma(L),p)\rightarrow\widehat{HF}_{\mathbb{Z}_{2}}(\Sigma(L),p).
\]
 Similarly construct a map $T_{\phi^{-1}}$ by choosing a sequence
of basic moves from $(\{\phi(A_{i})\},\{\phi(B_{i})\})$ to $(\{A_{i}\},\{B_{i}\})$.
Then we have $T_{\phi}\circ T_{\phi^{-1}}=\text{id}$.
\end{lem}

\begin{proof}
By Lemma \ref{lem:identity-diff}, we can assume, without losing generality,
that the given sequences of basic moves do not contain isotopies.
Since the maps induced by A-equivalences and B-equivalences commute,
it suffices to prove that the composition of all A-equivalence maps
arising in $T_{\phi}$ and $T_{\phi^{-1}}$ is the identity, since
it will also imply the same thing for B-equivalence maps. 

The Heegaard diagram given by taking the branched double cover of
$(\Sigma,\{A_{i}\},\{A_{i}^{\prime}\})$ has unique $\mathbb{Z}_{2}$-invariant
cocycle whch represents the top class, where $A_{i}^{\prime}$ is
a slight perturbation of $A_{i}$ so that $A_{i}\cap A_{i}^{\prime}=\partial A_{i}$.
Therefore we can use the arguments in the proof of Proposition \ref{prop:commutation-a-h}
to conclude that the composition of all A-equivalence maps in $T_{\phi}\circ T_{\phi^{-1}}$
is the identity, and so $T_{\phi}\circ T_{\phi^{-1}}=\text{id}$.
\end{proof}
\begin{prop}
The equivariant Floer cohomology of based knots in $S^{3}$ satisfies
the commutative diagrams of type (i)
\end{prop}

\begin{proof}
For any positive integer $n$, the pure mapping class group of a disk
with $n$ punctures is given by the pure braid group on $n$ strands,
and taking its quotient by the Dehn twists along the boundary gives
the mapping class group of a sphere with $n+1$ punctures. Thus, if
we consider the standard generating set $\{T_{i}\}$ of the pure braid
group $B_{n}$, where $T_{i}$ denotes a positive twist of the $i$th
and the $i+1$th strand, the set $\{T_{i}^{2}\}$ normally generates
the pure braid group $PB_{n}$, and thus generates the group $\text{PMod}(S_{0,n+1})$.
Hence, by Lemma \ref{lem:diff}, given a bridge diagram of a based
knot $(K,p)$ on a sphere $\Sigma$, we only have to prove that the
commutativity diagrams of type (i) holds for $\widehat{HF}_{\mathbb{Z}_{2}}(\Sigma(K),p)$
only for the full Dehn twists along an A-arc or a B-arc, which does
not contain $p$, and any choice of a sequence of basic moves.

Choose such an A-arc $A_{i}$. The bridge diagram given by appling
a Dehn twist along $A_{i}$ is drawn in Figure \ref{fig:6.3}. Its
branched double cover admits a unique $\mathbb{Z}_{2}$-invariant
cycle which represents the top generator in homology. Thus we can
compute the composition of the maps associated to our choice of basic
moves by computing the equivariant triangle map for the Heegaard triple-diagram,
which is given by taking the branched double cover of the diagram
drawn in Figure \ref{fig:6.4}. Note that $A_{i}$ is assumed to not
intersect any B-arcs other than the two B-arcs adjacent to it, and
the basepoint is placed near the leftmost point in Figure \ref{fig:6.4};
this is possible because $\widehat{HF}_{\mathbb{Z}_{2}}$ satisfies
commutative diagrams of type (h).

We claim that the only nonconstant triangles, each of which consists
of a green arc, a blue arc, and a constant red arc, and involves at
least one of the two endpoints of $A_{i}$, are those shown in Figure
\ref{fig:6.5}. Let $T$ be such a triangle. Then, without loss of
generality, we may assume that it uses the blue arc and the green
arc connected to the leftmost point in Figure \ref{fig:6.4}. If $T$
uses the leftmost red arc instead, then by the assumption on the placement
of the basepoint, $T$ must intersect the basepoint, so this case
is impossible. Hence $T$ must use the constant red arc at that point.
Then, the triangle we get is the one shown in Figure \ref{fig:6.5}. 

Now, the other triangles are exactly the ones which arise when calculating
the triangle map for the triple-diagram in Figure \ref{fig:6.6},
except for the two shaded regions. Thus, the triangle map for the
triple-diagram in Figure \ref{fig:6.4} is the composition of the
triangle map for the triple-diagram in Figure \ref{fig:6.6}, followed
by the diffeomorphism map induced by the Dehn twist along $A_{i}$.

However, using the argument of Proposition 9.8 in \cite{OSz-original},
we see that the triangle map for Figure \ref{fig:6.6} agrees with
the diffeomorphism map, induced by a diffeomorphism which is isotopic
to the identity. Therefore we see that the triangle map must give
the same result as the diffeomorphism map.
\end{proof}
\begin{figure}[tbph]
\resizebox{.7\textwidth}{!}{\includegraphics{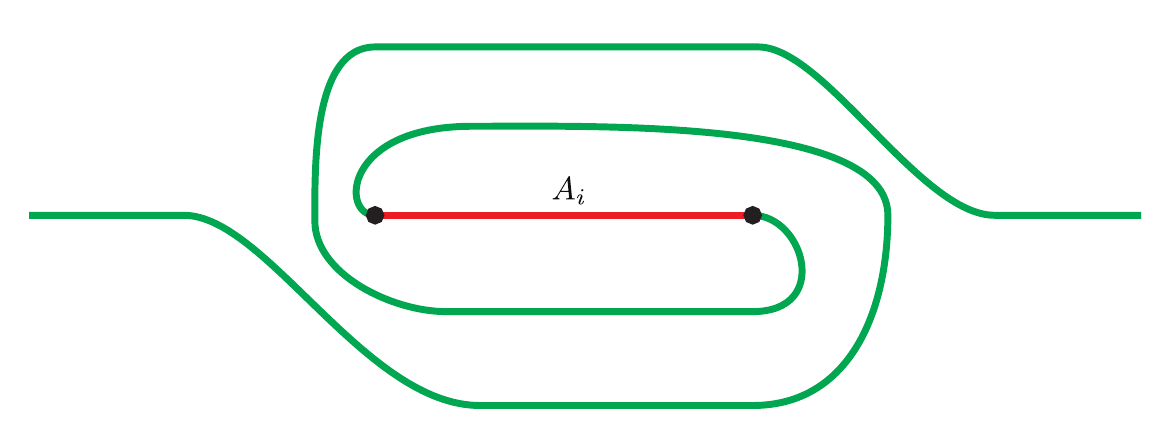}} \caption{\label{fig:6.3}The diagram after applying a Dehn twist along (the
boundary of a neighborhood of) $A_{i}$}
\end{figure}
\begin{figure}[tbph]
\resizebox{.7\textwidth}{!}{\includegraphics{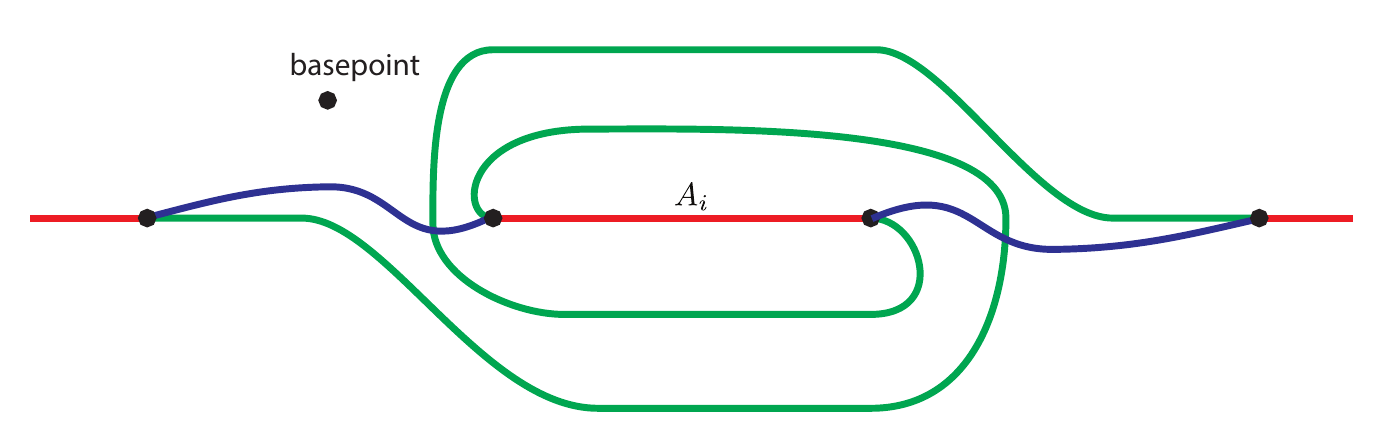}} \caption{\label{fig:6.4}The triple-diagram}
\end{figure}
\begin{figure}[tbph]
\resizebox{.7\textwidth}{!}{\includegraphics{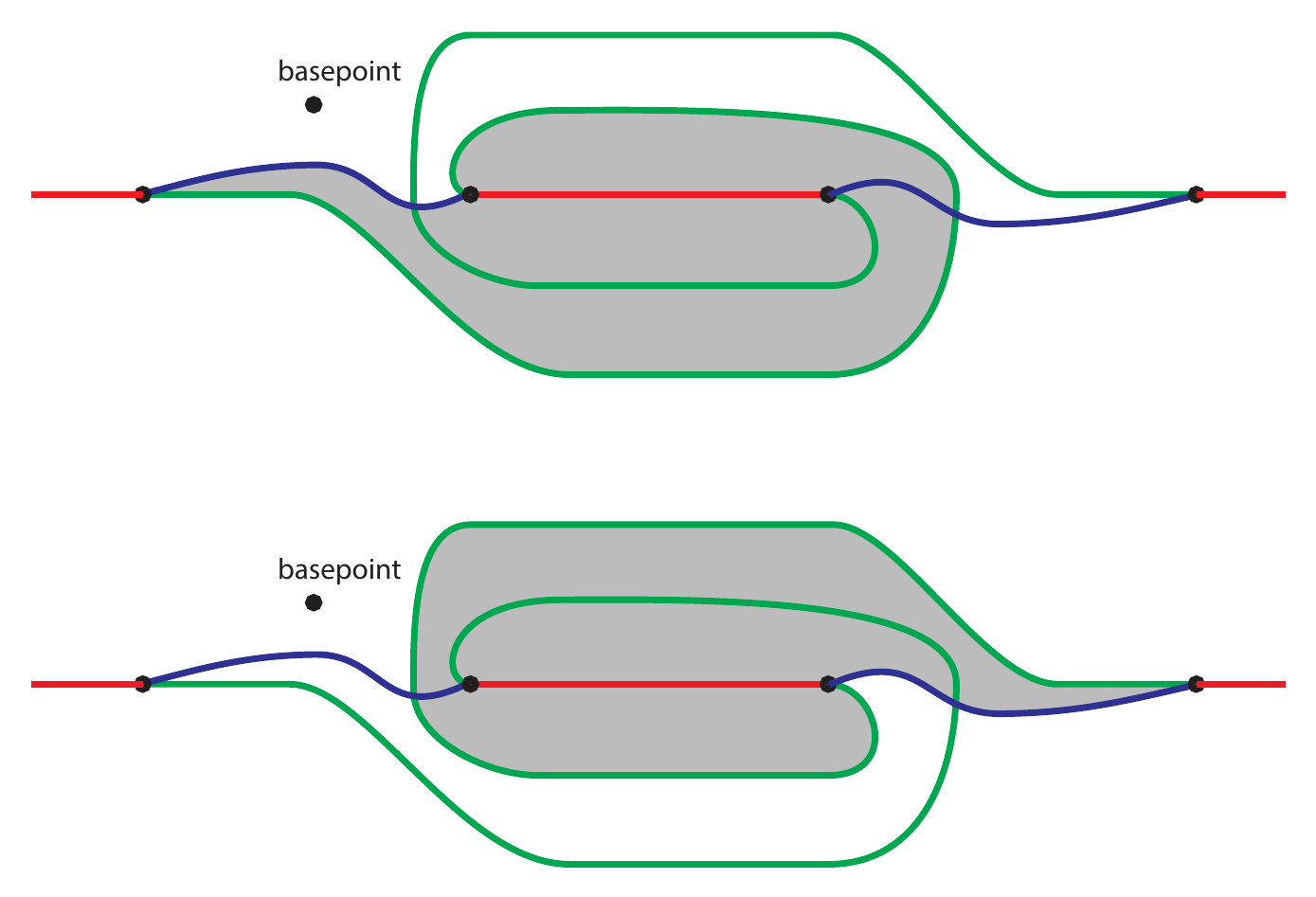}} \caption{\label{fig:6.5}The nonconstant triangles in (the branched double
cover of) the given triple-diagram}
\end{figure}
\begin{figure}[tbph]
\resizebox{.7\textwidth}{!}{\includegraphics{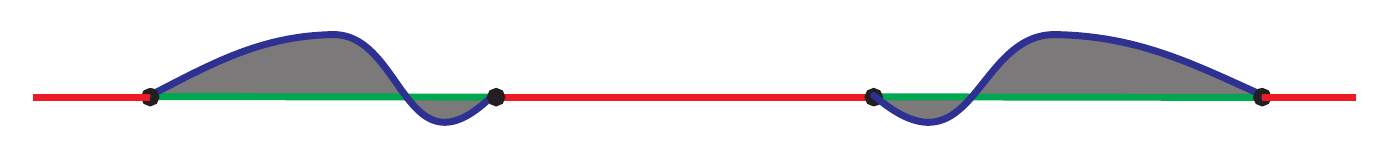}} \caption{\label{fig:6.6}The triple-diagram representing a small isotopy of
B-arcs}
\end{figure}

For a based link $(L,p)$ in $S^{3}$ and a genus 0 Heegaard surface
$\Sigma\subset S^{3}$, let $\mathcal{B}_{L,\Sigma,p}$ be the 2-dimensional
cell complex defined as follows.
\begin{itemize}
\item The 0-cells are bridge diagrams of $(L,p)$ on $\Sigma$,
\item The 1-cells are basic moves and diffeomorphism maps,
\item The 2-cells are commutative diagrams of type (a)-(i).
\end{itemize}
If we denote the space of parametrized based links in $S^{3}$ which
are isotopic to $(L,p)$ as $\text{Emb}_{L}\left(\coprod S^{1},S^{3}\right)$,
we have a canonical map 
\[
\mathcal{R}\,:\,\mathcal{B}_{L,\Sigma,p}\rightarrow\text{Emb}_{L}\left(\coprod S^{1},S^{3}\right).
\]
\begin{lem}
The map $\mathcal{R}$ is 1-connected.
\end{lem}

\begin{proof}
Let $x,y$ be the north and south pole of $S^{3}$, and choose a projection
function $p_{\Sigma}\,:\,S^{3}-\{x,y\}\rightarrow\Sigma$, together
with a height function $h_{\Sigma}\,:\,S^{3}-\{x,y\}\rightarrow\mathbb{R}$,
so that $p_{\Sigma}\times h_{\Sigma}$ is a diffeomorphism, $p_{\Sigma}|_{\Sigma}=\text{id}_{\Sigma}$,
and $h_{\Sigma}|_{\Sigma}=0$ . A generic point in $\text{Emb}_{L}\left(\coprod S^{1},S^{3}\right)$
is a based links in $S^{3}$, isotopic to $(L,p)$, which do not pass
through $x$ nor $y$, intersects $\Sigma$ transversely, no two points
$(x_{1},x_{2})$ on the link have the same image under $p_{\Sigma}$
if $x_{1}\in\Sigma$ or $x_{2}\in\Sigma$, and no three points on
the link have the same image under $p_{\Sigma}$. Such a link can
be canonically isotoped to a bridge position by moving it across $\Sigma$
so that, for each pair of points $(x_{1},x_{2})$ on the link which
satisfies $p_{\Sigma}(x_{1})=p_{\Sigma}(x_{2})$ and $h_{\Sigma}(x_{1})<h_{\Sigma}(x_{2})$,
we have 
\[
h_{\Sigma}(x_{1})<0<h_{\Sigma}(x_{2}),
\]
 while fixing the points in $L\cap\Sigma$. The codimension 1 points
are the links which satisfy one of the following cases.
\begin{lyxlist}{00.00.0000}
\item [{(1a)}] The link projects to the cusp $y^{2}=x^{3}$
\item [{(1b)}] Exactly two generic points $x_{1},x_{2}$ on the link have
the same image under $p_{\Sigma}$, and the projected imaged of the
segments of the link near $x_{1}$ and $x_{2}$ are tangent to each
other, where the order of tangency is 1
\item [{(1c)}] Exactly three generic points $x_{1},x_{2},x_{3}$ on the
link have the same image under $p_{\Sigma}$
\item [{(1d)}] The link is tangent to $\Sigma$ at a generic point $x_{\Sigma}$,
where the order of tangency is 1
\end{lyxlist}
The case (1a) corresponds to isotopies, (1b) and (1c) corresponds
to handleslides, and (1d) corresponds to stabilizations. Note that
a path of generic points, which does not pass through codimension
1 singularities, corresponds to isotopy maps or diffeomorphism maps.
The choice of a diffeomorphism map is unique up to distinguished squares
of type (i).

Now the codimension 2 points are given as follows.
\begin{lyxlist}{00.00.0000}
\item [{(2a)}] Exactly four generic points $x_{1},x_{2},x_{3},x_{4}$ on
the link have the same image under $p_{\Sigma}$
\item [{(2b)}] Exactly three points $x_{1},x_{2},x_{3}$ on the link have
the same image under $p_{\Sigma}$, where the segments of the link
near $x_{1},x_{2}$ are tangent to each other, where the order of
tangency is 1
\item [{(2c)}] Exactly two points $x_{1},x_{2}$ on the link have the same
image under $p_{\Sigma}$, where the segments of the link near $x_{1},x_{2}$
are tangent to each other, and the order of tangency is 2
\item [{(2d)}] Two codimension 1 states of type (1c) occur at two different
points of $\Sigma$
\item [{(2e)}] Exactly two points $x_{1},x_{2}$ on the link have the same
image under $p_{\Sigma}$, where the segments of the link near $x_{1}$
projects to the cusp $y^{2}=x^{3}$
\item [{(2f)}] The link projects to $\Sigma$ as the degenerate cusp $y^{2}=x^{5}$
\item [{(2g)}] A codimension 1 state of type (1d) and a state of type (1a)
occur at two different points of $\Sigma$
\item [{(2h)}] A codimension 1 state of type (1d) and a state of type (1b)
or (1c) occur at two different points of $\Sigma$
\item [{(2i)}] Two codimension 1 states of type (1d) occur at two different
points of $\Sigma$ 
\item [{(2j)}] The link is tangent to $\Sigma$, where the order of tangency
is 2
\end{lyxlist}
The monodromies of codimension 2 points are given in Table \ref{codim2-monodromy}.
Note that (none) means a monodromy along a boundary of a 2-cell which
does not contain codimension 2 points. Therefore, using the triangulation
technique of \cite{Juhasz-naturality}, we deduce that the map $\mathcal{R}$
induces isomorphisms of $\pi_{0}$ and $\pi_{1}$.
\end{proof}
\begin{table}[tbph]
\begin{centering}
\begin{tabular}{|c|c|c|c|c|c|c|c|c|c|c|c|}
\hline 
Codimension 2 points & (2a) & (2b) & (2c) & (2d) & (2e) & (2f) & (2g) & (2h) & (2i) & (2j) & (none)\tabularnewline
\hline 
\hline 
Monodromy & (d) & (b) & (c) & (c) & (a) & (i) & (g) & (e) & (f) & (g) & (b),(h)\tabularnewline
\hline 
\end{tabular}
\par\end{centering}
\caption{\label{codim2-monodromy} Monodromies of codimension 2 points in terms
of distinguished diagrams}
\end{table}

Now we stick to the case when $L=K$ is a knot. By the above lemma,
the map 
\[
\mathcal{R}\,:\,\mathcal{B}_{K,\Sigma,p}\rightarrow\text{Emb}_{K}(S^{1},S^{3})
\]
 is 1-connected. Consider the space $E_{K,p}$ of unparametrized based
knots isotopic to $(K,p)$. Since we have a fibration 
\[
\text{Diff}^{+}(S^{1},p)\rightarrow\text{Emb}_{K}(S^{1},S^{3})\rightarrow E_{K,p}
\]
 and the group $\text{Diff}^{+}(S^{1},p)$ is contractible, the map
$\text{Emb}_{K}(S^{1},S^{3})\rightarrow E_{K,p}$ is a homotopy equivalence.
Hence we have an isomorphism $\pi_{1}(\mathcal{B}_{K,\Sigma,p})\simeq\pi_{1}(E_{K,p})$.

The natural action of the diffeomorphism group on $E_{K,p}$ gives
a fibration 
\[
\text{Diff}^{+}(S^{3},K,p)\rightarrow\text{Diff}^{+}(S^{3})\rightarrow E_{K,p}.
\]
 Since $\text{Diff}^{+}(S^{3})$ is path-connected, $\pi_{1}\text{Diff}^{+}(S^{3})\simeq\mathbb{Z}_{2}$
with the generator given by rotation, and $\pi_{0}\text{Diff}^{+}(S^{3},K,p)$
is the mapping class group $MCG(S^{3},K,p)$, we get an exact sequence
\[
\xymatrix{\mathbb{Z}_{2}\ar[r] & \pi_{1}(E_{K,p})\ar[r]\ar[d]_{\sim} & MCG(S^{3},K,p)\ar[r] & 1.\\
 & \pi_{1}(\mathcal{B}_{K,\Sigma,p})
}
\]
 However, since we can place the genus 0 Heegaard surface of $S^{3}$
so that the generator of $\mathbb{Z}_{2}$ acts on it by a full rotation,
and such a rotation induces the identity map of $\widehat{HF}_{\mathbb{Z}_{2}}(\Sigma(K),p)$
by Lemma \ref{lem:identity-diff}, the monodromy representation 
\[
\pi_{1}(\mathcal{B}_{K,\Sigma,p})\rightarrow\text{GL}\left(\widehat{HF}_{\mathbb{Z}_{2}}(\Sigma(K),p)\right)
\]
 factors through $MCG(S^{3},K,p)$. Therefore we have a natural action
of the mapping class group on the equivariant Floer cohomology $\widehat{HF}_{\mathbb{Z}_{2}}(\Sigma(K),p)$. 
\begin{thm}
Let $\text{Knot}_{\ast}$ be the category whose objects are based
knots in $S^{3}$ and morphisms are self-diffeomorphisms of $S^{3}$
which preserve the knot (as a set) and the basepoint. Then we have
a functor 
\[
\widehat{HF}_{\mathbb{Z}_{2}}\,:\,\text{Knot}_{\ast}\rightarrow\text{Vect}_{\mathbb{F}_{2}},
\]
 agreeing up to isomorphism with the invariants $\widehat{HF}_{\mathbb{Z}_{2}}(\Sigma(K))$
defined in \cite{eqv-Floer}.
\end{thm}

\begin{rem*}
A similar argument shows that the same statement holds for links,
when each component has a basepoint. However, since we were working
with based links, where only one component has a basepoint, we cannot
say that $\widehat{HF}_{\mathbb{Z}_{2}}(\Sigma(L),p)$ is natural
with respect to the based link $(L,p)$. 

This causes a slight problem when proving functoriality, so we will
only consider cobordisms where both ends are knots. In this case,
we can ``push off'' the excessive monodromies coming from absense
of basepoints toward an end and then cancel them out.
\end{rem*}
Now we will consider cobordisms between based links in $S^{3}$.
\begin{defn}
A based cobordism $(S,s)$ between based links $(L_{1},p_{1})$ and
$(L_{2},p_{2})$ is an oriented cobordism $S\subset S^{3}\times I$
between $L_{1}$ and $L_{2}$, together with a smooth curve $s\,:\,I\rightarrow S$
such that $s(0)=p_{1}$ and $s(1)=p_{2}$.
\end{defn}

We first consider the case when the based cobordism $(S,s)$ is very
simple. There are three possible cases of such cobordisms, which we
will call as basic pieces. The basic pieces can be defined as the
follows.
\begin{enumerate}
\item cylindrical pieces
\item birth/death of a component without a basepoint
\item saddle along an arc joining two points on the link
\end{enumerate}
To define a map associated to a based cobordism, the most natural
strategy is to chop it into simple pieces. Given a based cobordism
$(S,s)$, consider the projection map $p\,:\,S^{3}\times I\rightarrow I$.
Then we can isotope the pair $(S,s)$ so that it satisfies the following
conditions.
\begin{itemize}
\item $p|_{S}$ is Morse on $S$,
\item $p\circ s$ is regular.
\item No two critical points of $p|_{S}$ have the same value under $p$.
\end{itemize}
Once such conditions are satisfied, we can cut $S$ horizontally to
reduce it into basic pieces, and thus we can represent $(S,s)$ as
a composition of basic pieces. Note that the third condition can be
satisfied because, if there are two critical points of $p$ with the
same value, then we can perturb $p$ slightly to make them have different
values, and if the perturbation is sufficiently $C^{1}$-small, then
the first and second conditions remain hold.

In \cite{eqv-Floer}, it is proved that when $S$ is a basic piece
from $(L_{1},p_{1})$ to $(L_{2},p_{2})$, there exists a map 
\[
\hat{f}_{S}\,:\,\widehat{HF}_{\mathbb{Z}_{2}}(\Sigma(L_{2}),p_{2})\rightarrow\widehat{HF}_{\mathbb{Z}_{2}}(\Sigma(L_{1}),p_{1}),
\]
 which is compatible with the cobordism map between $\widehat{HF}^{\ast}(\Sigma(L_{i}))$,
via the naturally defined spectral sequence 
\[
E_{1}=\widehat{HF}^{\ast}(\Sigma(L))\otimes\mathbb{F}_{2}[\theta]\Rightarrow\widehat{HF}_{\mathbb{Z}_{2}}(\Sigma(L)).
\]
 However, the construction of this map depends on choices of auxiliary
data, and we should prove that our maps are invariant under such choices.

Suppose that $S$ is a saddle along an arc $a$ from $x_{1}\in L_{1}$
to $x_{2}\in L_{2}$, which is indeed a based link cobordism from
$(L_{1},p_{1})$ to $(L_{2},p_{2})$, provided $x_{i}\ne p_{i}$.
Choose a genus zero Heegaard surface $\Sigma\subset S^{3}$. To define
a map $\hat{f}_{S}$ associated to $S$, we need to draw bridge diagrams
of $(L_{i},p_{i})$ on the surface $\Sigma$ and then project the
arc $a$ on $\Sigma$.
\begin{defn}
A saddle diagram of $S$ consists of bridge diagrams $(\{A_{j}^{i}\},\{B_{k}^{i}\})$
of $(L_{i},p_{i})$ on $\Sigma$ for $i=0,1$ and an arc $a_{\Sigma}\subset\Sigma$
from an $X$-arc $A_{j}^{0}$($B_{j}^{0}$) to an $X$-arc $A_{j}^{1}$($B_{j}^{1}$),
where $X$ is either A or B, so that the 1-subcomplex of $S^{3}$
given by pushing the A-arcs inside $\Sigma$ and the B-arcs outside
$\Sigma$ is isotopic to $L_{1}\cup a\cup L_{2}$, and $\text{int}(a_{\Sigma})\cap A_{j}^{i}(B_{k}^{i})=\emptyset$
for all $i$ and $j$($k$). 
\end{defn}

There are several choice of saddle diagrams which represent $S$.
However, given two bridge diagrams of $(L_{i},p_{i})$ for $i=0,1$,
which coincide outside the saddle region, we only have to choose the
placement of the arc $a_{\Sigma}$. Any two choice of $a_{\Sigma}$
are related by isotopies and handleslides(over A(B)-arcs). If $a_{\Sigma}$
and $a_{\Sigma}^{\prime}$ are related by a single handleslide, then
the bridge diagrams representing the saddle moves along $a_{\Sigma}$
and $a_{\Sigma}^{\prime}$ are related by a composition of two handleslides,
as shown in Figure \ref{fig:6.7}. This can be represented as a distinguished
diagram of bridge diagrams as follows. Note that we are using saddle
diagrams in a perturbed form, as in \cite{eqv-Floer}.
\[
\xymatrix{D_{1}\ar[rr]_{a_{\Sigma}^{\prime}}\ar[dr]_{a_{\Sigma}} &  & D_{3}\\
 & D_{2}\ar[ur]_{\text{handleslides}}
}
\]

\begin{figure}[tbph]
\resizebox{.4\textwidth}{!}{\includegraphics{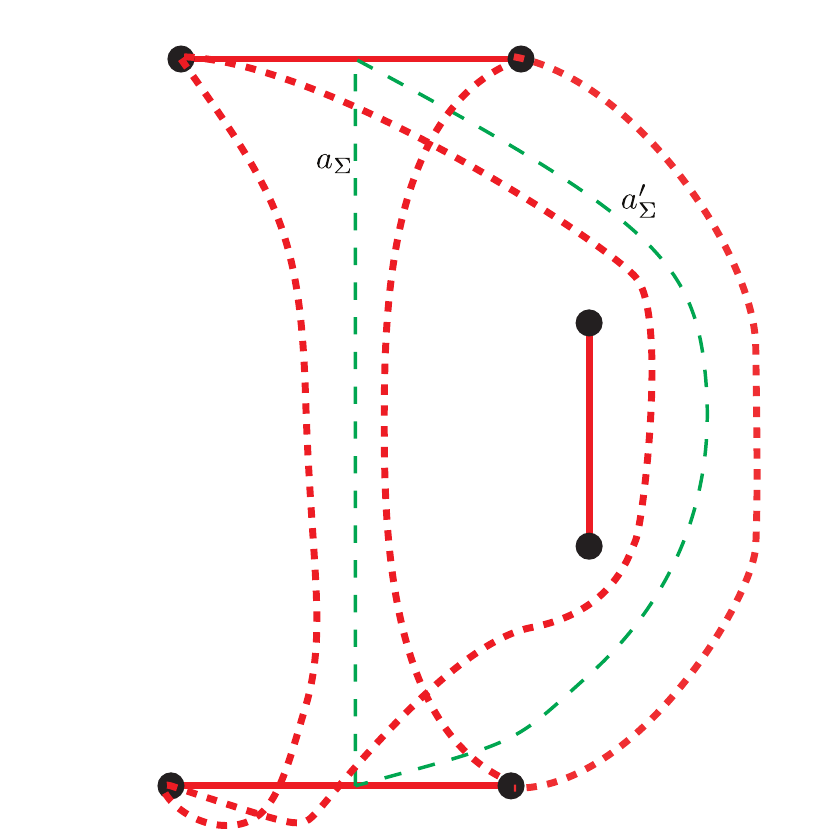}} \caption{\label{fig:6.7}Taking saddles along two arcs $a_{\Sigma}$ and $a_{\Sigma}^{\prime}$,
which differ by a handleslide}
\end{figure}
\begin{lem}
The above diagram induces a commutative diagram of corresponding $\widehat{HF}_{\mathbb{Z}_{2}}$
groups:
\[
\xymatrix{\widehat{HF}_{\mathbb{Z}_{2}}(\Sigma(L_{2}),p_{2})\ar[dr]_{\text{handleslides}}\ar[rr]_{\text{saddle along }a_{\Sigma}^{\prime}} &  & \widehat{HF}_{\mathbb{Z}_{2}}(\Sigma(L_{1}),p_{1}).\\
 & \widehat{HF}_{\mathbb{Z}_{2}}(\Sigma(L_{2}),p_{2})\ar[ur]_{\text{saddle along }a_{\Sigma}}
}
\]
\end{lem}

\begin{proof}
The Heegaard diagram for a saddle along $a_{\Sigma}^{\prime}$ admits
a unique $\mathbb{Z}_{2}$-invariant top generator.
\end{proof}
The above lemma means that the saddle map is invariant under the choice
of $a_{\Sigma}$. However, it remains to show the invariant under
the choice of bridge diagrams of $L_{1}$ and $L_{2}$. For that,
we need to show that the following commutative diagrams induce commutative
diagrams of $\widehat{HF}_{\mathbb{Z}_{2}}$ groups.
\begin{itemize}
\item Basic moves and saddles commute.
\item Diffeomorphisms and saddles commute.
\end{itemize}
\begin{equation}
\xymatrix{D\ar[r]_{\text{saddle}}\ar[d]_{\text{basic moves/diffeomorphisms}} & D_{2}\ar[d]^{\text{basic moves/diffeomorphisms}}\\
D_{1}\ar[r]_{\text{saddle}} & D_{3}
}
\label{eq:diag}
\end{equation}
\begin{lem}
The diagram (\ref{eq:diag}) induces commutative diagrams of $\widehat{HF}_{\mathbb{Z}_{2}}$
groups: 
\[
\xymatrix{\widehat{HF}_{\mathbb{Z}_{2}}(\Sigma(L_{2}),p_{2})\ar[r]_{\text{saddle}}\ar[d]_{\text{basic moves/diffeomorphisms}} & \widehat{HF}_{\mathbb{Z}_{2}}(\Sigma(L_{1}),p_{1})\ar[d]^{\text{basic moves/diffeomorphisms}}\\
\widehat{HF}_{\mathbb{Z}_{2}}(\Sigma(L_{2}),p_{2})\ar[r]_{\text{saddle}} & \widehat{HF}_{\mathbb{Z}_{2}}(\Sigma(L_{1}),p_{1})
}
\]
\end{lem}

\begin{proof}
The lemma is obvious for diffeomorphisms, and Lemma \ref{lemma-isotopy-tri}
implies that the lemma holds fo isotopies. The proof for handleslides
and stabilizations are the same as the proof of Proposition \ref{prop:commutation-a-h}
for distinguished diagrams of type (c) and (e), respectively.
\end{proof}
Using the above lemma, we see that the saddles give well-defined saddle
maps.
\begin{thm}
The basic pieces give well-defined maps between $\widehat{HF}_{\mathbb{Z}_{2}}$
groups.
\end{thm}

\begin{proof}
The saddle maps are well-defined by the above lemma. The birth map
and the death map correspond to taking/untaking a connecting sum with
an invariant Heegaard diagram of $S^{1}\times S^{2}$ on the branched
double cover. Since they commute with basic moves (and, obviously,
diffeomorphisms) when the connected sum neck is very long, the theorem
follows.
\end{proof}
Now consider the case when we are given a based cobordism $S$ from
a based knot $(K_{1},p_{1})$ to a based knot $(K_{2},p_{2})$. Then
we can isotope $S$, slice it into basic pieces, convert them into
maps between $\widehat{HF}_{\mathbb{Z}_{2}}$ groups, and then compose
them to get a cobordism map 
\[
\hat{f}_{S}\,\,:\,\widehat{HF}_{\mathbb{Z}_{2}}(\Sigma(K_{2}),p_{2})\rightarrow\widehat{HF}_{\mathbb{Z}_{2}}(\Sigma(K_{1}),p_{1}).
\]
 We claim that $\hat{f}_{S}$ does not depend on the process of slicing
$S$ into basic pieces.
\begin{lem}
\label{lemma-6.15}$\hat{f}_{S}$ depends only on the isotopy class
of based cobordisms $S=(S_{0},s)$ rel $s\cup\partial S_{0}$.
\end{lem}

\begin{proof}
Isotoping $S$ and then slicing it into basic pieces is equivalent
to representing the cobordism $S$ as a movie from $(K_{1},p_{1})$
to $(K_{2},p_{2})$. It is known that any two movies of $S$ are related
through a sequence of 15 possible types of movie moves, which are
defined by Carter and Saito in the paper \cite{Carter-Saito}. Note
that, although the result of Carter and Saito is about non-based cobordisms,
it can also be applied directly to based cobordisms, by taking saddles
moves to occur away from the basepoint.

Among the 15 types of movie moves, the only type that needs a proof
for functoriality of equivariant Floer cohomology is the one drawn
in Fig.30 of \cite{Carter-Saito}, which corresponds to the following
commutative triangles.
\[
\xymatrix{L\ar[rr]_{\text{destabilization}}\ar[dr]_{\text{birth}} &  & L, & L\ar[dr]_{\text{saddle}}\ar[rr]_{\text{stabilization}} &  & L\\
 & L\coprod U\ar[ur]_{\text{saddle}} &  &  & L\coprod U\ar[ur]_{\text{death}}
}
\]
 But the composition of a birth map followed by a saddle map is precisely
the definition of the stabilization map, which proves that the triangle
on the left induces a commutative triangle of $\widehat{HF}_{\mathbb{Z}_{2}}$
groups; see the proof of Theorem 1.24 of \cite{eqv-Floer} for the
definition. 

The triangle on the right induces a commuative triangle of $\widehat{HF}_{\mathbb{Z}_{2}}$
groups if the composite cobordism 
\[
S\,:\,\xymatrix{L\ar[r]_{\text{saddle}} & L\coprod U\ar[r]_{\text{death}} & L\ar[r]_{\text{birth}} & L\coprod U\ar[r]_{\text{saddle}} & L}
\]
 induces identity on equivariant Heegaard Floer cohomology. First,
if $L$ is a union of two unlinked unknots, then this is obvious.
Next, if $L=L_{0}\coprod U_{0}$, where $U_{0}$ is an unlinked unknot
component, the basepoint lies in $L_{0}$, and the saddle moves are
taken on the component $U_{0}$, then the question reduces to the
previous case by making a long neck between the bridge diagrams of
$L_{0}$ and $U_{0}$, as drawn in \ref{fig:6.8}. 

For general link $L$, consider the saddle move on $L$, as in Figure
\ref{fig:6.8-1} (the green dashed line is the saddle arc) so that
the resulting link $L^{\prime}$ is isotopic to $L\coprod\text{unknot}$
where the basepoint does not lie on the unknot component. Since the
saddle map $\widehat{HF}_{\mathbb{Z}_{2}}(\Sigma(L),p)\xrightarrow{\hat{f}_{S_{a}}}\widehat{HF}_{\mathbb{Z}_{2}}(\Sigma(L^{\prime}),p)$
postcomposed with the birth map $\widehat{HF}_{\mathbb{Z}_{2}}(\Sigma(L^{\prime}),p)\rightarrow\widehat{HF}_{\mathbb{Z}_{2}}(\Sigma(L),p)$
is the stabilization map, which is an isomorphism, we see that $\hat{f}_{S_{a}}$
is injective. Now, by the proof of \ref{prop:commutation-a-h} for
distinguished diagrams of type (c), the following diagram commutes.
\[
\xymatrix{\widehat{HF}_{\mathbb{Z}_{2}}(\Sigma(L))\ar[r]_{\hat{f}_{S}}\ar[d]_{\hat{f}_{S_{a}}} & \widehat{HF}_{\mathbb{Z}_{2}}(\Sigma(L))\ar[d]_{\hat{f}_{S_{a}}}\\
\widehat{HF}_{\mathbb{Z}_{2}}(\Sigma(L^{\prime}))\ar[r]_{\hat{f}_{S\coprod(L\times I)}} & \widehat{HF}_{\mathbb{Z}_{2}}(\Sigma(L^{\prime}))
}
\]
 But we already know that $\hat{f}_{S\coprod(L\times I)}$ is the
identity. Therefore, since $\hat{f}_{S_{a}}$ is injective, the map
$\hat{f}_{S}$, too, is the identity.
\end{proof}
\begin{figure}[tbph]
\resizebox{.7\textwidth}{!}{\includegraphics{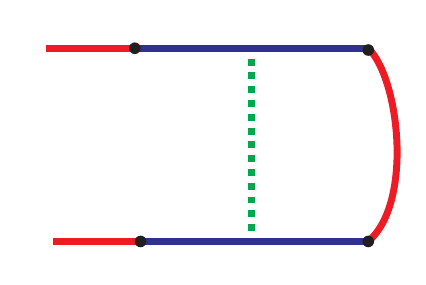}} \caption{\label{fig:6.8-1}Saddle along a small arc near $L$}
\end{figure}
\begin{figure}[tbph]
\resizebox{.4\textwidth}{!}{\includegraphics{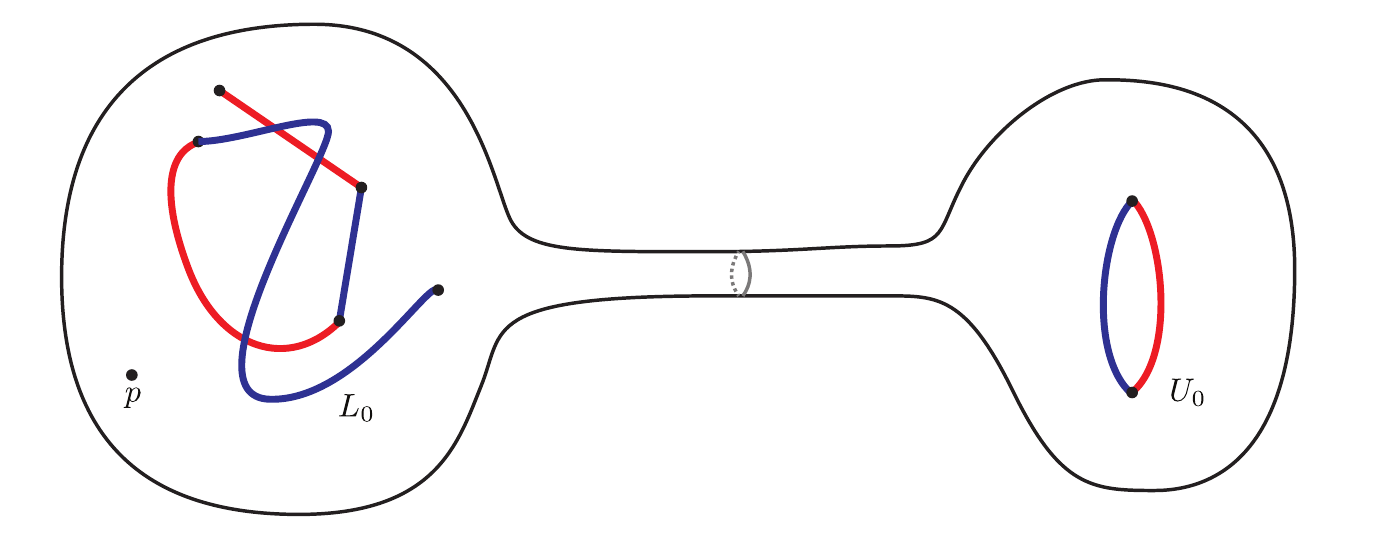}} \caption{\label{fig:6.8}The bridge diagram of $L=L_{0}\coprod U_{0}$ on a
sphere with a long neck}
\end{figure}
\begin{thm}
$\hat{f}_{S}$ depends only on the isotopy class of based cobordisms
$S=(S_{0},s)$ rel $\partial S_{0}$.
\end{thm}

\begin{proof}
Denote the projection map $S^{3}\times I\rightarrow I$ by $p$. We
only have to prove that, given a knot cobordism $S_{0}$ and two curves
$s$ and $s^{\prime}$ such that $\partial s=\partial s^{\prime}$,
$s$ is isotopic to $s^{\prime}$ rel $\partial s$, and $p$ is regular
when restricted to $s$ and $s^{\prime}$, we have 
\[
\hat{f}_{(S_{0},s)}=\hat{f}_{(S_{0},s^{\prime})}.
\]
 Isotope $S_{0}$ rel $s\cup s^{\prime}$, so that $p|_{S_{0}}$ is
Morse and the critical points of $p|_{S_{0}}$ have distinct images
under $p$. We can isotope $s$ slightly and horizontally so that
it intersects $s^{\prime}$ transversely. We know that $s$ and $s^{\prime}$
cobound a disk $D\subset S$, which is innermost in the sense that
$D\cap(s\cup s^{\prime})=\partial D$. Using Lemma \ref{lemma-6.15},
we can isotope $S$ so that $p|_{D}$ has no critical points. Then
we get an isotopy $\{s_{t}\}_{t\in[0,1]}$ from $s$ to $s^{\prime}$
rel $\partial s$, such that $h$ is regular on $s_{t}$ for each
$t\in[0,1]$. Since isotopy maps (on $\Sigma$) commute with all other
types of maps, we deduce that the based cobordisms $(S,s_{t})$ induce
the same map for all $t$. Therefore, by induction on $|s\cap s^{\prime}|$,
we get $\hat{f}_{(S_{0},s)}=\hat{f}_{(S_{0},s^{\prime})}$.
\end{proof}
Thus we can now rewrite the equivariant Heegaard Floer theory $\widehat{HF}_{\mathbb{Z}_{2}}(\Sigma(L),p)$
of based links as a well-defined functor. We first define a notion
of isotopies of based cobordisms, as follows.

We can now define a category $\mathbf{bCob}$ as follows.
\begin{itemize}
\item The objects of $\mathbf{bCob}$ are based knots in $S^{3}$.
\item The morphisms of $\mathbf{bCob}$ between based knots $K_{1},K_{2}$
are isotopy classes of based cobordisms from $K_{1}$ to $K_{2}$.
\end{itemize}
Then what we proved up to now can be rephrased as follows.
\begin{thm}
\label{thm:functoriality}$\widehat{HF}_{\mathbb{Z}_{2}}\,:\,\mathbf{bCob}\rightarrow\mathbf{Mod}_{\mathbb{F}_{2}[\theta]}$
is a functor.
\end{thm}

Note that the equivariant HF theory of links is an unoriented theory.
The group $\widehat{HF}_{\mathbb{Z}_{2}}(\Sigma(K),z)$ does not depend
on the orientation of $K$, and a knot cobordism induces maps between
$\widehat{HF}_{\mathbb{Z}_{2}}$ in both ways.
\begin{rem*}
Using the same argument, we can prove that a based cobordism of links
induces a cobordism map, which is well-defined up to monodromy. In
other words, if $S$ is a based cobordism from $(L_{1},p_{1})$ to
$(L_{2},p_{2})$ and we denote the monodromy of $\widehat{HF}_{\mathbb{Z}_{2}}(\Sigma(L_{i}),p_{i})$
by 
\[
\rho_{i}\,:\,\pi_{1}(\mathcal{B}_{L_{i},\Sigma,p_{i}})\rightarrow\text{GL}(\widehat{HF}_{\mathbb{Z}_{2}}(\Sigma(L_{i}),p_{i})),
\]
 then the cobordism map 
\[
\hat{f}_{S}\,:\,\widehat{HF}_{\mathbb{Z}_{2}}(\Sigma(L_{2}),p_{2})\rightarrow\widehat{HF}_{\mathbb{Z}_{2}}(\Sigma(L_{1}),p_{1})
\]
 is well-defined up to composition with elements in the images of
$\rho_{1}$ (and $\rho_{2}$).
\end{rem*}

\section{Functoriality of $c_{\mathbb{Z}_{2}}$ under certain symplectic cobordisms}

Now we restrict to the case when the links are transverse, with respect
to the standard contact structure $\xi_{std}$ on $S^{3}$. Then the
cobordisms we should use in this case are symplectic cobordisms. Recall
that symplectic cobordism of links is defined as follows.
\begin{defn}
Let $L_{0},L_{1}$ be transverse links in $(S^{3},\xi_{std})$. A
symplectic cobordism from $L_{0}$ to $L_{1}$ is an embedded symplectic
surface $S\subset(S^{3}\times[0,R],d(-e^{t}\xi_{std}))$ where $R$
is a positive real, $S\cap(S^{3}\times\{i\})=L_{i}$ for $i=0,1$,
and $S$ is cylindrical near both ends.
\end{defn}

\begin{example}
Suppose that we are given an isotopy $rels\cup s^{\prime}\{L_{t}\}$
of transverse links, $0\le t\le1$. Given a symplectization 
\[
(S^{3}\times[0,1],\omega_{R}=d(e^{Rt}\alpha_{std}))\simeq(S^{3}\times[0,R],d(e^{t}\alpha_{std}))
\]
 for $R>0$, where $\alpha_{std}=dz+xdy$(on $\mathbb{R}^{3}$), consider
the cobordism 
\[
S_{R}=\bigcup_{0\le t\le1}L_{t}\times\{Rt\}.
\]
 Choose a point $(p,t)\in S^{3}\times[0,R]$. Then the tangent plane
$T_{(p,Rt)}S_{R}$ is spanned by the vector $\partial_{t}+w_{p,t}$
and the line $T_{p}L_{t}$, for some vector $w_{p,t}\in T_{p}S^{3}$.
Let $v_{p}$ be a tangent vector spanning $T_{p}L_{t}$, which is
chosen to vary smoothly. Then we have 
\begin{align*}
(\omega_{R})_{(p,Rt)}(\partial_{t}+w_{p,t},v_{p}) & =e^{Rt}(Rdt\wedge\xi_{std}+d\alpha_{std})(\partial_{t}+w_{p,t},v_{p})\\
 & =Re^{Rt}(\alpha_{std}(v_{p})+\frac{1}{R}d\alpha_{std}(\partial_{t}+w_{p,t},v_{p})).
\end{align*}
 Since $L_{t}$ are transverse to $\xi_{std}$, by compactness, the
value of $\alpha_{std}(v_{p})$ is bounded away from zero. Hence,
for sufficiently large $R$, the cobordism $S$ is symplectic. In
other words, $S_{R}$ is symplectic for sufficiently large $R>0$;
we will call such surfaces as isotopy cylinders.

Now suppose that we are given a cobordism $S\subset(S^{3}\times[0,R],d(-e^{t}\alpha_{std}))$
which is symplectic and cylindrical near both ends. Let $L=S\cap(S^{3}\times\{0\})\subset S^{3}$
and choose a generator $v_{p}\in T_{p}L$ for $p\in S^{3}$. Since
$S$ is symplectic, by assumption, we must have 
\[
0\ne d(e^{t}\alpha_{std})(\partial_{t},v_{p})=e^{t}(dt\wedge\alpha_{std}+d\alpha_{std})(\partial_{t},v_{p})=e^{t}\alpha_{std}(v_{p}).
\]
 This implies $\alpha_{std}(v_{p})\ne0$, i.e. $L$ is transverse
to $\xi_{std}$. Therefore we see that symplectic cobordisms form
a good notion of cobordisms between transverse links.
\end{example}

Before we explicitly construct ``basic pieces'' of based symplectic
cobordisms which achieve the functoriality for equivariant contact
classes, we need to define the notion of weak symplectic isotopies
between symplectic cobordisms. They are defined as follows.
\begin{defn}
Two symplectic cobordisms $S_{1},S_{2}$ are weakly symplectically
isotopic if they are symplectically isotopic after concatenating with
(trivial) cylindrical cobordisms on both ends.

A based symplectic cobordism is a based cobordism $(S,s)$ where $S$
is symplectic. Two based symplectic cobordisms $(S,s)$ and $(S^{\prime},s^{\prime})$
are weakly symplectically isotopic if $S$ and $S^{\prime}$ are weakly
symplectically isotopic and, after performing the isotopy, the homotopy
class $[s]-[s^{\prime}]$ is contained in the image of the map $\pi_{1}(\partial S)\rightarrow\pi_{1}(S)$.
\end{defn}

Thus we are led to define a category $\mathbf{sCob}_{w}$ as follows:
\begin{itemize}
\item objects of $\mathbf{sCob}_{w}$ are based transverse knots;
\item morphisms of $\mathbf{sCob}_{w}$ are weak symplectic based isotopy
classes of based symplectic cobordisms.
\end{itemize}
Then we have a well-defined natural functor 
\[
\mathbf{sCob}_{w}\rightarrow\mathbf{bCob},
\]
 which forgets all the contact and symplectic structures involved
and reverses the directions of cobordisms. In other words, a weak
symplectic based isotopy class of based symplectic cobordisms can
be seen as a based isotopy class of based cobordisms.

Now we construct the ``basic pieces'' of symplectic cobordisms,
which will only be defined up to weak symplectic based isotopy. The
idea is to mimic the symplectic handle constructions on the branched
double cover side, so we can construct the symplectic models of ``birth''
and ``saddle'' cobordisms. Note that a birth corresponds to a 1-handle
and a saddle corresponds to a 2-handle, while a death corresponds
to a 3-handle; 4-dimensional symplectic 3-handles do not exist.

The birth case is easy:
\begin{defn}
\label{birth-symp-def}Given a transverse based link $L$ in $(S^{3},\xi_{std})$,
choose a point $p\in S^{3}$ whch lies outside $L$. Consider a 1-parameter
family $\{U_{t}\}_{t\in(0,1]}$ which converges to $p$, so that in
the front projection $(x,y,z)\mapsto(x,z)$, the family is represented
by Figure \ref{fig:7.1}. Then the symplectic birth cobordism from
$L$ is defined as 
\[
B(L)=L\times[0,R]\cup\left(\bigcup_{t\in[0,1]}U_{1-t}\times\{Rt\}\right)\subset S^{3}\times[0,R],
\]
 where $R>0$ is chosen so that $B(L)$ becomes symplectic.
\end{defn}

\begin{figure}[tbph]
\resizebox{.4\textwidth}{!}{\includegraphics[scale=0.75]{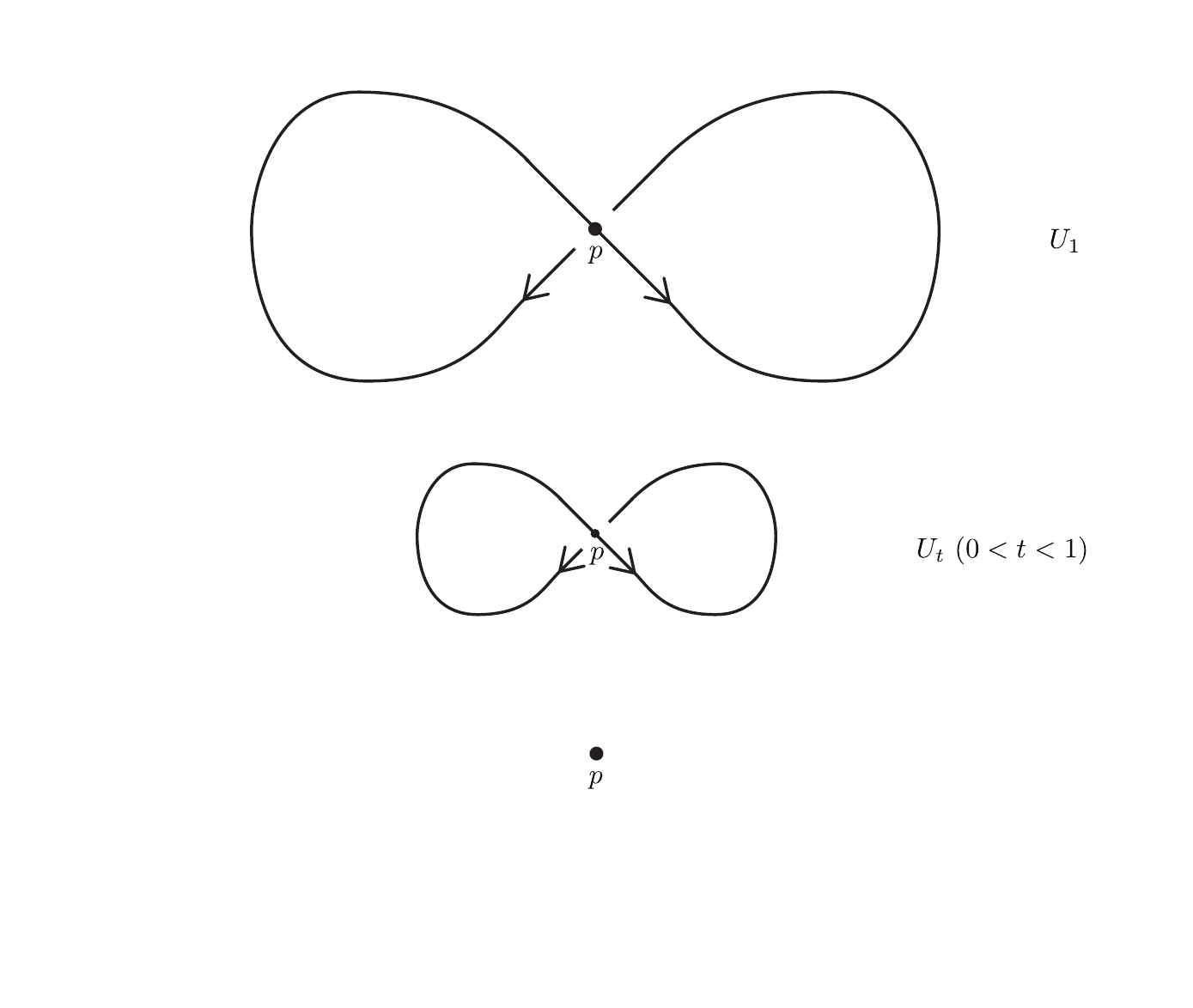}}
\caption{\label{fig:7.1}The family $\{U_{t}\}_{t\in(0,1]}$ of transverse
unknots which converge to $p$ as $t\rightarrow0$}
\end{figure}
\begin{lem}
There exists $R>0$ such that $B(L)$ in Definition \ref{birth-symp-def}
becomes symplectic, and the weak symplectic based isotopy type of
$B(L)$ depends only on $L$.
\end{lem}

\begin{proof}
The proof is very similar to the proof of Theorem \ref{saddle-symp},
so we do not write it down here.
\end{proof}
We will now define symplectic saddle cobordisms. In the case of ordinary
links and cobordisms, a saddle is defined by the following topological
constructions. 
\begin{itemize}
\item (C1) Given a link $L$, choose an arc $a$ which intersects transversely
with $L$ at $\partial a$, such that $a\cap L=\partial a$.
\item (C2) Remove a small neighborhood of $\partial a$ in $L$ and replace
it by two arcs parallel to $a$. The saddle between those two links
is the desired saddle cobordism of $L$ along $a$.
\end{itemize}
In our case, the link $L$ is transverse, and the arcs $a_{i}$ should
be Legendrian. The reason is explained in Example \ref{exa:legendriansaddlearc}.
\begin{example}
\label{exa:legendriansaddlearc}Let $L$ be a transverse link in $(S^{3},\xi_{std})$.
As we have seen previously, $\xi_{std}$ lifts to a contact structure
$\xi_{L}$ on the branched double cover $\Sigma(L)$, which is determined
uniquely up to isotopy(actually, $\mathbb{Z}_{2}$-eqivariant isotopy).
As a smooth manifold, $\Sigma(L)$ is defined by removing a standard
neighborhood 
\[
(N(L),\xi_{std}|_{N(L)})\simeq(S^{1}\times D^{2},\ker(dz+r^{2}d\theta)),
\]
 taking double cover with respect to the meridian $\partial D^{2}$,
and then regluing a copy of $S^{1}\times D^{2}$ where the covering
transformation acts by $(r,\theta)\mapsto(r^{2},2\theta)$.

Suppose that a small smooth curve $\gamma\subset\Sigma(L)$ defined
near a point $p\in L$ is invariant under $\mathbb{Z}_{2}$. Then,
in the parametrization $N(L)\simeq S^{1}\times D^{2}$, the tangent
line $T_{p}\gamma$ must be spanned by the vector $\partial_{r}$.
Since the contact structure $\xi_{L}$ on $\Sigma(L)$ is defined
near $L$ by $\ker(\alpha_{L})$ where $\alpha_{L}=dz+f(r)d\theta$
for a good increasing $f$, we deduce that, for any $v\in T_{p}\gamma$,
\[
(\alpha_{L})_{p}(v)=dz(v)+f(r)d\theta(v)=0.
\]
 Therefore $\gamma$ should be Legendrian at $\gamma\cap L$, whch
implies that it is very natural to work with Legendrian arcs while
dealing with saddle cobordisms of transverse links.
\end{example}

Now let $L$ be a based transverse link with basepoint $z\in L$,
and suppose that we are given a Legendrian arc $a$ satisfying the
conditions (C1) and (C2), which does not contain the basepoint $z$.
Then the set $L\cup a$ is an embedded graph in $S^{3}$, such that
it consists of transverse edge-cycles together with one Legendrian
edge connecting points on the transverse cycles, and has a basepoint
in some transverse edge. We shall impose a further condition on $a$
to ensure that the lift $\gamma$ of $a$ becomes a smooth Legendrian
knot in $(\Sigma(L),\xi_{L})$:
\begin{itemize}
\item There exists a standard neighborhood $(N(L),\xi_{std}|_{L})\simeq\coprod(S^{1}\times D^{2},\ker(dz+r^{2}d\theta))$
such that the curve segment $a\cap N(L)$ is a straight radial line
(i.e. point towards the $r$-axis).
\end{itemize}
When this condition is also satisfied, we shall call such a graph
$(L,a,z)$ a nice graph.

Suppose that a nice graph $G=(L,a,z)$ is given. By the definition
of nice graphs, there exists a standard neighborhood $(N(L),\xi_{std}|_{L})\simeq\coprod(S^{1}\times D^{2},\ker(dz+r^{2}d\theta))$
such that the curve segment $a\cap N(L)$ is a straight radial line
(i.e. towards the $r$-axis). So we have a contact 1-form representative
$\alpha$ of $\xi_{std}$ such that 
\[
\alpha|_{N(L)}=dz+r^{2}d\theta,
\]
 in the given neighborhood parametrization. Now, by the standard neighborhood
theorem for Legendrian arcs, we have a neighborhood parametrization,
which is uniquely determined up to isotopy:
\[
(N(a),\alpha|_{N(a)})\simeq([0,\pi/2]\times D^{2},\ker(\cos(z)dx+\sin(z)dy)).
\]
 Here, by a (closed) neighborhood of the arc $a$, we mean an embedded
closed disk-bundle over $a$. Sometimes we will choose a slight extension
of the above parametrization, so that we have a contact embedding
of the cylinder 
\[
\left(\left[-\epsilon,\frac{\pi}{2}+\epsilon\right]\times D^{2},\ker(\cos(z)dx+\sin(z)dy)\right).
\]
\begin{example}
Consider the contact manifold 
\[
U=([0,\pi/2]\times D^{2},\alpha=\cos(z)dx+\sin(z)dy),
\]
 which is the standard neighborhood of the Legendrian arc $[0,\pi]\times\{0\}$.
If a curve $\gamma=(\gamma_{x},\gamma_{y},\gamma_{z})$ satisfies
\[
\frac{\gamma_{y}^{\prime}}{\gamma_{x}^{\prime}}=2\tan(z),
\]
 Then $\alpha(\gamma^{\prime})=(1+\sin^{2}z)\cdot v\ne0$ where $v=\frac{\gamma_{x}^{\prime}}{\cos(z)}=\frac{\gamma_{y}^{\prime}}{2\sin(z)}$.
So, given a regular curve 
\[
a(t)=(a_{x}(t),a_{y}(t))\in D^{2},
\]
 the induced curve 
\[
V_{a}=\left(a_{x},a_{y},\arctan\left(\frac{a_{y}^{\prime}}{2a_{x}^{\prime}}\right)\right)\in D\times[0,\pi/2],
\]
 if defined, is always transverse. Also note that we have
\begin{align*}
d\alpha_{U} & =dz\wedge(-\sin(z)dx+\cos(z)dy).
\end{align*}
\end{example}

\begin{rem*}
In this section, when we draw a regular curve $a$ on $D^{2}$, we
will mean the induced transverse curve $V_{a}$ in a slightly extended
cylinder $\left[-\epsilon,\frac{\pi}{2}+\epsilon\right]\times D^{2}$
for a very small $\epsilon>0$.
\end{rem*}
Since $L$ is a Reeb orbit of $\alpha$ and the Reeb orbits of $\cos\theta dx+\sin\theta dy$
passing through a point in $I\times\{\mathbf{0}\}$ are radial lines,
we deduce that the curve segments $L\cap N(a_{i})$ are given by the
$y$-axes on $D^{2}\times\{\pm\pi/2\}$. Then we consider a family
of curves $\{L_{t}\}_{t\in[0,1]}$defined as in Figure \ref{fig:7.2},
embedded in the cylinder $S^{3}\times[0,R]$.
\begin{figure}[tbph]
\resizebox{.7\textwidth}{!}{\includegraphics{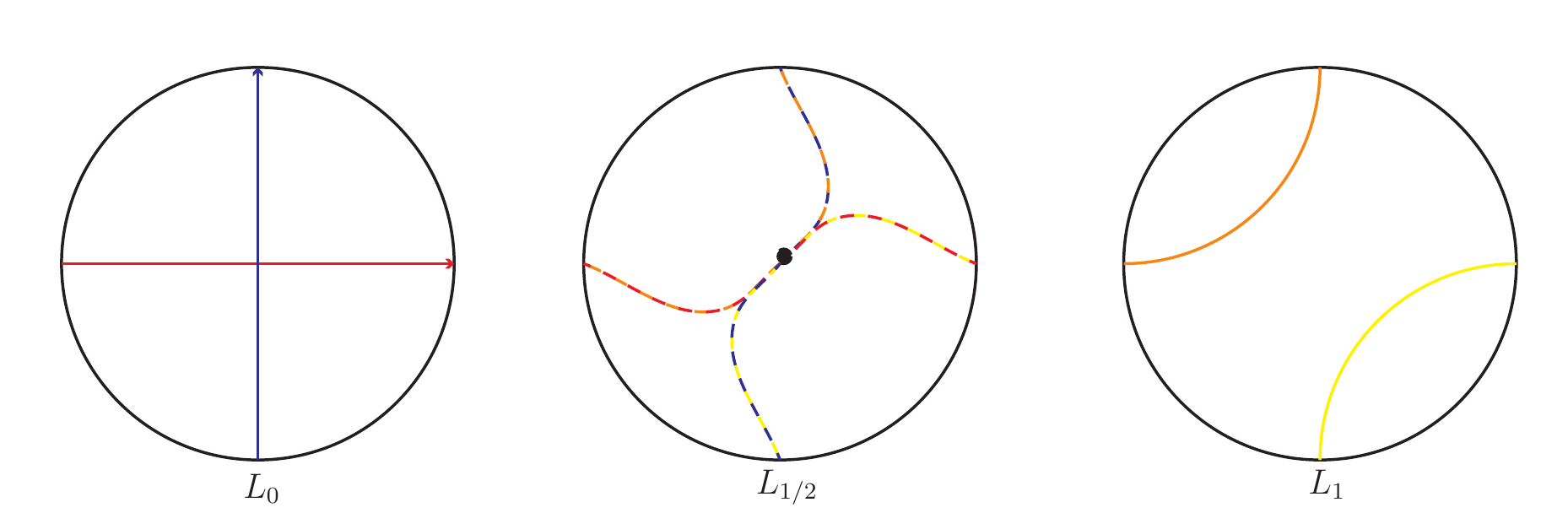}} \caption{\label{fig:7.2}The family of curves $L_{t}$ at $t=0$, $t=\frac{1}{2}$,
and $t=1$.}
\end{figure}

However, since the diagram in the middle gives two curves which are
tangent at a point with tangent line $\mathbb{R}\partial_{z}$, we
see that taking the union $\cup_{t\in[0,1]}L_{t}\times\{Rt\}$ of
such curves does not give us a surface in $S^{3}\times[0,R]$. To
resolve this problem, we perform a slight pertubation to give a new
family $\{\tilde{L}_{t}\}$, where the projection of $\tilde{L}_{1/2}\subset D^{2}\times I$
to $D^{2}$ is given as in Figure \ref{fig:7.4}.
\begin{figure}[tbph]
\resizebox{.4\textwidth}{!}{\includegraphics{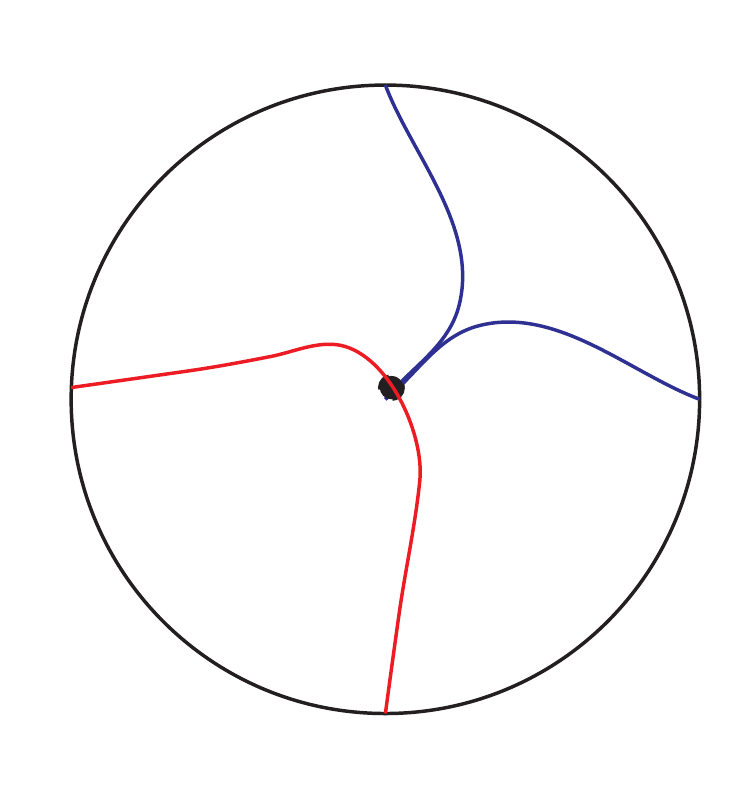}} \caption{\label{fig:7.4}The projection of $\tilde{L}_{1/2}$ along the $z$-axis.}
\end{figure}

If the perturbation is small enough, $\tilde{L}_{t}$ is transverse
for $t\ne\frac{1}{2}$ and $\tilde{L}_{t}-\{p\}$ is transverse for
$t=\frac{1}{2}$. Also, the two tangent vectors to $\tilde{L}_{1/2}$
at $p$ define a positive orientation of $D^{2}$. We can now form
a symplectic cobordism from such a family.
\begin{thm}
\label{saddle-symp}When $R$ is sufficiently large, the cobordism
$S_{R}=\cup_{t\in[0,1]}\tilde{L}_{t}\times\{Rt\}$ is symplectic,
and its weak symplectic isotopy class depends only on the given transverse
link $L$ and a Legendrian arc $a$.
\end{thm}

\begin{proof}
For simplicity, fix $R=1$ and let the symplectic structures on $S^{3}\times I$
vary:
\[
\omega_{R}=d(e^{Rt}\alpha_{std})=e^{Rt}(Rdt\wedge\alpha_{std}+d\alpha_{std}).
\]
We are asking whether the given cobordism, which we will denote by
$S$, is symplectic with respect to the symplectic form 
\[
\frac{1}{Re^{Rt}}\omega_{R}=dt\wedge\alpha_{std}+\frac{1}{R}d\alpha_{std}.
\]
Choose a smooth nonvanishing tangent bivector field $\sigma$ on $S$.
By the transversality assumption, the term $dt\wedge\alpha_{std}(\sigma)$
is everywhere nonnegative, and it vanishes only at the saddle point,
which we will call $p$. On the other hand, the term $d\alpha_{std}(\sigma)$
may not be everywhere nonnegative, but it takes a positive value at
$p$. So, suppose that $d\alpha_{std}(\sigma)>0$ in the $r$-ball
centered at the saddle point, and let 
\[
\max_{S}|d\alpha_{std}(\sigma)|<M,\,\min_{S-B_{r}(p)}dt\wedge\alpha_{std}(\sigma)>m.
\]
 Then, if $R>\frac{M}{m}$, we get 
\[
\left|\left(dt\wedge\alpha_{std}+\frac{1}{R}d\alpha_{std}\right)(\sigma)\right|\ge m-\frac{M}{R}>0
\]
 on $S-B_{r}(p)$. But this quantity is also positive inside $B_{r}(p)$
by our assumptions. Therefore, for all such $R$, the cobordism becomes
symplectic. The statement about the weak symplectic isotopy class
follows from the fact that, given a smooth 1-parameter family of auxiliary
choices that we have made, we can choose a very large stretching factor
$R$, which gives us an 1-parameter family of symplectic saddle cobordisms.
\end{proof}
\begin{rem*}
The same argument can also be used for any surfaces constructed by
a similar procedure. In particular, it works for symplectic births.
\end{rem*}
\begin{defn}
Denote such a cobordism, with the curve $\{z\}\times[0,R]$ on it,
by $S=S(L,a,z)$. Then $S$ has $(L,z)$ as a concave end and a based
transverse link $L^{\prime}=(C(L,a),z)$ as a convex end. The (transverse
isotopy class of) based transverse link $(C(L,a),z)$ is called the
transverse surgery of $(L,z)$ along $a$. The (weak symplectic isotopy
class of) based symplectic cobordism $S(L,a,z)$ is called the symplectic
saddle of $(L,z)$ along $a$.
\end{defn}

With those constructions in mind, we will say that the based symplectic
cobordisms which can be constructed from the symplectic births and
saddles are symplectically constructible. More precisely, we have
the following definition.
\begin{defn}
The weak symplectic isotopy class of based symplectic cobordisms,
obtained from gluing the classes of symplectic births and saddles,
are called constructible classes. The cobordisms contained in constructible
classes are called symplectically constructible based cobordisms.
\end{defn}

Now we prove that the maps between $\widehat{HF}_{\mathbb{Z}_{2}}$,
associated to symplectically constructible based cobordisms, maps
the equivariant contact class of the concave end to the equivariant
contact class of the convex end. We start from the simplest case,
when the based link $L$ is already braided along the $z$-axis and
the Legendrian arc $a$ is in a very nice position.
\begin{defn}
Let $L$ be a transverse link in $(S^{3},\xi_{std})$, which is braided
along the $z$-axis. A simple Legendrian arc $a$ is basic if there
exists a genus-zero open book supporting $(S^{3},\xi_{std})$, whose
binding is the $z$-axis, such that $a$ is contained in a single
page.
\end{defn}

\begin{lem}
Let $(L,a,z)$ be a nice graph, where $L$ is braided along the $z$-axis,
and suppose that $a$ is basic. Let
\[
\hat{f}_{S(L,a,z)}\,:\,\widehat{HF}_{\mathbb{Z}_{2}}(\Sigma(C(L,a)),z)\rightarrow\widehat{HF}_{\mathbb{Z}_{2}}(\Sigma(L),z)
\]
 be the map associated to the based cobordism class $S(L,a,z)$. Then
we have 
\[
\hat{f}_{S(L,a,z)}(c_{\mathbb{Z}_{2}}(\xi_{C(L,a)},z))=c_{\mathbb{Z}_{2}}(\xi_{L},z).
\]
\end{lem}

\begin{proof}
Since $a$ is basic, we know that the contact branched double cover
of $(S^{3},\xi_{std})$ along $C(L,a)$ is the contact (-1)-surgery
of $(\Sigma(L),\xi_{L})$ along the lift $\gamma$ of $a$. The map
$\hat{f}_{S(L,\{a_{i}\},z)}$ is given by the equivariant triangle
map. The Heegaard triple-diagram we get is described in Figure \ref{fig:7.5}.

As in the proof of invariance under positive stabilizaton, by the
convenient placement of a basepoint, all triangles connecting $\mathbf{x}$
and $\Theta$ in the diagram are small. So we deduce that the higher
order terms in the equivariant triangle map vanish, since they count
triangles of Maslov index at most 0. Thus we get the desired equality:
\begin{align*}
\hat{f}_{S(L,a,z)}(\xi_{C(L,a)},z) & =c_{\mathbb{Z}_{2}}(\xi_{L},z)+\text{higher order terms}\\
 & =c_{\mathbb{Z}_{2}}(\xi_{L},z).
\end{align*}
\end{proof}
\begin{figure}[tbph]
\resizebox{.7\textwidth}{!}{\includegraphics{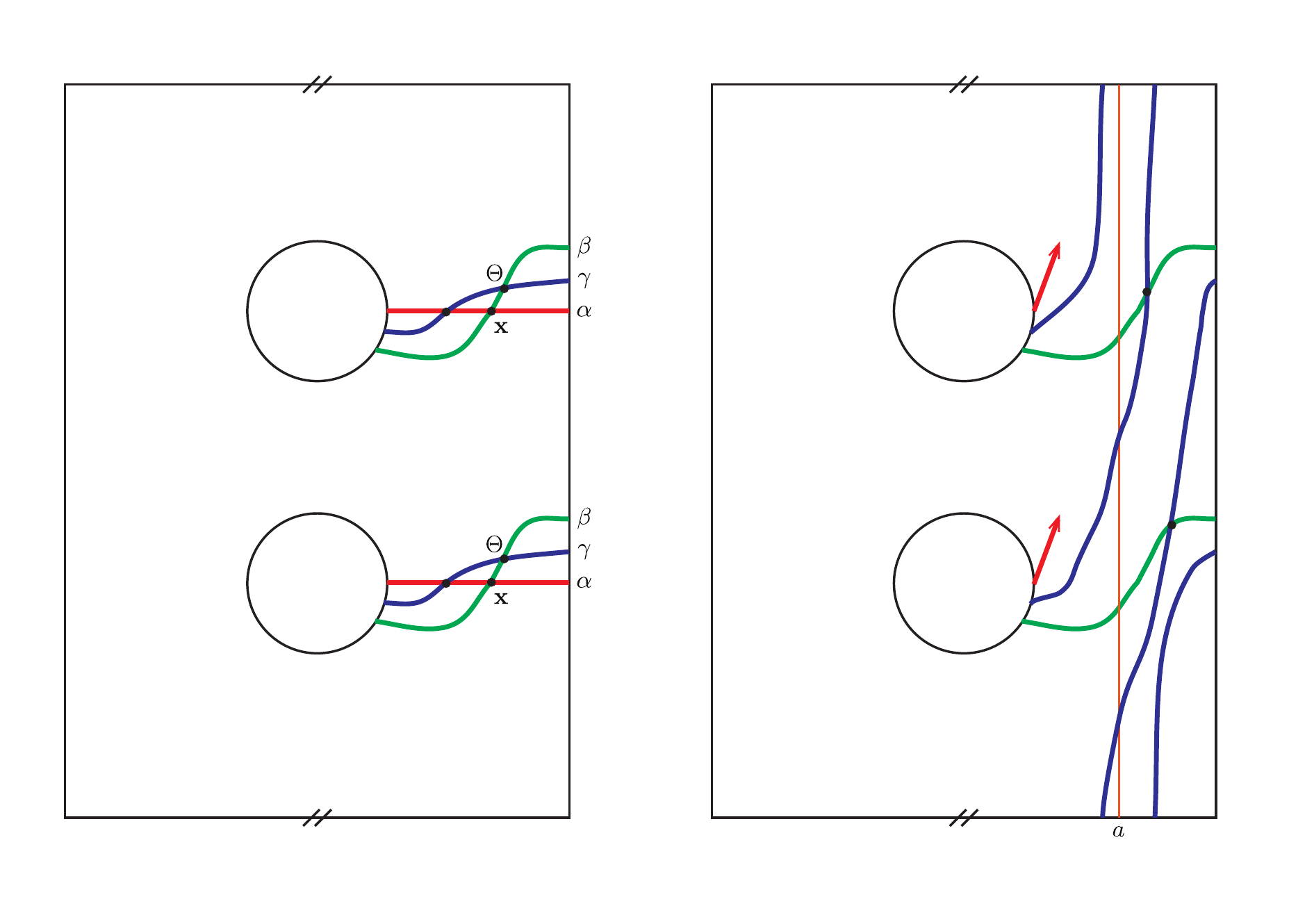}} \caption{\label{fig:7.5}The associated Heegaard triple-diagram.}
\end{figure}

In the general case when $L$ is not in a braid position and $a$
is arbitrary, we argue that we can alway isotope the whole situation
to the above case, where $L$ is braided and $a$ is basic. For that
we will have to deform our nice graph into a 4-valent transverse graph.
\begin{defn}
A based 4-valent transverse graph embedded in $(S^{3},\xi_{std})$
is a 4-valent directed graph $\Gamma$ with exactly one 4-valent vertex
and several 2-valent vertices, together with a basepoint $z$ on an
edge of $\Gamma$ and an embedding
\[
\Gamma\hookrightarrow S^{3},
\]
 such that each edge is transverse, the two adjacent edges of 2-valent
vertices glue smoothly, and the 4-valent vertex has an labelling $l_{1}^{\pm},l_{2}^{\pm}$
of its adjacent edges so that, as directed smooth curves, $l_{1}^{\pm}$
glue smoothly with $l_{2}^{\pm}$.
\end{defn}

Recall that, given a nice graph $(L,a,z)$, we have defined the surgered
link $C(L,a)$ and the symplectic saddle $S(L,a,z)$ by a sequence
of diagrams drawn on $D^{2}$. By considering the intermediate slice
(before applying a perturbation to make it a well-defined symplectic
surface), which is defined as in Figure \ref{fig:7.6}, we can see
that $L$ transforms to $C(L,a)$ through a based 4-valent transverse
graph $G(L,a,z)$ in $(S^{3},\xi_{std})$.
\begin{figure}[tbph]
\resizebox{.4\textwidth}{!}{\includegraphics{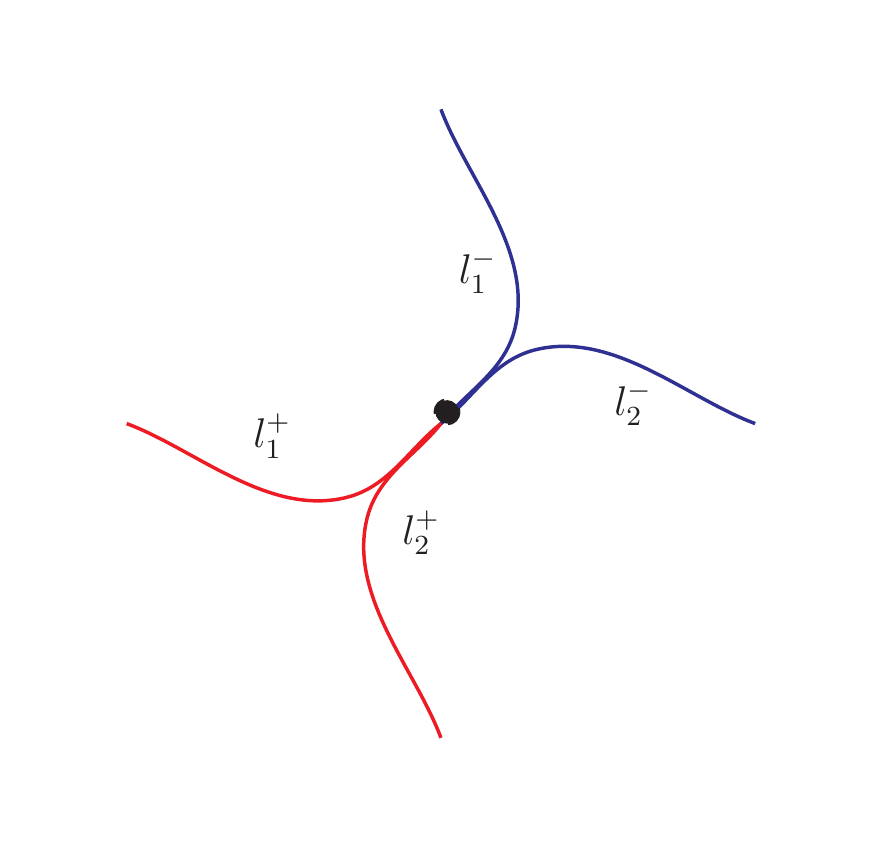}} \caption{\label{fig:7.6}The intermediate slice.}
\end{figure}

Now, given a nice graph $(L,a,z)$, we want to isotope it, through
a 1-parameter family of nice graphs, to another nice graph $(L^{\prime},a^{\prime},z^{\prime})$,
where $L^{\prime}$ is braided along the $z$-axis and $a^{\prime}$
is basic. To do that, we first isotope the based 4-valent transverse
graph $G(L,a,z)$, so that its edges are in braid position. This is
always possible due to the following (variant of) theorem of Bennequin.
\begin{thm}
\label{bennequin}\cite{Bennequin-graphs} Any smooth transverse graph
can be transversely isotoped into a braided position along the $z$-axis.
\end{thm}

\begin{lem}
Every nice graph $(L,a,z)$ can be isotoped to another nice graph
$(L^{\prime},a^{\prime},z)$ through nice graphs, so that $L^{\prime}$
is braided and $a^{\prime}$ is basic.
\end{lem}

\begin{proof}
By Theorem \ref{bennequin}, we can transversely isotope the based
4-valent transverse graph $G(L,a,z)$, so that the isotopy is supported
away from its unique 4-valent vertex, say $v$. Since $(L,a,z)$ can
be isotoped so that it agrees with $G(L,a,z)$ away from the vertex
$v$ by the definition of $G(L,a,z)$, we see that the transverse
isotopy of $G(L,a,z)$ can also be applied to $(L,a,z)$ so that $a$
is very short and $L$ is braided. Therefore, by isotoping the pages
of the open book of $S^{3}$ in a small neighborhood of $v$, we see
that $a$ can be made basic, while remaining $L$ braided.
\end{proof}
Now we can prove the functoriality in the most general setting.
\begin{lem}
Given a nice graph $(L,a,z)$, let 
\[
\hat{f}_{S(L,\{a_{i}\},z)}\,:\,\widehat{HF}_{\mathbb{Z}_{2}}(\Sigma(C(L,a),z)\rightarrow\widehat{HF}_{\mathbb{Z}_{2}}(\Sigma(L),z)
\]
 be the map associated to the based cobordism class $S(L,a,z)$. Then
we have 
\[
\hat{f}_{S(L,a,z)}(c_{\mathbb{Z}_{2}}(\xi_{C(L,a)},z))=c_{\mathbb{Z}_{2}}(\xi_{L},z).
\]
\end{lem}

\begin{proof}
Using Bennequin's theorem, we can isotope the decorated transverse
4-valent graph 
\[
(\Gamma,V,z)=G(L,a,z)
\]
 into a braided position along the $z$-axis, with respect to some
genus-0 open book having the $z$-axis as its binding. Furthermore,
we can isotope the vector field $V$ so that it is tangent to the
pages of the open book. Then its detachment is a nice graph $(L^{\prime},a^{\prime},z^{\prime})$,
whch is isotopic to the original nice graph $(L,a,z)$, such that
$L^{\prime}$ is braided along the $z$-axis and $a$ is basic. Therefore,
by the above lemma and the invariance of equivariant contact classes
under isotopies, we must have 
\[
\hat{f}_{S(L,a,z)}(c_{\mathbb{Z}_{2}}(\xi_{C(L,a)},z))=c_{\mathbb{Z}_{2}}(\xi_{L},z).
\]
\end{proof}
\begin{thm}
Given any symplectically constructible (weak symplectic isotopy) class
$S$ of based symplectic cobordisms, with its concave and convex ends
given by transverse isotopy classes $(L_{1},z_{1})$ and $(L_{2},z_{2})$
of based transverse links, we have 
\[
\hat{f}_{S}(c_{\mathbb{Z}_{2}}(\xi_{L_{2}},z_{2}))=c_{\mathbb{Z}_{2}}(\xi_{L_{1}},z_{1}),
\]
 where $\hat{f}_{S}\,:\widehat{HF}_{\mathbb{Z}_{2}}(\Sigma(L_{2}),z_{2})\rightarrow\widehat{HF}_{\mathbb{Z}_{2}}(\Sigma(L_{1}),z_{1})$
is the cobordism map, induced by $S$.
\end{thm}

\begin{proof}
We already have the functoriality for both symplectic birth cobordisms
and symplectic saddle cobordisms, and the functoriality for cylindrical
cobordisms is obvious. Therefore, by composition, we get the functoriality
for all symplectically constructible cobordisms.
\end{proof}
\begin{cor}
We have a functor 
\[
(\widehat{HF}_{\mathbb{Z}_{2}},c_{\mathbb{Z}_{2}}):\mathbf{sCob}_{w}^{c}\rightarrow(\mathbb{F}_{2}[\theta]\downarrow\mathbf{Mod}_{\mathbb{F}_{2}[\theta]}),
\]
 where $\mathbf{sCob}_{w}^{c}$ is the wide subcategory of $\mathbf{sCob}_{w}$
spanned by symplectically constructible cobordism classes and $\mathbb{F}_{2}[\theta]\downarrow\mathbf{Mod}_{\mathbb{F}_{2}[\theta]}$
is the category of modules over $\mathbb{F}_{2}[\theta]$ with a $\theta$-tower
generated by a distinguished element.
\end{cor}

\begin{proof}
This is just a category-theoretic statement for the theorem above.
\end{proof}

\section{Properties}

The first property of $c_{\mathbb{Z}_{2}}(\xi_{L},z)$ is that it
contains the information about the ordinary contact class $c(\xi_{L})$.
\begin{thm}
\label{thm:forget}The natural map 
\[
\widehat{HF}_{\mathbb{Z}_{2}}(\Sigma(L),z)\rightarrow\widehat{HF}^{\ast}(\Sigma(L))
\]
sends $c_{\mathbb{Z}_{2}}(\xi_{L},z)$ to $c(\xi_{L})$.
\end{thm}

\begin{proof}
The chain level map is given by truncating all terms with nontrivial
$\theta$-degree, so it sends $EH_{\mathbb{Z}_{2}}^{\ast}(\xi_{L})=EH^{\ast}(\xi_{L})\otimes\theta^{0}$
to $EH^{\ast}(\xi_{L})$. Hence, on the cohomology level, $c_{\mathbb{Z}_{2}}(\xi_{L})$
is sent to $c(\xi_{L})$.
\end{proof}
Also, as in Section 6.1 of \cite{eqv-Floer}, we have a localization
isomorphism 
\[
\theta^{-1}\widehat{HF}_{\mathbb{Z}_{2}}(\Sigma(K))\simeq\widehat{HF}^{\ast}(S^{3})\otimes\mathbb{F}_{2}[\theta,\theta^{-1}],
\]
 where $K$ is a knot and $\theta^{-1}$ means that we are formally
inverting $\theta$, i.e. we define 
\[
\theta^{-1}\widehat{HF}_{\mathbb{Z}_{2}}(\Sigma(K))=\widehat{HF}_{\mathbb{Z}_{2}}(\Sigma(K),z)\otimes_{\mathbb{F}_{2}[\theta]}\mathbb{F}_{2}[\theta,\theta^{-1}].
\]
 In turns out that the image of the equivariant contact class under
the localization isomorphism takes a very simple form.
\begin{thm}
\label{thm:localization}Let $K$ be a transverse knot in $(S^{3},\xi_{std})$.
Then the localization map 
\[
\theta^{-1}\widehat{HF}_{\mathbb{Z}_{2}}(\Sigma(K))\xrightarrow{\sim}\widehat{HF}(S^{3})\otimes\mathbb{F}_{2}[\theta,\theta^{-1}]\simeq\mathbb{F}_{2}[\theta,\theta^{-1}],
\]
 which is defined up to multiplication by powers of $\theta$, sends
$c_{\mathbb{Z}_{2}}(\xi_{K})$ to a power of $\theta$.
\end{thm}

\begin{proof}
By the construction of the bare localization map in \cite{localization-map},
the localization map can be written as follows:
\[
c_{\mathbb{Z}_{2}}(\xi_{K})\mapsto(c(\xi_{std})+\mbox{higher order terms})\otimes\theta^{d},
\]
 for some $d$. But since there are no holomorphic disks going towards
$EH^{\ast}(\xi_{K})$, the higher order terms vanish. Therefore the
localization map sends $c_{\mathbb{Z}_{2}}(\xi_{K})$ to $c(\xi_{std})\otimes\theta^{d}=\theta^{d}$.
\end{proof}
This theorem gives us a lower bound for the $d_{3}$-invariants of
the branched double covers along transverse knots in $(S^{3},\xi_{std})$.
Recall that $q_{\tau}(K)$ is defined as:
\[
q_{\tau}(K)=2\cdot\min\{\text{gr}(x)\,\vert\,x\in\widehat{HF}_{\mathbb{Z}_{2}}(\Sigma(K)),\,\theta^{k}x\ne0\text{ for all }k\ge0\}.
\]
\begin{cor}
\label{cor:ineq}For a knot $K$ in $S^{3}$, denote the set of all
transverse representatives of $K$ as $\mathcal{T}_{K}$. Then we
have 
\[
\frac{q_{\tau}(K)-1}{2}\le\min_{T\in\mathcal{T}_{K}}d_{3}(\xi_{T}).
\]
\end{cor}

\begin{proof}
For any transverse representative $T\in\mathcal{T}_{K}$, the absolute
$\mathbb{Q}$-grading $\text{gr}(EH(\xi_{T}))$ of the contact element
of $(\Sigma(K),\xi_{T})$, which is the same as $\text{gr}(c_{\mathbb{Z}_{2}}(\xi_{T}))$,
is given by $\frac{1}{2}+d_{3}(\xi_{T})$; see Proposition 4.6 of
\cite{OSz-contact}. By the above theorem, $c_{\mathbb{Z}_{2}}(\xi_{T})$
cannot be annihilated by a power of $\theta$, since the localization
map is an isomorphism of $\mathbb{F}_{2}[\theta]$-modules. Therefore
we have $\frac{q_{\tau}(K)-1}{2}\le\min_{T\in\mathcal{T}_{K}}d_{3}(\xi_{T})$
for all such $T$.
\end{proof}
The functoriality of equivariant contact classes for symplectically
constructible cobordisms also gives us some results about symplectic
representatives of link cobordisms.
\begin{thm}
Let $(L_{1},z_{1})$ and $(L_{2},z_{2})$ be two based transverse
links in $(S^{3},\xi_{std})$. If the based isotopy class of a based
cobordism $S$ from $(L_{1},z_{1})$ to $(L_{2},z_{2})$ has a symplectically
constructible representative, then we must have 
\[
\hat{f}_{S}(c_{\mathbb{Z}_{2}}(L_{2},z_{2}))=c_{\mathbb{Z}_{2}}(L_{1},z_{1}),
\]
 where $\hat{f}_{S}$ is the cobordism map induced by $S$.
\end{thm}

\begin{proof}
This follows directly from the functoriality and the fact that $\hat{f}_{S}$
depends only on the based isotopy class of $S$.
\end{proof}
We can also explicitly calculate the equivariant contact class for
some very simple transverse knots.
\begin{example}
Consider the trivial transverse braid $U$ and its positive/negative
stabilizations $P$, $N$, respectively. We will see from the proof
of the next theorem that $c_{\mathbb{Z}_{2}}^{\ast}(\xi_{U})=c_{\mathbb{Z}_{2}}^{\ast}(\xi_{P})=1$
but $c_{\mathbb{Z}_{2}}^{\ast}(\xi_{N})=\theta$. This reflects the
fact that, while the contact elements of $\xi_{U}$ and $\xi_{P}$
have Maslov degree zero, the contact element of $\xi_{N}$ has Maslov
degree one, in the Floer chain complex of $S^{3}$. 
\end{example}

Note that, while performing a positive stabilization to a transverse
link (on any of its components) does not change its transverse isotopy
class, performing a negative stabilization does change its transverse
isotopy class. However, the topological isotopy class does not change
under negative stabilizations, so the equivariant contact class of
a transverse link and its positive stabilization lie in the same group.
It turns out that the behavior of the equivariant contact class under
a negative stabilization is very simple.
\begin{thm}
Let $L$ be a transverse link in $(S^{3},\xi_{std})$ and denote its
negative stabilization(i.e. transverse stabilization), applied to
any of its components, by $L^{-}$. Then we have 
\[
c_{\mathbb{Z}_{2}}(\xi_{L^{-}})=\theta\cdot c_{\mathbb{Z}_{2}}(\xi_{L}).
\]
\end{thm}

\begin{proof}
Put $L$ in a braided position. Then the negative stabilization $L^{-}$
is given by adding a negative twist to the last two strands in $L\coprod U$
where $U$ is the trivial braid. So the equivariant Heegaard diagram
for $\Sigma(L^{-})$ near the last strand is given by Figure \ref{fig:8.2}.
\begin{figure}[tbph]
\resizebox{.4\textwidth}{!}{\includegraphics{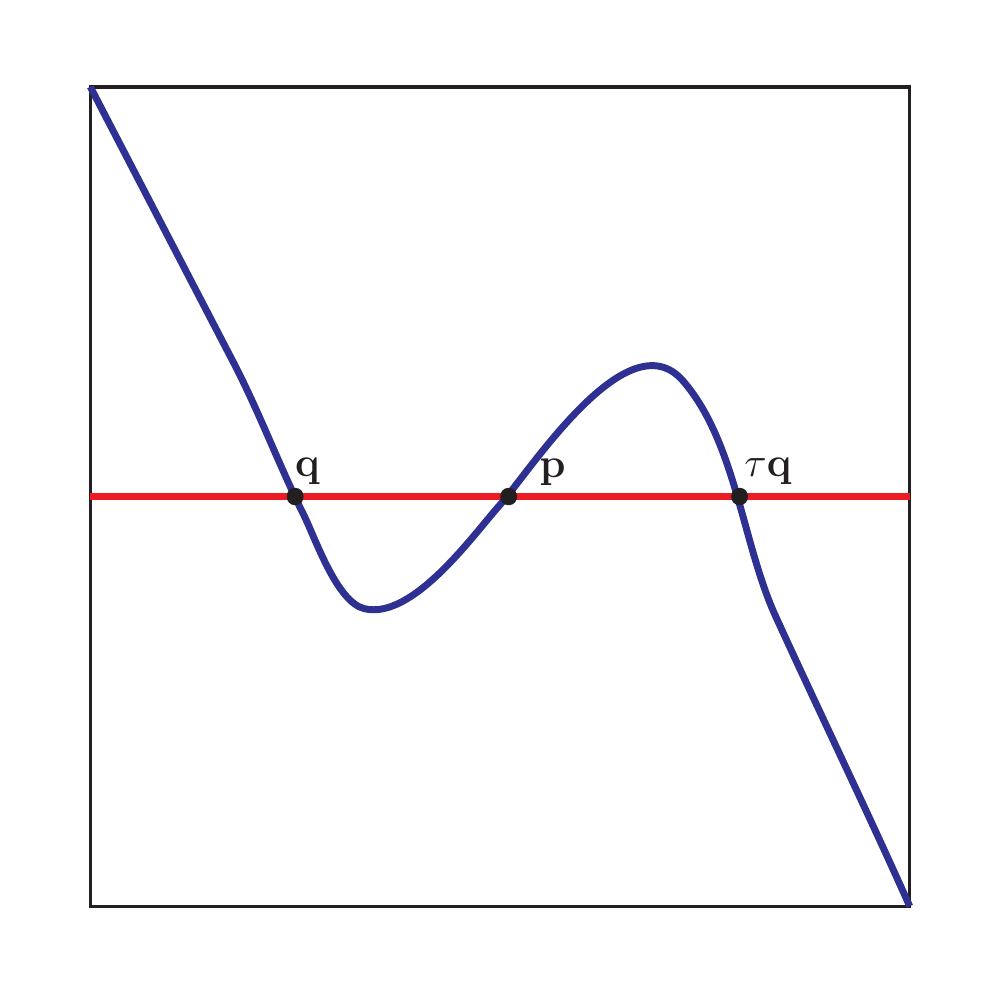}} \caption{\label{fig:8.2}The associated equivariant Heegaard diagram, near
the last strand.}
\end{figure}
 Put $\mathbf{x}=EH(\xi_{L})\otimes\mathbf{q}$. Then, in the dual
of the freed Floer complex, we have 
\begin{align*}
d_{\mathbb{Z}_{2}}\mathbf{x} & =EH(\xi_{L})\otimes\mathbf{p}\otimes\theta^{0}+EH(\xi_{L})\otimes(\mathbf{q}+\tau\mathbf{q})\otimes\theta^{1}\\
 & =EH_{\mathbb{Z}_{2}}(\xi_{L})+EH(\xi_{L})\otimes(\mathbf{q}+\tau\mathbf{q})\otimes\theta^{1}.
\end{align*}
 Now consider the equivariant triple Heegaard diagram in Figure \ref{fig:8.3},
which describes the negative stabilization $L^{-}$ and the positive
stabilization $L^{+}$ of $L$:
\begin{figure}[tbph]
\resizebox{.4\textwidth}{!}{\includegraphics{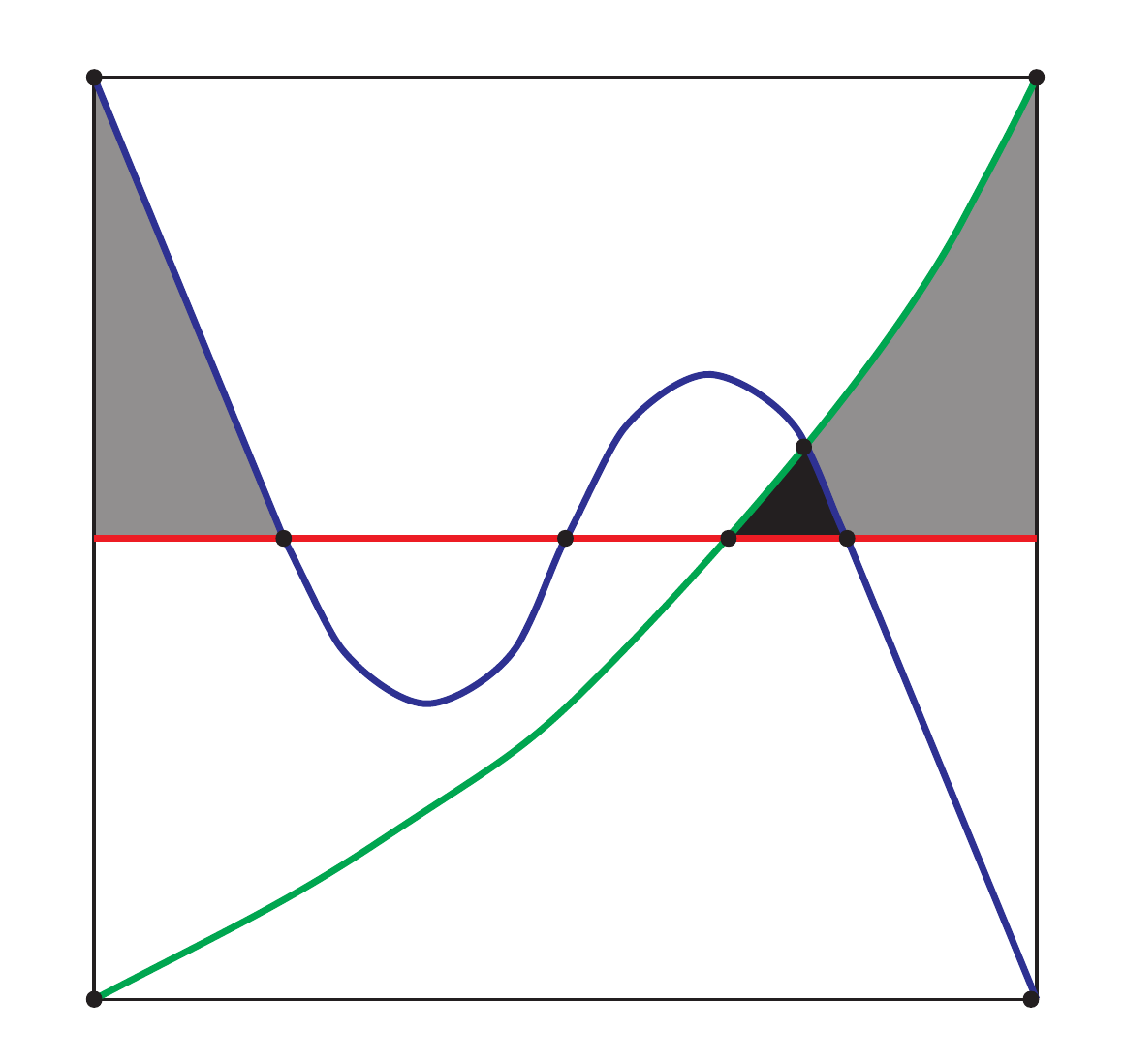}} \caption{\label{fig:8.3}Relevant triangles in the associated equivariant Heegaard
triple-diagram, each of which has zero Maslov index.}
\end{figure}
 The two shaded triangles are the only holomorphic triangles in the
above diagram, so by the associativity of equivariant triangle maps,
working with equivariant contact classes in equivariant HFs should
be the same as working with contact classes in ordinary HFs. In other
words, we have
\[
c_{\mathbb{Z}_{2}}(\xi_{L^{+}})=[EH_{\mathbb{Z}_{2}}(\xi_{L^{+}})]=[EH(\xi_{L})\otimes(\mathbf{q}+\tau\mathbf{q})\otimes\theta^{0}]\in\widehat{HF}_{\mathbb{Z}_{2}}(\Sigma(L)).
\]
 Therefore we get 
\[
c_{\mathbb{Z}_{2}}(\xi_{L^{-}})=\theta\cdot c_{\mathbb{Z}_{2}}(\xi_{L^{+}})=\theta\cdot c_{\mathbb{Z}_{2}}(\xi_{L}).
\]
\end{proof}
Finally, using the naturality and functoriality of equivariant Heegaard
Floer cohomology, we can construct an isotopy invariant of slice disks
of a given slice knot in $S^{3}$, as follows. Note that a similar
invariant can be constructed using functoriality of knot Floer homology
under decorated cobordisms, as the $t_{S,P}$ invariant in \cite{Juhasz-Marengon}.
\begin{thm}
Let $K\subset S^{3}$ be a (smoothly) slice knot, and $D\subset B^{4}$
be a slice disk which bounds $K$. Choose any point $p\in D$, its
neighborhood $B(p)\subset B^{4}$, and draw a smooth simple arc $s$
on the annulus $D^{2}-B(p)$, so that $s\cap K$ is a point and $(D^{2}-B(p),s)$
is a based cobordism between $(K,s\cap K)$ and the based unknot.
Consider the induced cobordism map:
\[
\hat{f}_{(D^{2}-B(p),s)}\,:\,\widehat{HF}_{\mathbb{Z}_{2}}(\Sigma(\text{unknot}),\text{pt})\rightarrow\widehat{HF}_{\mathbb{Z}_{2}}(\Sigma(K),s\cap K).
\]
 Then the element $\hat{f}_{(D^{2}-B(p),s)}(\mathbf{x})$, where $\mathbf{x}\in\widehat{HF}_{\mathbb{Z}_{2}}(\Sigma(\text{unknot}),\text{pt})$,
is nonvanishing and depends only on the isotopy class of $D$ rel
$K$.
\end{thm}

\begin{proof}
The fact that $\hat{f}_{(D^{2}-B(p),s)}(\mathbf{x})$ depends only
on the isotopy class of $D$ rel $K$ following from Theorem \ref{thm:functoriality},
since $D^{2}-B(p)$ is an annulus and thus the choice of $s$ is unique
up to homotopy. The fact that $\hat{f}_{(D^{2}-B(p),s)}(\mathbf{x})$
is nonvanishing follows from Theorem \ref{thm:localization} and Lemma
6.11 of \cite{eqv-Floer}.
\end{proof}

\section{Vanishing and nonvanishing of $c(\xi_{K})$}

Recall that, for any knot $K\subset S^{3}$, we have a spectral sequence
\[
E_{1}=\widehat{HF}^{\ast}(\Sigma(K))\otimes_{\mathbb{F}_{2}}\mathbb{F}_{2}[\theta]\Rightarrow\widehat{HF}_{\mathbb{Z}_{2}}(\Sigma(K)),
\]
 constructed in Section 6.1 of \cite{eqv-Floer}, whose pages depend
only on the isotopy class of $K$. This spectral sequence is induced
by the $\theta$-filtration on $\widehat{CF}_{\mathbb{Z}_{2}}(\Sigma(K))$,
which can be written up to quasi-isomorphism as 
\begin{align*}
\widehat{CF}_{\mathbb{Z}_{2}}(\Sigma(K)) & =(\widehat{CF}(-\Sigma(K))\otimes_{\mathbb{F}_{2}}\mathbb{F}_{2}[\theta],d_{\mathbb{Z}_{2}}),\\
d_{\mathbb{Z}_{2}}(\mathbf{x}\otimes\theta^{i}) & =d\mathbf{x}\otimes\theta^{i}+(\mathbf{x}+\tau\mathbf{x})\otimes\theta^{i},
\end{align*}
 where $d$ denotes the differential on the cochain complex $\widehat{CF}^{\ast}(\Sigma(K))$
and $\tau$ denotes the generator of the $\mathbb{Z}_{2}$-action.

From the construction on the transverse knot invariant $c_{\mathbb{Z}_{2}}(\xi_{K})$,
we know that, given a transverse braid representation of $K$ along
the $z$-axis, we have an element 
\[
EH_{\mathbb{Z}_{2}}(\xi_{K})=EH(\xi_{K})\otimes\theta^{0}\in\widehat{CF}_{\mathbb{Z}_{2}}(\Sigma(K)),
\]
 which is a $d_{\mathbb{Z}_{2}}$-cocycle. The same element represents
$c(\xi_{K})\otimes\theta^{0}$ in the $E_{1}$ page, so we see that
the element $c(\xi_{K})\otimes\theta^{0}$ in the $E_{1}$ page of
our spectral sequence induces an element on each page.

However, we have to be careful here: the limit of $c(\xi_{K})\otimes\theta^{0}$
on the $E^{\infty}$ page is not $c_{\mathbb{Z}_{2}}(\xi_{K})$. This
is because our spectral sequence actually does not converge directly
to $\widehat{HF}_{\mathbb{Z}_{2}}(\Sigma(K))$, but instead converges
to its associated graded module 
\[
\mathbf{gr}_{\theta}\widehat{HF}_{\mathbb{Z}_{2}}(\Sigma(K))=\bigoplus_{i=0}^{\infty}\theta^{i}\widehat{HF}_{\mathbb{Z}_{2}}(\Sigma(K))/\theta^{i+1}\widehat{HF}_{\mathbb{Z}_{2}}(\Sigma(K)).
\]
 Thus, considering the bigrading on each page of the sequence, we
see that the limit of $c(\xi_{K})\otimes\theta^{0}$ is the image
$c_{\mathbb{Z}_{2}}(\xi_{K})$ under the following map:
\[
\widehat{HF}_{\mathbb{Z}_{2}}(\Sigma(K))\twoheadrightarrow\widehat{HF}_{\mathbb{Z}_{2}}(\Sigma(K))/\theta\widehat{HF}_{\mathbb{Z}_{2}}(\Sigma(K))\hookrightarrow\mathbf{gr}_{\theta}\widehat{HF}_{\mathbb{Z}_{2}}(\Sigma(K)).
\]
 Similarly, for every nonnegative integer $n$, the limit of $c(\xi_{K})\otimes\theta^{n}$
is the image of $\theta^{n}\cdot c_{\mathbb{Z}_{2}}(\xi_{K})$ under
the map 
\[
\widehat{HF}_{\mathbb{Z}_{2}}(\Sigma(K))\twoheadrightarrow\theta^{n}\widehat{HF}_{\mathbb{Z}_{2}}(\Sigma(K))/\theta^{n+1}\widehat{HF}_{\mathbb{Z}_{2}}(\Sigma(K))\hookrightarrow\mathbf{gr}_{\theta}\widehat{HF}_{\mathbb{Z}_{2}}(\Sigma(K)).
\]
 For simplicity, we will denote that image by $c_{\mathbb{Z}_{2}}^{n}(\xi_{K})$.

Now suppose that the transverse knot $K$ achieves equality in the
inequality of Corollary \ref{cor:ineq}, i.e. we have 
\[
\mathbf{gr}(c_{\mathbb{Z}_{2}}(\xi_{K}))=d_{3}(\xi_{K})+\frac{1}{2}=\frac{q_{\tau}(K)}{2}.
\]
 Then $c_{\mathbb{Z}_{2}}(\xi_{K})$ lies in the smallest possible
grading among all elements of $\widehat{HF}_{\mathbb{Z}_{2}}(\Sigma(K))$
not annihilated by any powers of $\theta$. Under this setting, $c_{\mathbb{Z}_{2}}(\xi_{K})$
cannot be written as a multiple of $\theta$ by the grading minimality
condition, so we have $c_{\mathbb{Z}_{2}}^{0}(\xi_{K})\neq0$ in $\mathbf{gr}_{\theta}\widehat{HF}_{\mathbb{Z}_{2}}(\Sigma(K))$,
and similarly we have $c_{\mathbb{Z}_{2}}^{n}(\xi_{K})$ for every
nonnegative integer $n$. This simple fact can now be used to prove
the following nonvanishing condition for $c(\xi_{K})$.
\begin{thm}
\label{thm:nonvanishing}Let $K$ be a knot in $S^{3}$, and suppose
that a transverse representative $T$ of $K$ satisfies $d_{3}(\xi_{K})=\frac{q_{\tau}(K)-1}{2}$.
Then the following statements hold.
\begin{enumerate}
\item $c(\xi_{T})\ne a+\tau^{\ast}a$ for every $a\in\widehat{HF}(\Sigma(K))$,
where $\tau$ denotes the deck transformation of the branched covering
map $\Sigma(K)\rightarrow S^{3}$.
\item The cardinality of the set 
\[
\left\{ c(\xi_{T})\left|\,T\text{ is a transverse representative of }K\text{ satisfying }d_{3}(\xi_{K})=\frac{q_{\tau}(K)-1}{2}\right.\right\} 
\]
 is at most half of the cardinality of the $d_{3}(\xi_{T})+\frac{1}{2}$-graded
component of $\widehat{HF}(\Sigma(K),\mathfrak{s}_{0}^{K})$, where
$\mathfrak{s}_{0}^{K}$ is the $\text{Spin}^{c}$-structure on $\Sigma(K)$
induced by the unique spin structure on $\Sigma(K)$.
\end{enumerate}
\end{thm}

\begin{proof}
The statement $c(\xi_{T})\neq a+\tau^{\ast}a$ for every $a\in\widehat{HF}(\Sigma(K))$
means that the element $c(\xi_{K})\otimes\theta^{1}$ in the $E_{1}$
page also survives in the $E_{2}$ page. By the condition we imposed
on $K$, we have $c_{\mathbb{Z}_{2}}^{n}(\xi_{K})\ne0$ for every
nonnegative integer $n$, which means that the element $c(\xi_{K})\otimes\theta^{n}$
in the $E_{1}$ page must survive on every page. This proves (1).

Also, statement (3) follows from the fact that the rank of the $(q_{\tau}(K)+N)$-graded
component of $\widehat{HF}_{\mathbb{Z}_{2}}(\Sigma(K),\mathfrak{s}_{0}^{K})$
converges to $1$ under the limit $N\rightarrow\infty$ (and that
$c(\xi_{K})\otimes\theta^{N}$ must survive in every page). Note that
this a direct corollary of the existence of the localization isomorphism
\[
\widehat{HF}_{\mathbb{Z}_{2}}(\Sigma(K))\otimes_{\mathbb{F}_{2}[\theta]}\mathbb{F}_{2}[\theta,\theta^{-1}]\xrightarrow{\sim}\mathbb{F}_{2}[\theta,\theta^{-1}].
\]
\end{proof}
On the other hand, by analyzing the structure of $\widehat{HF}_{\mathbb{Z}_{2}}(\Sigma(K))$,
we can find a condition which implies the vanishing of $c(\xi_{T})$
for a transverse representative $T$ of $K$. We start by observing
that taking a quotient of $\widehat{CF}_{\mathbb{Z}_{2}}(\Sigma(K))$
by the $\theta$-action gives the ordinary Floer cochain complex $\widehat{CF}^{\ast}(\Sigma(K))$:
\[
\widehat{CF}_{\mathbb{Z}_{2}}(\Sigma(K))\otimes_{\mathbb{F}_{2}[\theta]}\mathbb{F}_{2}\simeq\widehat{CF}^{\ast}(\Sigma(K)).
\]
 We already know from Theorem \ref{thm:forget} that the induced map
\[
\widehat{HF}_{\mathbb{Z}_{2}}(\Sigma(K))\rightarrow\widehat{HF}^{\ast}(\Sigma(K))
\]
 maps $c_{\mathbb{Z}_{2}}(\xi_{K})$ to $c(\xi_{K})$. Now, since
the above map is an $\mathbb{F}_{2}[\theta]$-module homomorhpism
and the $\theta$-action on its codomain $\widehat{HF}^{\ast}(\Sigma(K))$
is trivial, we get a map 
\[
\widehat{HF}_{\mathbb{Z}_{2}}(\Sigma(K))/\theta\widehat{HF}_{\mathbb{Z}_{2}}(\Sigma(K))\rightarrow\widehat{HF}^{\ast}(\Sigma(K)),
\]
 which then maps $c_{\mathbb{Z}_{2}}^{0}(\xi_{K})$ to $c(\xi_{K})$.
Hence, if $c_{\mathbb{Z}_{2}}(\xi_{K})$ is divisible by $\theta$,
then $c(\xi_{K})=0$. This can be used to prove a vanishing property
of $c(\xi_{K})$.
\begin{defn}
Given a knot $K$ in $S^{3}$, we define $v_{\tau}(K)$ to be the
biggest nonnegative integer $N$ such that there exists an element
in the $(q_{\tau}(K)+N)$-graded piece of $\widehat{HF}_{\mathbb{Z}_{2}}(\Sigma(K),\mathfrak{s}_{0}^{K})$
which is not annihilated by any powers of $\theta$ and not divisible
by $\theta$; such an integer always exists by the existence of the
localization isomorphism. Here, $\mathfrak{s}_{0}^{K}$ is defined
as in the statement (2) of Theorem \ref{thm:nonvanishing}.
\end{defn}

\begin{thm}
\label{thm:vnv}Let $K$ be a knot in $S^{3}$ and $T$ be a transverse
representative of $K$. Then $c(\xi)\ne0$ if $d_{3}(\xi_{K})=\frac{q_{\tau}(K)-1}{2}$
and $c(\xi)=0$ if $d_{3}(\xi_{K})>\frac{q_{\tau}(K)-1}{2}+v_{\tau}(K)$.
\end{thm}

\begin{proof}
If $d_{3}(\xi_{K})>v_{\tau}(K)-\frac{1}{2}$, then $c_{\mathbb{Z}_{2}}(\xi_{K})$
is divisible by $\theta$, so $c(\xi_{K})=0$. On the other hand,
if $d_{3}(\xi_{K})=\frac{q_{\tau}(K)-1}{2}$, then $c(\xi_{K})\ne a+\tau^{\ast}a$
for any $a\in\widehat{HF}_{\mathbb{Z}_{2}}(\Sigma(K))$ by Theorem
\ref{thm:nonvanishing}, so $c(\xi_{K})\ne0$ as $0+\tau^{\ast}0=0$.
\end{proof}
Theorem \ref{thm:vnv} can be seen as a Heegaard Floer analogue of
the following theorem of Plamenevskaya.
\begin{thm}[Theorem 1.2, \cite{Plamenevskaya-2008}]
If $K$ is a transverse knot such that $\mathbf{sl}(K)=s(K)-1$,
then $\psi(K)\neq0$, where $s(K)$ stands for the Rasmussen invariant\cite{Rasmussen},
and $\psi(K)$ stands for the Plamenevskaya invariant\cite{plamanevskaya-Kh}.
The converse holds if $K$ is $Kh_{\mathbb{F}_{2}}$-thin.
\end{thm}

As its corollary, Plamenevskaya gives a vanishing/nonvanishing property
of $c(\xi_{K})$ when $K$ is a transverse representative of a quasi-alternating
knot.
\begin{cor}[Corollary 1.3, \cite{Plamenevskaya-2008}]
\label{cor:pl}Let $K$ be a transverse representative of a quasi-alternating
knot. Then $c(\xi_{K})\ne0$ if and only if $\mathbf{sl}(K)=\sigma(K)-1$.
\end{cor}

We will now show that Corollary \ref{cor:pl} is also a direct consequence
of Theorem \ref{thm:vnv}, by giving another proof of it.
\begin{proof}
When $K$ is a quasi-alternating knot and $T$ is a transverse representative
of $K$ which satisfies $\mathbf{sl}(T)=s(K)-1$, then by Proposition
6 of \cite{contact-branch}, we have
\[
d_{3}(\xi_{T})=-\frac{3}{4}\sigma(X)-\frac{1}{2}\mathbf{sl}(T)=-\frac{3}{4}\sigma(X)-\frac{1}{2}\sigma(K)-\frac{1}{2},
\]
 where $X$ is a 4-manifold consisting only of 2-handles, as defined
in section 3.1 of \cite{contact-branch}. Note that the original formula
of Plamenevskaya is wrong by an additive factor of $\frac{1}{2}$;
it is easy to check it by testing it with the transverse unknot, which
has self-linking number $-1$.
\end{proof}
Now, by the construction of $X$, it is a branched double cover of
the 4-ball $B^{4}$ along a smooth surface which bounds $K$, but
with the opposite orientation, so by Theorem 3.1 of \cite{signature},
we have $\sigma(X)=-\sigma(K)$. Hence we get 
\[
d_{3}(\xi_{T})=\frac{3}{4}\sigma(K)-\frac{1}{2}\sigma(K)-\frac{1}{2}=\frac{1}{4}\sigma(K)-\frac{1}{2}.
\]
 On the other hand, since $\Sigma(K)$ is an L-space, we have 
\[
\frac{q_{\tau}(K)}{2}=d(K,\mathfrak{s}_{0}^{K}),
\]
 which is then equal to $\frac{\sigma(K)}{4}$ by Theorem 1 of \cite{Lisca-Owens};
note that we are using the sign convention which makes the right handed
trefoil have signature $2$. Also, $v_{\tau}(K)=0$ by the same reason.
Hence we finally get 
\[
d_{3}(\xi_{T})+\frac{1}{2}=\frac{q_{\tau}(K)}{2}=\frac{q_{\tau}(K)}{2}+v_{\tau}(K).
\]
 Therefore, by Theorem \ref{thm:vnv}, $c(\xi_{T})\ne0$. Similarly,
we can prove that $c(\xi_{T})=0$ if $\mathbf{sl}(T)>s(K)-1$ by replacing
equalities by inequalities.

\section{Conclusion}

Given a based link $(L,p)$ in $S^{3}$, the isomorphism type of $\widehat{HF}_{\mathbb{Z}_{2}}(\Sigma(L),p)$
is well-defined. When $L=K$ is a knot, then $\widehat{HF}_{\mathbb{Z}_{2}}(\Sigma(K),p)$
satisfies naturality, and an isotopy class of a based cobordism between
two based knots induces a uniquely defined map between $\widehat{HF}_{\mathbb{Z}_{2}}$. 

Given a transverse based link $L$ in $(S^{3},\xi_{std})$, we have
constructed a distinguished element 
\[
c_{\mathbb{Z}_{2}}(\xi_{L})\in\widehat{HF}_{\mathbb{Z}_{2}}(\Sigma(L),z),
\]
 which is invariant under transverse isotopies and thus is a transverse
based link invariant. When $L=K$ is a transverse knot, then we can
talk about its image under cobordism maps; it satisfies functoriality
under symplectically constructible based cobordisms. When we work
with a transverse knot $K$, we can forget the choice of a basepoint,
so we get an element 
\[
c_{\mathbb{Z}_{2}}(\xi_{K})\in\widehat{HF}_{\mathbb{Z}_{2}}(\Sigma(K)),
\]
 which may not be functorial under unbased cobordisms. The most natural
question to ask about it would be about the effectivity of $c_{\mathbb{Z}_{2}}$
in distinguishing topologically isotopic transverse knots.

Recall that we have two important classical invariants of transverse
knot invariants:
\begin{itemize}
\item the topological knot type;
\item the self-linking number.
\end{itemize}
Thus a transverse knot invariant is said to be effective if it can
distinguish two transverse knots, which are topologically isotopic
and have the same self-linking number, but are not transversely isotopic,
i.e. isotopic through transverse knots. 

The LOSS invariant, defined in \cite{LOSS-inv}, is an example of
an effective transverse knot invariant, lying in the knot Floer homology
of a given transverse knot:
\[
\hat{c}(K)\in\widehat{HFK}(-Y,K),\,c^{-}(K)\in HFK^{-}(-Y,K).
\]
 The invariant $c^{-}$ has some basic properties which similar to
some properties of equivariant contact classes. First of all, it lies
in the knot Floer homology of $(-Y,K)$, which means that it is a
cohomological invariant in $(Y,K)$. Also, if $K^{-}$ is the transverse(negative)
stabilization of $K$, then we have 
\[
c^{-}(K^{-})=U\cdot c^{-}(K),
\]
 which is very similar to the property $c_{\mathbb{Z}_{2}}(\xi_{K^{-}})=\theta\cdot c_{\mathbb{Z}_{2}}(\xi_{K})$.
So, it is natural to ask the following question.

\subsection*{Question.}

Is there a way to calculate $c^{-}(K)$ or $\hat{c}(K)$using $c_{\mathbb{Z}_{2}}(\xi_{K})$?
If not, then is there a way to calculate $c_{\mathbb{Z}_{2}}(\xi_{K})$
using $c^{-}(K)$?\\

Of course, if we can recover either $c^{-}(K)$ or $\hat{c}(K)$from
$c_{\mathbb{Z}_{2}}(\xi_{K})$, then we immediately see that $c_{\mathbb{Z}_{2}}$
is an effective transverse link invariant. However, even if we cannot,
we can still ask whether the equivariant contact class is an effective
invariant:

\subsection*{Question.}

Is there a knot $K$ in $S^{3}$ such that there exist two transverse
knots $K_{1},K_{2}$ in $(S^{3},\xi_{std})$, topologically isotopic
to $K$, with $\mathbf{sl}(K_{1})=\mathbf{sl}(K_{2})$, but
\[
c_{\mathbb{Z}_{2}}(K_{1})\ne c_{\mathbb{Z}_{2}}(K_{2})
\]
 as elements in $\widehat{HF}_{\mathbb{Z}_{2}}(\Sigma(K))$?\\

If the answer to the above question is yes, then we get a new effective
transverse invariant. However, if the answer is no, we also have an
interesting consequence. Given a knot $K$ in $S^{3}$, let $K_{1},K_{2}$
be two transverse knots, topologically isotopic to $K$. Then, we
apply the following well-known theorem.
\begin{thm*}
Any two topologically isotopic transverse knots are related by a sequence
of (de)stabilizations.
\end{thm*}
Using the above theorem, suppose that the $n$th and $m$th positive
stabilizations of $K_{1}$ and $K_{2}$ are transversely isotopic,
and assume that $n\ge m$. Then the $n-m$th positive stabilization
$K_{1}^{n-m}$ of $K_{1}$ satisfies $\mathbf{sl}(K_{1}^{n-m})=\mathbf{sl}(K_{2})$.
Hence, by the assumed non-effectiveness of $c_{\mathbb{Z}_{2}}$ and
its behavior under positive stabilizations, we get 
\[
c_{\mathbb{Z}_{2}}(\xi_{K_{2}})=c_{\mathbb{Z}_{2}}(\xi_{K_{1}^{n-m}})=c_{\mathbb{Z}_{2}}(\xi_{K_{1}})\cdot\theta^{n-m}.
\]
 Hence, if $T_{K}$ is the transverse representative of $K$ with
the minimal self-linking number, then for any other transverse representative
$T$ of $K$, we must have 
\[
c_{\mathbb{Z}_{2}}(\xi_{T})=c_{\mathbb{Z}_{2}}(\xi_{T_{K}})\cdot\theta^{n_{T}}
\]
 for some $n_{T}\ge0$. Therefore we deduce that the subset 
\[
\{c_{\mathbb{Z}_{2}}(\xi_{T})\,\vert\,T\text{ is a transverse representative of }K\}\subset\widehat{HF}_{\mathbb{Z}_{2}}(\Sigma(K))
\]
 is a single $\theta$-tower; its minimal-order element $c_{\mathbb{Z}_{2}}(\xi_{T_{K}})\in\widehat{HF}_{\mathbb{Z}_{2}}(\Sigma(K))$
becomes a knot invariant.

\subsection*{Question.}

Is there a non-constructible symplectic cobordism between two transverse
knots in the standard contact $S^{3}$?

\end{document}